\documentclass[11pt, reqno]{amsart}

\usepackage[T1]{fontenc}
\usepackage{amsthm, amsfonts, mathrsfs}
\usepackage{amssymb, amsmath, theoremref}
\usepackage{times, txfonts, accents}

\newcommand{\R}{{\mathbf{R}}}

\newcommand{\C}{{\mathbf{C}}}

\newcommand{\lefthook}{\mbox{$\, \rule{8pt}{.5pt}\rule{.5pt}{6pt}\, \, $}}

\newcommand{\SE}{\mathop{\! \, \rm SE }\nolimits}
\newcommand{\SO}{\mathop{\! \, \rm SO }\nolimits}

\newcommand{\Diff}{\mathop{\! \, \rm Diff }\nolimits}
\newcommand{\diff}{\mathop{\! \, \rm diff }\nolimits}

\theoremstyle{plain}
\newtheorem{claim}{\sc Claim}[section]
\newtheorem{corollary}[claim]{\sc Corollary}
\newtheorem{lemma}[claim]{\sc Lemma}
\newtheorem{proposition}[claim]{\sc Proposition}
\newtheorem{theorem}[claim]{\sc Theorem}

\theoremstyle{definition}
\newtheorem{definition}[claim]{\sc Definition}

\newtheorem{example}[claim]{\sc Example}
\theoremstyle{remark}
\newtheorem{remark}[claim]{\sc Remark}

\newtheorem{conclusion}[claim]{\sc Conclusion}
\newtheorem{summary}[claim]{\sc Summary}

\renewcommand{\small}{}

\begin{document}

\title{Elastica as a dynamical system}
\author[L. Bates, R. Chhabra and J. \'Sniatycki]{Larry Bates, Robin Chhabra and 
J\k{e}drzej \'{S}niatycki}
\date{}

\begin{abstract}
The elastica is a curve in $\R^3$ that is stationary under variations
of the integral of the square of the curvature. Elastica are viewed as a 
dynamical system that arises from the second order calculus of variations, 
and its quantization is discussed.
\end{abstract}

\maketitle

\section{Introduction}

\noindent Ever since the beginning of the calculus of variations, 
second order problems such as the classical problem of the elastica have been 
considered.  The peculiar situation that distinguishes most of the interesting 
examples in second order problems from the more familiar first order theory is 
that they are parameter independent, and so the theory of such problems has a 
somewhat distinctive tone from that of the more familiar first order theory.  A 
comprehensive review of this theory, as it was understood up until the 1960s, 
may be found in the monograph by Gr\"asser \cite{grasser}.\footnote{This 
monograph is especially noteworthy for its comprehensive bibliography.}
By way of contrast, this paper emphasizes the relation of the variational 
problem to the geometry of the corresponding Cartan form and its interplay with 
symmetry, conservation laws, the Noether theorems and the associated first 
order canonical formalism. The problem of elastica is then re-examined in light 
of this discussion.

The rationale for this paper is to give a biased view of that portion 
of the theory of second order variational problems that could reasonably be 
expected to be useful for understanding the common behaviour of several 
geometric functionals on curves.  Good examples to keep in mind while 
reading this paper (these are the three main examples that motivated our 
study) are the elastica, the shape of a real 
M\"obius band in terms of the geometry of the central 
geodesic \cite{wunderlich}, and the curve of least friction.    

For reasons not entirely clear to us, the geometric theory of higher order 
variational problems seems to have developed in a manner largely detached from 
the needs and concerns of concrete problems.  This is a startling contrast to 
recent developments in geometric mechanics and their understanding of 
stability, bifurcation, numerical schemes, the incorporation of nonholonomic 
constraints, etc. The consequences of this are at least two-fold: first, it 
leads to a palpable sense of dread\footnote{This may be reduced to mere 
frustration by those less ignorant than the authors.} when faced with trying to 
look up a formulation of some part of the theory that will cleanly explain how 
to compute something obvious, and second, a real disconnect between the 
theoretical insights and the actual computational methods.  This disconnect is  
vividly illustrated in the problem of elastica.  

Planar elastica (the equilibrium shape of a linearly elastic thin wire) were 
considered at least as early as 1694 by James Bernoulli.\footnote{See the 
delightful 
discussion by Levien in \cite{levien} or 
\cite{levien:EECS-2008-103}.}  
However, it was not until about 1742 that Daniel Bernoulli convinced Euler to 
solve the problem by using the isoperimetric method (the old name for the 
calculus of variations before Lagrange.) From a variational point of view, the 
elastica is idealized as a curve that minimizes 
the integral over its length of the square of the curvature (that is, minimize 
$\int\kappa^2\,ds$), and is thus 
naturally treated as a second order problem in the calculus of 
variations.  Exhaustive results were then published 
by Euler in 1744 \cite{euler1744}.  Since Euler's results were so 
comprehensive, it is not surprising that the study of elastica remained 
somewhat 
dormant until taken up again by Max Born in his thesis \cite{born06}. More 
recently a striking  result was obtained in 1984 by Langer and Singer 
\cite{langer-singer} when they demonstrated the existence of closed elastica 
that were torus knots.  Their proof was noteworthy because they eschewed the 
usual variational machinery and employed clever {\it ad hoc} geometric 
arguments 
such as an adapted cylindrical coordinate system to aid their integration.  In 
fact, a significant motivation for this paper was to see to what extent their 
results could be understood by a more pedestrian use of the second order 
calculus 
of variations that looked more like just `turning the crank' on the variational 
machine, and thus had the comfort of familiarity of technique.  

Some features of the elastica problem instantly spring to 
mind in the modern geometrically oriented 
reader.  The first is that the problem is manifestly invariant by the action of 
the Euclidean group.  The second is that it would be very nice to have a theory 
that explained how to reduce the symmetry using the concomitant 
conservation laws that Emmy Noether taught us are in the problem, and then wind 
up with some form of reduced Euler-Lagrange equations.  Assuming we can solve 
these reduced equations, and hence know the curvature and torsion of the 
elastic 
curve, we would expect a good theory to show us methods to determine the 
shape of our curve that go beyond a trite referral to the fundamental theorem 
of 
curves stating that the curvature and torsion of the curve determine it up 
to a Euclidean motion.  Given all this, what we actually find when we look at 
the published work on elastica (such as 
\cite{langer-singer} or \cite{bryant-griffiths}) is that it proceeds 
somewhat differently.  
In particular, almost none of the actual computations seem to follow any 
 method that resembled the current theory.  There are good reasons 
for this, and it is not due 
to ignorance of those geometers but a reflection that the theory 
at that time was presented in such a way as to simply be unhelpful, and unable 
to easily identify the geometric meaning of some of their calculations.  This 
is 
the best explanation we have of the situation at the time and why it was still 
necessary a decade after \cite{langer-singer} appeared for Foltinek (see 
\cite{foltinek}) to write a paper 
demonstrating the integration constants in elastica in terms of the conserved 
Noetherian momenta.      

The plan of this paper is to first discuss the Euler-Lagrange equations.  This 
section will serve to fix notation and some basic notions.  Next, we discuss 
the Cartan form as the geometrization of the variational problem using the 
structure of the jet bundle. This is followed by discussion of parametrization 
invariance and the Hamiltonian formalism.  Elastica are then studied from 
this point of view.  The problem is then recast as a constrained Hamiltonian 
system and the reparametrization group and its reduction are examined.  The 
paper concludes with the geometric quantization, and discusses the quantum 
representations of the groups $\SE(3)$ and $\Diff_+\R$ as well as the quantum 
implementation of constraints.

\section{Second order variational problems}

\subsection{Variation of the action integral}

A curve $[t_{0},t_{1}]\rightarrow \R^{n}:t\mapsto x(t)$ can be
uniquely described by the corresponding section
\begin{equation}
\sigma :[t_{0},t_{1}]\rightarrow \lbrack t_{0},t_{1}]\times \R^{n}:t\mapsto 
(t,x(t)).  \label{sigma}
\end{equation}
We consider $Q=\R\times \R^n$ as a bundle over $\R$ with
typical fibre $\R^{n}$. For each integer $k\geq 0$, we denote by $ 
J^{k}$ the $k$-th jet of sections of $Q$. Furthermore, interpret the section $ 
\sigma $ given in (\ref{sigma}) as a local section of $Q$ and denote by $ 
j^{k}\sigma :[t_{0},t_{1}]\rightarrow J^{k}$ the $k$-jet extension of $ 
\sigma $. This paper considers variational problems defined by
second order Lagrangians. For a Lagrangian $L:J^{2}\rightarrow \R 
:(t,x,\dot{x},\ddot{x})\mapsto L(t,x,\dot{x},\ddot{x})$, the corresponding
action integral is 
\begin{equation*}
A(\sigma )=\int_{t_{0}}^{t_{1}}(L\circ j^{2}\sigma
)\,dt=\int_{t_{0}}^{t_{1}}L(t,x(t),\dot{x}(t),\ddot{x}(t))\,dt.
\end{equation*} 
A variation of a section (without variation of time) $\sigma $ $\mapsto
\sigma +\delta \sigma :t\mapsto x(t)+\delta x(t)$ extends to the second jets as 
\begin{equation*}
j^{2}\sigma \mapsto j^{2}\sigma +\delta j^{2}\sigma :t\mapsto (x(t)+\delta
x(t),\dot{x}(t)+\delta \dot{x}(t),\ddot{x}(t)+\delta \ddot{x}(t)),
\end{equation*} 
where 
\begin{equation*}
\delta \dot{x}(t)=\frac{d}{dt}\delta x(t)\text{ \ and \ }\delta \ddot{x}(t)= 
\frac{d}{dt}\delta \dot{x}(t)=\frac{d^{2}}{dt^{2}}\delta x(t).
\end{equation*} 
Then, integrating by parts twice, it follows that the action varies as 
\begin{eqnarray*}
\delta A(\sigma ) &=&\int_{t_{0}}^{t_{1}}\delta L(t,x(t),\dot{x}(t),\ddot{x} 
(t))\,dt \\
&=&\int_{t_{0}}^{t_{1}}\left\{ \frac{\partial L}{\partial x}(t)-\frac{d}{dt} 
\frac{\partial L}{\partial \dot{x}}(t)+\frac{d^{2}}{dt^{2}}\frac{\partial L}{ 
\partial \ddot{x}}(t)\right\} \,\delta x \, dt+ \\
&&+\left. \left( \frac{\partial L}{\partial \dot{x}}(t)-\frac{d}{dt}\frac{ 
\partial L}{\partial \ddot{x}}(t)\right) \,\delta x\,\right\vert
_{t_{0}}^{t_{1}}+\left. \frac{\partial L}{\partial \ddot{x}}(t)\, \delta 
\dot{x} \,
\right\vert _{t_{0}}^{t_{1}}.
\end{eqnarray*} 
It follows from the fundamental lemma of the calculus of variations that

\begin{conclusion}
\label{2.1}The action integral 
\begin{equation*}
A(\sigma )=\int_{t_{0}}^{t_{1}}(L\circ j^{2}\sigma
)\,dt=\int_{t_{0}}^{t_{1}}L(t,x(t),\dot{x}(t),\ddot{x}(t))\,dt
\end{equation*} 
is stationary with respect to all variations $\sigma $ $\mapsto \sigma
+\delta \sigma :t\mapsto x(t)+\delta x(t)$, such that $\delta x$ and $\delta 
\dot{x}$ vanish on the boundary, if and only if the section $\sigma$ satisfies
the Euler-Lagrange equations 
\begin{equation}
\frac{\partial L}{\partial x}-\frac{d}{dt}\frac{\partial L}{\partial \dot{x}} 
+\frac{d^{2}}{dt^{2}}\frac{\partial L}{\partial \ddot{x}}=0.
\label{Euler Lagrange}
\end{equation}
\end{conclusion}

\bigskip

Consider now the boundary terms in the variation. The partial derivative $ 
\frac{\partial L}{\partial \ddot{x}}$ is a map from $J^{2}$ to $\R 
^{n}$, and 
\begin{equation*}
\frac{\partial L}{\partial \ddot{x}}(t)\,\delta \dot{x}=\left\langle \frac{ 
\partial L}{\partial \ddot{x}}(t,x(t),\dot{x}(t),\ddot{x}(t)), \, \delta 
\dot{x}(t)\right\rangle ,
\end{equation*} 
where the angle bracket denotes the Euclidean scalar product in $\R 
^{n}$. However, 
\begin{eqnarray*}
\frac{d}{dt}\left( \frac{\partial L}{\partial \ddot{x}}(t)\right) =\left( 
\frac{\partial }{\partial t}\frac{\partial L}{\partial \ddot{x}}
+ \dot{x}\frac{\partial }{ 
\partial x} \frac{\partial L}{\partial \ddot{x}}  
+ \ddot{x}\frac{\partial }{\partial \dot{x}} \frac{ 
\partial L}{\partial \ddot{x}} + \dddot{x}\frac{\partial }{\partial 
\ddot{x}} \frac{\partial L}{ 
\partial \ddot{x}}\right) 
\end{eqnarray*} 
depends on the third derivative $\dddot{x}$ of the section $\sigma $, 
and hence, $\frac{d}{dt}\frac{\partial L}{\partial \ddot{x}}$ can be interpreted
as a map from $J^{3}$ to $\R^{n}$. Using the projection map 
\begin{equation*}
\pi _{32}:J^{3}\rightarrow J^{2}:(t,x,\dot{x},\ddot{x},\dddot{x})\mapsto
(t,x,\dot{x},\ddot{x}),
\end{equation*} 
 define Ostrogradski's momenta by 
\begin{eqnarray*}
p_{\dot{x}} &=&\pi _{32}^{\ast }\frac{\partial L}{\partial \ddot{x}}, \\
p_{x} &=&\pi _{32}^{\ast }\frac{\partial L}{\partial \dot{x}}-\frac{d}{dt} 
\frac{\partial L}{\partial \ddot{x}},
\end{eqnarray*} 
and interpret them as maps from $J^{3}$ to $\R^{n}$. In the following, in order 
to
simplify the notation, the pull-back sign is omitted  and an overdot is used to
denote the derivative with respect to $t$. This leads to the usual
expressions 
\begin{eqnarray}\label{usual-momenta}
p_{\dot{x}} &=&\frac{\partial L}{\partial \ddot{x}},
\label{Ostrogradski's momenta} \\
p_{x} &=&\frac{\partial L}{\partial \dot{x}}-\frac{d}{dt}\frac{\partial L}{ 
\partial \ddot{x}}=\frac{\partial L}{\partial \dot{x}}-\dot{p}_{\dot{x}}.  \notag
\end{eqnarray} 
With this notation, the variation equation is 
\begin{equation}
\delta A(\sigma )=\int_{t_{0}}^{t_{1}}\left\{ \frac{\partial L}{\partial x} 
(t)-\frac{d}{dt}\frac{\partial L}{\partial \dot{x}}(t)+\frac{d^{2}}{dt^{2}} 
\frac{\partial L}{\partial \ddot{x}}(t)\right\} \,\delta x\,dt+\left. p_x\,\delta
x\right\vert _{t_{0}}^{t_{1}}+\left. p_{\dot{x}}\,\delta \dot{x}\right\vert
_{t_{0}}^{t_{1}}.  \label{Variation 1}
\end{equation}

With an eye towards towards the Cartan form, it is convenient to reinterpret
a variation as the Lie derivative with respect to a vector field. 
Let a variation $\sigma $ $\mapsto \sigma +\delta \sigma
:t\mapsto x(t)+\delta x(t)$ of $\sigma$ be given by a vector field $X$ on 
$J^{2}$
that is tangent to the fibres of the source map $J^{2}\rightarrow \lbrack
t_{0},t_{1}]:(t,x,\dot{x},\ddot{x})\mapsto t$. In other words, if $X=X_{x} 
\frac{\partial }{\partial x}$, then 
\begin{equation*}
\delta x(t)=X_{x}(\sigma(t)).
\end{equation*} 
Then the variation $j^{2}\sigma \mapsto j^{2}\sigma +\delta j^{2}\sigma $ is
given by the prolongation 
\begin{equation*}
X^{2}=X_{x}\frac{\partial}{\partial x}+X_{\dot{x}}\frac{\partial }{\partial 
\dot{x}}+X_{\ddot{x}}\frac{\partial }{\partial \ddot{x}}
\end{equation*} 
of $X$ to $J^{2}$, where  
\begin{equation*} 
X_{\dot{x}}=\frac{d }{d t}X_{x}\text{ and }X_{\ddot{x}}=\frac{ 
d }{dt}X_{\dot{x}}=\frac{d^{2}}{dt^{2}}X_{x}.
\end{equation*} 
In other words, 
\begin{equation*}
\delta \dot{x}(t)=X_{\dot{x}}(j^{2}\sigma (t))\text{ \ and \ }\delta \ddot{x} 
(t)=X_{\ddot{x}}(j^{2}(\sigma )).
\end{equation*} 
With this identification, 
\begin{eqnarray}
\delta A(\sigma ) &=&\int_{t_{0}}^{t_{1}}\left\{ \frac{\partial L}{\partial x 
}(t)\,\delta x+\frac{\partial L}{\partial \dot{x}}(t)\, \delta \dot{x}+\frac{ 
\partial L}{\partial \ddot{x}}(t)\,\delta \ddot{x}\right\} dt
\label{Variation0} \\
&=&\int_{t_{0}}^{t_{1}}\pounds _{X^{2}}(L\,dt)=\int_{t_{0}}^{t_{1}}X^{2} 
 \lefthook
d(L\,dt)  \notag
\end{eqnarray} 
because 
\begin{equation*}
\pounds _{X^{2}}(L\,dt)=X^{2} 
\lefthook
d(L\,dt)+d(X^{2} 
\lefthook
L\,dt),
\end{equation*} 
and the the assumption that $X$ is tangent that to the fibres of the source
map implies that $X^{2} 
\lefthook
L\,dt=0.$ Comparing equations (\ref{Variation 1}) and (\ref{Variation0})
yields 
\begin{eqnarray}
\int_{t_{0}}^{t_{1}}X^{2} 
\lefthook
d(L\,dt) &=&\int_{t_{0}}^{t_{1}}\left\{ \frac{\partial L}{\partial x}-\frac{d}{ 
dt}\frac{\partial L}{\partial \dot{x}}+\frac{d^{2}}{dt^{2}}\frac{\partial L}{ 
\partial \ddot{x}}\right\} X_{x}\,dt+\left. p_xX_{x}\right\vert
_{t_{0}}^{t_{1}}+\left. p_{\dot{x}}X_{\dot{x}}\right\vert _{t_{0}}^{t_{1}}
\label{Variation 2} \\
&=&\int_{t_{0}}^{t_{1}}\left\{ \frac{\partial L}{\partial x}-\frac{d}{dt} 
\frac{\partial L}{\partial \dot{x}}+\frac{d^{2}}{dt^{2}}\frac{\partial L}{ 
\partial \ddot{x}}\right\} X_{x}\,dt+\left. \left\langle 
p_x\,dx+p_{\dot{x}}\,d\dot{x} 
,X^{1}\right\rangle \right\vert _{t_{0}}^{t_{1}},  \notag
\end{eqnarray} 
where $\langle p_x\,dx+p_{\dot{x}}d\dot{x},X^{1}\rangle $ is the 
evaluation of
a 1-form $p_x\,dx+p_{\dot{x}}\,d\dot{x}$ on the first jet bundle on the first jet
prolongation $X^{1}$ of $X$ (see proposition (\ref{4.1})).  In 
equation (\ref{Variation 2}) $p_x\,dx$ and 
$p_{\dot{x}}\,d\dot{x}$ are interpreted as one-forms on $J^{1}$.

\subsection{The Cartan form}

The contact forms of the second jet bundle $J^{2}$ are 
\begin{equation*}
\theta _{1}=dx-\dot{x}\,dt\text{ \ and \ }\theta _{2}=d\dot{x}-\ddot{x}\,dt.
\end{equation*} 
Their importance stems from the following

\begin{proposition}
\label{3.1}A section $\sigma :[t_{0},t_{1}]\rightarrow J^{2\text{ } 
}:t\mapsto (t,x(t),\dot{x}(t),\ddot{x}(t))$ is the jet extension of the section
of its projection $[t_{0},t_{1}]\rightarrow \R^{n}:t\mapsto (t,x(t))$
by the source map $J^{2}\rightarrow \lbrack t_{0},t_{1}]:(t,x,\dot{x},\ddot{x 
})\mapsto (t,x)$ if and only if $\sigma^{\ast }\theta _1=0$ and $\sigma
^{\ast }\theta _2=0.$
\end{proposition}

\begin{definition}
\label{3.2}The Cartan form corresponding to a Lagrangian $L$ is the one-form $ 
\Theta $ on $J^{3}$ given by 
\begin{equation}
\Theta =L\,dt+p_{x}(dx-\dot{x}\,dt)+p_{\dot{x}}(d\dot{x}-\ddot{x}\,dt),
\label{Theta}
\end{equation} 
where $p_x$ and $p_{\dot{x}}$ are the Ostrogradski momenta 
(\ref{usual-momenta}).
\end{definition}

Observe that $\Theta $ may by written in the form 
\begin{equation}
\Theta =p_x\,dx+p_ {\dot{x}}\,d\dot{x}-H\,dt,  \label{Theta1}
\end{equation} 
where 
\begin{equation}
H=p_x\dot{x}+p_{\dot{x}}\ddot{x}-L  \label{Hamiltonian}
\end{equation} 
is the Hamiltonian of the theory. \ Since $\Theta $ differs from the
Lagrange form $L\,dt$ by terms that are proportional to the contact forms,
it follows that for any section $\sigma $ the action $A(\sigma)$ can be 
expressed
as the integral of $\Theta $ over $j^{3}\sigma .$ In other words, 
\begin{equation}
A(\sigma )=\int_{t_{0}}^{t_{1}}(L\circ j^{2}\sigma
)\,dt=\int_{t_{0}}^{t_{1}}(j^{2}\sigma )^{\ast
}L\,dt=\int_{t_{0}}^{t_{1}}(j^{3}\sigma )^{\ast }\Theta .  \label{Action1}
\end{equation} 
Therefore, the Cartan form $\Theta $ may be used instead of the Lagrange form 
$L\,dt$ to describe the variational problem under consideration. Other aspects
of the Cartan form are discussed in  \cite{krupka-krupkova-saunders}.

\begin{definition}
\label{3.3}A Langrangian $L$ is {\it regular} if the matrix 
\begin{equation*}
\frac{\partial ^{2}L}{\partial \ddot{x}_{j}\partial \ddot{x}_{i}}
\end{equation*} 
is non-singular.
\end{definition}

\begin{theorem}
\label{3.4}Let $\gamma $ be a section of the source map $J^{3}\rightarrow
\lbrack t_{0},t_{1}]$ projecting to a section $\sigma $ of $ 
[t_{0},t_{1}]\times \R^{n}\rightarrow \lbrack t_{0},t_{1}]$ and let $ 
j^{3}\sigma $ be the the third jet extension of $\sigma$.\newline
\begin{enumerate}
  \item  If $\gamma =j^{3}\sigma $, then $\sigma $ satisfies the 
Euler-Lagrange
equations if and only if the tangent bundle of the range of $\gamma $ is
contained in the kernel of $d\Theta .$\newline
  \item  If the Lagrangian $L$ is regular and the tangent bundle of the 
range of 
$\gamma $ is contained in the kernel of $d\Theta $, then $\gamma
=j^{3}\sigma $ and $\sigma $ satisfies the the Euler-Lagrange equations.
  \end{enumerate}
  \end{theorem}

\begin{proof}
The exterior differential of $\Theta $ can be written as 
\begin{equation*}
d\Theta =\dot{p}_{\dot{x}}\,d\dot{x}\wedge dt+\frac{\partial L}{\partial 
x}\,dx\wedge
dt+dp_x\wedge (dx-\dot{x}\,dt)+dp_{\dot{x}}\wedge (d\dot{x}-\ddot{x}\,dt),
\end{equation*} 
because 
\begin{eqnarray*}
d\Theta &=&dL\wedge dt-p_x\,d\dot{x}\wedge dt-p_{\dot{x}}\,d\ddot{x}\wedge dt+dp_x\wedge (dx- 
\dot{x}\,dt)+dp_{\dot{x}}\wedge (d\dot{x}-\ddot{x}\,dt) \\
&=&(L_{\ddot{x}}-p_{\dot{x}})\,d\ddot{x}\wedge dt+(L_{\dot{x}}-p_x)\,d\dot{x}\wedge
dt+L_{x}\,dx\wedge dt+ \\
&&+dp_x\wedge (dx-\dot{x}\,dt)+dp_{\dot{x}}\wedge (d\dot{x}-\ddot{x}\,dt) \\
&=&\dot{p}_{\dot{x}}\,d\dot{x}\wedge dt+L_{x}\,dx\wedge dt+dp_x\wedge (dx-\dot{x}\, 
dt)+dp_{\dot{x}}\wedge (d\dot{x}-\ddot{x}\,dt).
\end{eqnarray*} 
A section $\gamma :t\mapsto $ $\gamma (t)=(t,x(t),\dot{x}(t),\ddot{x}(t))$
of the source map projects to a section $\sigma :t\mapsto \sigma
(t)=(t,x(t)) $. Moreover, 
\begin{equation*}
T\gamma (\partial_t)=\frac{{\small  
\partial }}{{\small \partial t}}+\frac{{\small dx(t)}}{{\small dt}}\frac{ 
{\small \partial }}{{\small \partial x}}+\frac{{\small d\dot{x}(t)}}{{\small  
dt}}\frac{{\small \partial }}{{\small \partial \dot{x}}}+\frac{{\small d 
\ddot{x}(t)}}{{\small dt}}\frac{{\small \partial }}{{\small \partial \ddot{x} 
}}
\end{equation*} 
spans the tangent bundle of the range of $\gamma $, and 
\begin{equation*}
T\gamma (\partial_ t)(p_x)=\dot{p}_x~~\text{ 
and}~~T\gamma (\partial_t)(p_{\dot{x}})=\dot{p}_{\dot{x}}.
\end{equation*} 
Hence, 
\begin{eqnarray*}
&&T\gamma (\partial_t) 
\lefthook
d\Theta = \\
&=&\dot{p}_{\dot{x}}\frac{{\small d\dot{x}(t)}}{{\small 
dt}}\,dt-\dot{p}_{\dot{x}}\,d\dot{x}+\frac{ 
{\small \partial L}}{\partial x}\frac{{\small dx(t)}}{{\small dt}}dt-\frac{ 
{\small \partial L}}{{\small \partial 
x}}\,dx+p_{\dot{x}}(dx-\dot{x}\,dt)+\dot{p}_{\dot{x}}(d\dot{x}- 
\ddot{x}\,dt)+ \\
&&-\left\langle dx-\dot{x}\,dt,T\gamma (\partial_t)\right\rangle dp_x-\left\langle d\dot{x}-\ddot{x}\,dt,T\gamma ( \partial_t)\right\rangle dp_{\dot{x}} \\
&=&\dot{p}_{\dot{x}}\left( \frac{{\small d\dot{x}(t)}}{{\small dt}}-\ddot{x}\right) dt+ 
\frac{{\small \partial L}}{{\small \partial x}}\left( \frac{{\small dx(t)}}{ 
{\small dt}}-\dot{x}\right) dt+\left( \dot{p}_x-\frac{{\small \partial L}}{ 
{\small \partial x}}\right) (dx-\dot{x}\,dt)+ \\
&&-\left\langle dx-\dot{x}\,dt,T\gamma (\partial_t)\right\rangle dp_x-\left\langle d\dot{x}-\ddot{x}\,dt,T\gamma (\partial_t)\right\rangle dp_{\dot{x}}.
\end{eqnarray*}

To prove the first statement observe that if $\gamma =j^{3}\sigma $, \ then 
\begin{eqnarray*}
\frac{{\small d\dot{x}(t)}}{{\small dt}}-\ddot{x} &=&0, \\
\frac{{\small dx(t)}}{{\small dt}}-\dot{x}\,dt &=&0, \\
\left\langle dx-\dot{x}dt,T\gamma (\partial_t)\right\rangle &=&0, \\
\left\langle d\dot{x}-\ddot{x}dt,T\gamma (\partial_t)\right\rangle &=&0,
\end{eqnarray*} 
and 
\begin{equation*}
T\gamma (\partial_t) 
\lefthook
d\Theta =\left( \dot{p}_x-\frac{{\small \partial L}}{{\small \partial x}} 
\right) (dx-\dot{x}\,dt).
\end{equation*} 
Hence, $T\gamma (\partial_t)$ is in the
kernel of $d\Theta $ if and only if $\dot{p}_x-\frac{{\small \partial L}}{ 
{\small \partial x}}=0$. By definition of the Ostrogradski momenta, 
\begin{equation*}
\dot{p}_x-\frac{{\small \partial L}}{{\small \partial x}}=\frac{d}{dt}\left( 
\frac{{\small \partial L}}{{\small \partial \dot{x}}}-\frac{d}{dt}\frac{ 
{\small \partial L}}{{\small \partial \ddot{x}}}\right) -\frac{{\small  
\partial L}}{{\small \partial x}},
\end{equation*} 
and so $T\gamma (\partial_t)$ is in
the kernel of $d\Theta $ if and only if the section $\sigma $ satisfies the
Euler-Lagrange equations.

To prove the second part, suppose that $T\gamma (\partial_t)$ is in the kernel of $d\Theta $. Then 
\begin{eqnarray}
&&\dot{p}_{\dot{x}}\left( \frac{{\small d\dot{x}(t)}}{{\small dt}}-\ddot{x}\right) dt+ 
\frac{{\small \partial L}}{{\small \partial x}}\left( \frac{{\small dx(t)}}{ 
{\small dt}}-\dot{x}\right) dt+\left( \dot{p}_x-\frac{{\small \partial L}}{ 
{\small \partial x}}\right) (dx-\dot{x}\,dt)+  \label{Proof1} \\
&&-\left\langle dx-\dot{x}\,dt,T\gamma (\partial_t)\right\rangle dp_x-\left\langle d\dot{x}-\ddot{x}\,dt,T\gamma ( \partial_t)\right\rangle dp_{\dot{x}}=0.  \notag
\end{eqnarray} 
Evaluating the left hand side on a vector $v\frac{\partial }{\partial \dddot{ 
x}}$, where $v$ is an arbitrary vector in $\R^{n}$, yields
\begin{equation*}
\left\langle dx-\dot{x}\,dt,T\gamma (\partial_t)\right\rangle v\frac{\partial p_x}{\partial \dddot{x}}=0,
\end{equation*} 
because only $p_x$ depends on $\dddot{x}$. Since
\begin{equation*}
p_x=\frac{{\small \partial L}}{{\small \partial \dot{x}}}-\frac{d}{dt}\frac{ 
{\small \partial L}}{{\small \partial \ddot{x}}}=\frac{{\small \partial L}}{ 
{\small \partial \dot{x}}}-\frac{{\small \partial }^{2}{\small L}}{{\small  
\partial t\partial \ddot{x}}}-\dot{x}\frac{{\small \partial }^{2}{\small L}}{ 
{\small \partial x\partial \ddot{x}}}-\ddot{x}\frac{{\small \partial }^{2} 
{\small L}}{{\small \partial \dot{x}\partial \ddot{x}}}-\dddot{x}\frac{ 
{\small \partial }^{2}{\small L}}{{\small \partial \ddot{x}\partial \ddot{x}} 
},
\end{equation*} 
it follows that
\begin{equation*}
v\frac{\partial p_x}{\partial \dddot{x}}=-v\frac{{\small \partial }^{2}{\small  
L}}{{\small \partial \ddot{x}\partial \ddot{x}}}.
\end{equation*} 
Therefore, 
\begin{equation*}
\left\langle dx-\dot{x}\,dt,T\sigma (\partial_t)\right\rangle v\frac{{\small \partial }^{2}{\small L}}{{\small  
\partial \ddot{x}\partial \ddot{x}}}=0
\end{equation*} 
for an arbitrary vector $v$. By assumption of regularity of the Lagrangian,
the matrix $\left( \frac{{\small \partial }^{2}{\small L}}{{\small \partial 
\ddot{x}\partial \ddot{x}}}\right) $ is non-singular. Hence, 
\begin{equation*}
\left\langle dx-\dot{x}\,dt,T\sigma (\partial_t)\right\rangle =0,
\end{equation*} 
which implies that $\frac{{\small dx(t)}}{{\small dt}}-\dot{x}(t)=0$.
Substituting these results into equation (\ref{Proof1}) yields 
\begin{equation}
\dot{p}_{\dot{x}}\left( \frac{{\small d\dot{x}(t)}}{{\small dt}}-\ddot{x}\right)
dt+\left( \dot{p}-\frac{{\small \partial L}}{{\small \partial x}}\right) (dx- 
\dot{x}\,dt)-\left\langle d\dot{x}-\ddot{x}\,dt,T\gamma (\partial_t)\right\rangle dp_{\dot{x}}=0.  \label{Proof2}
\end{equation} 
Evaluating the left hand side of this equation on a vector $v\frac{\partial 
}{\partial \ddot{x}}$, where $v$ is an arbitrary vector in $\R^{n}$,
yields 
\begin{equation*}
\left\langle d\dot{x}-\ddot{x}\,dt,T\gamma (\partial_t)\right\rangle v\frac{\partial p_{\dot{x}}}{\partial \ddot{x}}=0.
\end{equation*} 
Since
\begin{equation*}
\frac{\partial p_{\dot{x}}}{\partial \ddot{x}}=\frac{\partial ^{2}L}{\partial \ddot{x} 
\partial \ddot{x}},
\end{equation*} 
\begin{equation*}
\left\langle d\dot{x}-\ddot{x}\,dt,T\gamma (\partial_t)\right\rangle v\frac{\partial ^{2}L}{\partial \ddot{x}\partial 
\ddot{x}}=0
\end{equation*} 
for every vector $v$ in $\R^{n}.$ The assumed regularity of $L$
implies that 
\begin{equation*}
\left\langle d\dot{x}-\ddot{x}\,dt,T\gamma (\partial_t)\right\rangle =0,
\end{equation*} 
so that $\frac{{\small d\dot{x}(t)}}{{\small dt}}-\ddot{x}=0.$ Hence, $ 
\gamma =j^{3}\sigma $ and equation (\ref{Proof2}) reads 
\begin{equation*}
\left( \dot{p}_x-\frac{{\small \partial L}}{{\small \partial x}}\right) (dx- 
\dot{x}\,dt)=0,
\end{equation*} 
which implies that $\sigma$ satisfies the Euler-Lagrange equations.
\end{proof}

\subsection{Symmetries and conservation laws}

\subsubsection{Symmetries of the Lagrange form}

Consider an infinitesimal transformation in $(t_{0},t_{1})\times \R 
^{n}$ given by 
\begin{equation*}
\bar{t}=t+\epsilon \tau (t,x)\text{, \ \ }\bar{x}^{i}=x^{i}+\epsilon \xi 
^{i}(t,x).
\end{equation*} 
It corresponds to a local one-parameter group of local diffeomorphisms
generated by the vector field 
\begin{equation}
X=\tau \frac{\partial }{\partial t}+\xi ^{i}\frac{\partial }{\partial x^{i}}.
\label{X0}
\end{equation}

\begin{proposition}
\label{4.1}The prolongations $X^{1}$, $X^{2}$ and $X^{2}$ of the vector
field $X$ in equation (\ref{X0}) to the jet bundles $J^{1},$ $J^{2}$ and $ 
J^{3}$, respectively, are 
\begin{eqnarray*}
X^{1} &=&\tau \frac{\partial }{\partial t}+\xi ^{i}\frac{\partial }{\partial
x^{i}}+(\dot{\xi}-\dot{x}\dot{\tau})\frac{\partial }{\partial \dot{x}}, \\
X^{2} &=&\tau \frac{\partial }{\partial t}+\xi \frac{\partial }{\partial x}+( 
\dot{\xi}-\dot{x}\dot{\tau})\frac{\partial }{\partial \dot{x}}+(\ddot{\xi}-2 
\ddot{x}\dot{\tau}-\dot{x}\ddot{\tau})\frac{\partial }{\partial \ddot{x}}, \\
X^{3} &=&\tau \frac{\partial }{\partial t}+\xi \frac{\partial }{\partial x}+( 
\dot{\xi}-\dot{x}\dot{\tau})\frac{\partial }{\partial \dot{x}}+(\ddot{\xi}-2 
\ddot{x}\dot{\tau}-\dot{x}\ddot{\tau})\frac{\partial }{\partial \ddot{x}}+( 
\dddot{\xi}-3\dddot{x}\dot{\tau}-3\ddot{x}\ddot{\tau}-\dot{x}\dddot{\tau}) 
\frac{\partial }{\partial \dddot{x}}.
\end{eqnarray*}
\end{proposition}

It remains to relate the prolongations of $X$ to the contact forms $\theta
_{1}=dx-\dot{x}\,dt$ and $\theta _{2}=d\dot{x}-\ddot{x}\,dt.$

\begin{proposition}
\label{4.2} Let $I=[t_0,t_1]$. For a section $\sigma $ of $I\times\R^{n} 
\rightarrow
I$, 
\begin{equation*}
j^{1}\sigma ^{\ast }\pounds _{X^{1}}\theta _{1}=0~~~~and~~~~j^{2}\sigma
^{\ast }\pounds _{X^{2}}\theta _{2}=0.
\end{equation*}
\end{proposition}

\begin{proof}
Since 
\begin{equation*}
X^{1}=\tau \frac{\partial }{\partial t}+\xi \frac{\partial }{\partial x}+( 
\dot{\xi}-\dot{x}\dot{\tau})\frac{\partial }{\partial \dot{x}},
\end{equation*} 
then 
\begin{eqnarray*}
\pounds _{X^{1}}\theta _{1} &=&X^{1} 
\lefthook
d\theta _{1}+d\langle \theta _{1},X^{1}\rangle \\
&=&\left( \frac{\partial \xi }{\partial x}-\dot{x}\frac{\partial \tau }{ 
\partial x}\right) \left( dx-\dot{x}\,dt\right) .
\end{eqnarray*} 
Therefore 
\begin{eqnarray*}
j^{1}\sigma ^{\ast }\pounds _{X^{1}}\theta _{1} &=&j^{1}\sigma ^{\ast }(- 
\dot{\xi}\,dt+\dot{x}\dot{\tau}\,dt+d\xi -\dot{x}\,d\tau ), \\
&=&-\dot{\xi}\,dt+\dot{x}\dot{\tau}\,dt+\dot{\xi}\,dt-\dot{x}\dot{\tau}\,dt, \\
&=&0.
\end{eqnarray*}

On the other hand, since $\theta _{2}=d\dot{x}-\ddot{x}\,dt,$ and 
\begin{equation*}
X^{2}=\tau \frac{\partial }{\partial t}+\xi \frac{\partial }{\partial x}+( 
\dot{\xi}-\dot{x}\dot{\tau})\frac{\partial }{\partial \dot{x}}+(\ddot{\xi}-2 
\ddot{x}\dot{\tau}-\dot{x}\ddot{\tau})\frac{\partial }{\partial \ddot{x}}
\end{equation*} 
it follows that 
\begin{eqnarray*}
\pounds _{X^{2}}\theta _{2} &=&X^{2} 
\lefthook
d\theta _{2}+d\langle \theta _{2},X^{2}\rangle \\
&=&\left( d\dot{\xi}-\ddot{\xi}\,dt\right) -\dot{\tau}(d\dot{x}-\ddot{x}\,dt)- 
\ddot{x}\left\langle \frac{\partial \tau }{\partial x},\left( dx-\dot{x} \,
dt\right) \right\rangle \\
&&+\dot{x}\left\langle \frac{\partial }{\partial x}\left( \frac{\partial
\tau }{\partial t}+\frac{\partial \tau }{\partial x}\dot{x}\right) ,(dx-\dot{ 
x}\,dt)\right\rangle +\dot{x}\left\langle \frac{\partial \tau }{\partial x},(d 
\dot{x}-\ddot{x}\,dt)\right\rangle .
\end{eqnarray*} 
Therefore, 
\begin{eqnarray*}
j^{2}\sigma ^{\ast }\pounds _{X^{2}}\theta _{2} &=&j^{2}\sigma ^{\ast
}\left\{ \left( d\dot{\xi}-\ddot{\xi}\,dt\right) -\dot{\tau}(d\dot{x}-\ddot{x}\, 
dt)-\ddot{x}\left\langle \frac{\partial \tau }{\partial x},\left( dx-\dot{x} \,
dt\right) \right\rangle \right\} + \\
&&+j^{2}\sigma ^{\ast }\left\{ \dot{x}\left\langle \frac{\partial }{\partial
x}\left( \frac{\partial \tau }{\partial t}+\frac{\partial \tau }{\partial x} 
\dot{x}\right) ,(dx-\dot{x}\,dt)\right\rangle +\dot{x}\left\langle \frac{ 
\partial \tau }{\partial x},(d\dot{x}-\ddot{x}\,dt)\right\rangle \right\} \\
&=&0.
\end{eqnarray*}
\end{proof}

Let $\sigma $ be section of $I\times \R^{n}\rightarrow I,$ where $ 
I=[t_{0},t_{1}]$ and 
\begin{equation}
A(\sigma)=\int_{I}(L\circ j^{2}\sigma)\,dt=\int_{I}j^{2}\sigma^{\ast
}(L\,dt)=\int_{j^{2}\sigma(I)}L\,dt.  \label{A}
\end{equation}

\begin{proposition}
\label{4.3}The action integral (\ref{A}) is invariant under the one-parameter
local group $\exp tX^2$ of local diffeomorphisms of $J^{2}$ generated
by $X^{2}$ if 
\begin{equation*}
j^{2}\sigma ^{\ast }(\pounds _{X^{2}}(L\,dt))=0.
\end{equation*} 
Moreover, $\frac{d}{dt}A(\exp tX(\sigma ))|_{t=0}$ for every section $\sigma $, if and only if $\pounds _{X^{2}}(L\,dt)=0.$
\end{proposition}

\begin{proof}
Consider a vector field $X$ on $I\times \R^{n}$. Denote by $\exp tX$ the local one-parameter group of local diffeomorphisms of $ 
I\times \R^{n}$ generated by $X$, and by $\exp tX^{2}$ the
local one-parameter group of local diffeomorphisms of $J^{2}$ generated by
the prolongation $X^{2}$ of $X$ to $J^{2}$. Then, 
\begin{equation*}
A(\exp tX(\sigma ))=\int_{\exp tX^{2}(j^{2}\sigma (I))}L\,dt.
\end{equation*} 
The change of variables theorem asserts that $\int_{\phi _{t}(c)}\omega
=\int_{c}\phi _{t}^{\ast }\omega $ for any form $\omega ,$ chain $c$ and a
one-parameter group $\phi _{t}$ diffeomorphisms. Hence, 
\begin{equation*}
\int_{\exp tX^{2}(j^{2}\sigma (I))}L\,dt=\int_{j^{2}\sigma (I)}(\exp tX^{2})^{\ast }L\,dt.
\end{equation*} 
Therefore, 
\begin{equation*}
A(\exp tX(\sigma ))=\int_{j^{2}\sigma (I)}(\exp tX^{2})^{\ast }L\,dt.
\end{equation*} 
Now, differentiating under the integral sign with respect to $t$, 
\begin{equation*}
\frac{d}{dt}A(\exp tX(\sigma ))=\int_{j^{2}\sigma (I)} 
\frac{d}{dt}(\exp tX^{2})^{\ast }L\,dt=\int_{j^{2}\sigma (I)}(\exp tX^{2})^{\ast }\pounds _{X^{2}}L\,dt,
\end{equation*} 
and setting $t=0$, 
\begin{equation*}
\frac{d}{dt}A(\exp tX(\sigma ))|_{t =0}=\int_{j^{2}\sigma (I)}\pounds_{X^{2}}L\,dt=\int_{I}j^{2}\sigma ^{\ast }( 
\pounds_{X^{2}}(L\,dt)),
\end{equation*} 
it follows that if $j^{2}\sigma ^{\ast }\pounds _{X^{2}}L\,dt=0$, then $\frac{ 
d}{dt}A(\exp tX(\sigma ))|_{t=0}=0$. Moreover, \linebreak $\frac{d}{dt}A(\exp 
tX(\sigma ))|_{t=0}=0$ for every section $\sigma $, if and only if $\pounds 
_{X^{2}}L\,dt=0.$
\end{proof}

\begin{definition}
\label{4.4}A vector field $X$ on $I\times \R^{n}$ is an
infinitesimal symmetry of the Lagrange form $L\,dt$ if $\pounds  
_{X^{2}}(L\,dt)=0 $.
\end{definition}

\begin{lemma}
\label{4.5}The Lie derivative of the Lagrange form $L\,dt$ with respect to $ 
X^{2}$ is 
\begin{equation*}
\pounds _{X^{2}}(L\,dt)=\left( \tau \frac{\partial L}{\partial t}+\xi ^{i} 
\frac{\partial L}{\partial x^{i}}+(\dot{\xi}-\dot{x}\dot{\tau})\frac{ 
\partial L}{\partial \dot{x}}+(\ddot{\xi}-2\ddot{x}\dot{\tau}-\dot{x}\ddot{ 
\tau})\frac{\partial L}{\partial \ddot{x}}\right) dt+L\,d\tau .
\end{equation*} 
Hence, for every section $\sigma $ of $I\times \R^{n}\rightarrow I$, 
\begin{equation*}
j^{2}\sigma ^{\ast }(\pounds _{X^{2}}(L\,dt))=\left( \tau \frac{\partial L}{ 
\partial t}+\xi ^{i}\frac{\partial L}{\partial x^{i}}+(\dot{\xi}-\dot{x}\dot{ 
\tau})\frac{\partial L}{\partial \dot{x}}+(\ddot{\xi}-2\ddot{x}\dot{\tau}- 
\dot{x}\ddot{\tau})\frac{\partial L}{\partial \ddot{x}}+L\dot{\tau}\right)
dt.
\end{equation*}
\end{lemma}

\begin{proof}
The Lie derivative of the Lagrangian form $L\,dt$ with respect to $X^{2}$  is 
\begin{eqnarray*}
\pounds _{X^{2}}(L\,dt) &=&X^{2} 
\lefthook
dL\wedge dt+d(X^{2} 
\lefthook
L\,dt) \\
&=&X^{2} 
\lefthook
\left( \left( \frac{\partial L}{\partial t}dt+\frac{\partial L}{\partial
x^{i}}dx+\frac{\partial L}{\partial \dot{x}}d\dot{x}+\frac{\partial L}{ 
\partial \ddot{x}}d\ddot{x}\right) \wedge dt\right) +d(X^{2} 
\lefthook
L\,dt) \\
&=&X^{2} 
\lefthook
\left( \left( \frac{\partial L}{\partial x^{i}}dx+\frac{\partial L}{\partial 
\dot{x}}d\dot{x}+\frac{\partial L}{\partial \ddot{x}}d\ddot{x}\right) \wedge
dt\right) +d(X^{2} 
\lefthook
L\,dt)
\end{eqnarray*} 
because $dt\wedge dt=0.$ Hence, 
\begin{eqnarray*}
\pounds _{X^{2}}(Ldt) &=&\left( \xi ^{i}\frac{\partial L}{\partial x^{i}}+( 
\dot{\xi}-\dot{x}\dot{\tau})\frac{\partial L}{\partial \dot{x}}+(\ddot{\xi}-2 
\ddot{x}\dot{\tau}-\dot{x}\ddot{\tau})\frac{\partial L}{\partial \ddot{x}} 
\right) dt+ \\
&&-\tau \left( \frac{\partial L}{\partial x^{i}}dx+\frac{\partial L}{ 
\partial \dot{x}}d\dot{x}+\frac{\partial L}{\partial \ddot{x}}d\ddot{x} 
\right) +L\,d\tau +\tau \,dL \\
&=&\left( \tau \frac{\partial L}{\partial t}+\xi ^{i}\frac{\partial L}{ 
\partial x^{i}}+(\dot{\xi}-\dot{x}\dot{\tau})\frac{\partial L}{\partial \dot{ 
x}}+(\ddot{\xi}-2\ddot{x}\dot{\tau}-\dot{x}\ddot{\tau})\frac{\partial L}{ 
\partial \ddot{x}}\right) dt+L\,d\tau .
\end{eqnarray*}
\end{proof}

\begin{lemma}
\label{4.6}(Noether identities.) The equation $j^{2}\sigma ^{\ast }(\pounds  
_{X^{2}}(L\,dt))=0$ is equivalent to 
\begin{eqnarray}
&&\frac{d}{dt}\left( L\tau +\left( \frac{\partial L}{\partial \dot{x}}-\frac{ 
d}{dt}\left( \frac{\partial L}{\partial \ddot{x}}\right) \right) (\xi -\tau 
\dot{x})+\frac{\partial L}{\partial \ddot{x}}\frac{d}{dt}(\xi -\tau \dot{x} 
)\right) =  \label{Noether0} \\
&=&-\left( \frac{\partial L}{\partial x}-\frac{d}{dt}\left( \frac{\partial L 
}{\partial \dot{x}}\right) +\frac{d^{2}}{dt^{2}}\left( \frac{\partial L}{ 
\partial \ddot{x}}\right) \right) (\xi -\tau \dot{x}).  \notag
\end{eqnarray}
\end{lemma}

\begin{proof}
The details of this routine but lengthy calculation may be found in \cite{logan}.
\end{proof}
\begin{remark}
\label{4.7}The Noether identity is essentially the extension of the equation 
$j^{2}\sigma ^{\ast }(\pounds _{X^{2}}(L\,dt))=0$ to the fourth jet bundle.
More precisely, if $\pi _{4,2}:J^{4}\rightarrow J^{2}$ is the natural
projection and $X^{4}$ is the prolongation of $X$ to $J^{4}$, then 
\begin{eqnarray}
j^{4}\sigma ^{\ast }(\pi _{4,2}^{\ast }(\pounds _{X^{2}}(L\,dt)))
&=&j^{4}\sigma ^{\ast }\left( \left( \frac{\partial L}{\partial x}-\frac{d}{ 
dt}\left( \frac{\partial L}{\partial \dot{x}}\right) +\frac{d^{2}}{dt^{2}} 
\left( \frac{\partial L}{\partial \ddot{x}}\right) \right) (\xi -\tau \dot{x} 
)\right) +  \label{Noether2} \\
&& \! \! \! \! \! \! \! \! \! \! \! \! \! \! \!+j^{4}\sigma ^{\ast }\left( 
\frac{d}{dt} \left( L\tau +\left( \frac{ 
\partial L}{\partial \dot{x}}-\frac{d}{dt}\left( \frac{\partial L}{\partial 
\ddot{x}}\right) \right) (\xi -\tau \dot{x})+\frac{\partial L}{\partial 
\ddot{x}}\frac{d}{dt}(\xi -\tau \dot{x})\right) \right) .  \notag
\end{eqnarray}
\end{remark}

An immediate corollary of the Noether identity is the following
conservation law.

\begin{theorem}
\label{4.8}(First Noether theorem.) To every infinitesimal symmetry $X=\tau 
\frac{\partial }{\partial t}+\xi \frac{\partial }{\partial x}$ of the
Lagrange form $L\,dt,$ there corresponds a conserved quantity 
\begin{equation}
\mathscr{J}_{X}=L\tau +\left( \frac{\partial L}{\partial \dot{x}}-\frac{d}{dt 
}\left( \frac{\partial L}{\partial \ddot{x}}\right) \right) (\xi -\tau \dot{x 
})+\frac{\partial L}{\partial \ddot{x}}\frac{d}{dt}(\xi -\tau \dot{x}).
\label{PX}
\end{equation} 
That is, $\mathscr{J}_{X}$ is constant along solutions of the Euler-Lagrange equations 
\begin{equation*}
\frac{\partial L}{\partial x}-\frac{d}{dt}\left( \frac{\partial L}{\partial 
\dot{x}}\right) +\frac{d^{2}}{dt^{2}}\left( \frac{\partial L}{\partial \ddot{ 
x}}\right) =0.
\end{equation*} 
In other words, if $\sigma $ satisfies the Euler-Lagrange equations, then $ 
j^{3}\sigma ^{\ast }(X^{3} 
\lefthook
\Theta )$ is constant.
\end{theorem}

\begin{example}
\label{Conservation of linear momentum}(Conservation of linear momentum.) If
the Lagrangian $L$ does not depend on the coordinate $x^{i}$, then $X=\frac{ 
\partial }{\partial x^{i}}$ is an infinitesimal symmetry and the momentum 
\begin{equation*}
\mathscr{J}_{\frac{\partial }{\partial x^{i}}}=\frac{\partial L}{\partial 
\dot{x}^{i}}-\frac{d}{dt}\left( \frac{\partial L}{\partial \ddot{x}^{i}} 
\right)
\end{equation*} 
is conserved.
\end{example}

\begin{example}
\label{energy}(Conservation of energy) If the Lagrangian $L$ does not depend
on the parameter $t$, then $X=\frac{\partial }{\partial t}$ is an
infinitesimal symmetry and the energy 
\begin{equation*}
H=p_x\dot{x}+p_{\dot{x}}\ddot{x}-L=-\mathscr{J}_{\frac{\partial }{\partial t}}
\end{equation*} 
is conserved.
\end{example}

\begin{proof}
\begin{eqnarray*}
\pounds _{\frac{\partial }{\partial t}}(L\,dt) &=&\frac{\partial }{\partial t} 
\lefthook
d(L\,dt)+d(\frac{\partial }{\partial t} 
\lefthook
L\,dt) \\
&=&\frac{\partial }{\partial t} 
\lefthook
dL\wedge dt+dL \\
&=&-\left( \frac{\partial L}{\partial x}dx+\frac{\partial L}{\partial \dot{x} 
}d\dot{x}+\frac{\partial L}{\partial \ddot{x}}d\ddot{x}\right) +\left( \frac{ 
\partial L}{\partial t}dt+\frac{\partial L}{\partial x}dx+\frac{\partial L}{ 
\partial \dot{x}}d\dot{x}+\frac{\partial L}{\partial \ddot{x}}d\ddot{x} 
\right) \\
&=&\frac{\partial L}{\partial t}dt.
\end{eqnarray*} 
Hence, $\frac{\partial L}{\partial t}=0$ implies that $\frac{\partial }{ 
\partial t}$ is an infinitesimal symmetry of the Lagrange form $L\,dt.$
\end{proof}

Other conserved quantities will be discussed later.

\begin{remark}
There is a vast amount of work on symmetry principles and conservation laws
following Noether's fundamental paper \cite{noether}. Three works in 
particular are
noteworthy: the monographs by Logan \cite{logan} and 
Kosmann-Schwarzbach \cite{kosmann-schwarzbach}, and a
review by Krupkova \cite{krupkova2009}.
\end{remark}

\subsubsection{The Cartan form approach}

Recall that the Cartan form $\Theta $ is 
\begin{eqnarray}
\Theta &=&L\,dt+p_x\theta _{1}+p_{\dot{x}}\theta _{2}  \label{Theta2} \\
&=&L\,dt+\left( \frac{\partial L}{\partial \dot{x}}-\frac{d}{dt}\frac{\partial
L}{\partial \ddot{x}}\right) (dx-\dot{x}\,dt)+\frac{\partial L}{\partial \ddot{ 
x}}(d\dot{x}-\ddot{x}\,dt).  \notag
\end{eqnarray}

\begin{lemma}\label{cartan-form-lemma}
\label{4.9}For every section $\sigma :I\rightarrow I\times \R^{n}$,
and each vector field $X=\tau \frac{\partial }{\partial t}+\xi \frac{ 
\partial }{\partial x}$ on $I\times \R^{n}$, 
\begin{equation*}
j^{3}\sigma ^{\ast }(\pounds _{X^{3}}\Theta )=j^{2}\sigma ^{\ast }(\pounds  
_{X^{2}}(L\,dt)).
\end{equation*} 
In particular, if 
\begin{equation*}
j^{2}\sigma ^{\ast }(\pounds _{X^{2}}(L\,dt))=0,
\end{equation*} 
then, 
\begin{equation*}
j^{3}\sigma ^{\ast }(\pounds _{X^{3}}\Theta )=0.
\end{equation*}
\end{lemma}

\begin{proof}
By proposition \ref{4.2}, $\ $ 
\begin{equation*}
j^{2}\sigma ^{\ast }\pounds _{X^{2}}\theta _{2}=0\quad\text{and}\quad 
j^{1}\sigma
^{\ast }\pounds _{X^{1}}\theta _{1}=0.
\end{equation*} 
Lifting these equations to $J^{3}$, we get 
\begin{equation*}
j^{3}\sigma ^{\ast }\pounds _{X^{3}}\pi _{3,2}^{\ast }\theta _{2}=0\quad 
\text{and}\quad j^{3}\sigma ^{\ast }\pounds _{X^{3}}\pi _{3,1}^{\ast }\theta 
_{1}=0.
\end{equation*} 
Equation (\ref{Theta2}) written in terms of pull-backs of the jet
projections $\pi _{j,i}$ reads 
\begin{equation*}
\Theta =\pi _{3,2}^{\ast }(L\,dt)+\left( \frac{\partial L}{\partial \dot{x}}- 
\frac{d}{dt}\frac{\partial L}{\partial \ddot{x}}\right) \pi _{3,1}^{\ast
}\theta _{1}+\frac{\partial L}{\partial \ddot{x}}\pi _{3,2}^{\ast }\theta
_{2},
\end{equation*} 
where we consider the coefficients $\left( \frac{\partial L}{\partial \dot{x} 
}-\frac{d}{dt}\frac{\partial L}{\partial \ddot{x}}\right) $ and $\frac{ 
\partial L}{\partial \ddot{x}}$ as functions on $J^3$. Hence, 
\begin{eqnarray*}
j^{3}\sigma ^{\ast }(\pounds _{X^{3}}\Theta ) &=&j^{3}\sigma ^{\ast }( 
\pounds _{X^{3}}(\pi _{3,2}^{\ast }(L\,dt))+j^{3}\sigma ^{\ast }\left( \pounds  
_{X^{3}}\left( \frac{\partial L}{\partial \dot{x}}-\frac{d}{dt}\frac{ 
\partial L}{\partial \ddot{x}}\right) \right) j^{3}\sigma ^{\ast }\pi
_{3,1}^{\ast }\theta _{1} \\
&&+j^{3}\sigma ^{\ast }\left( \left( \frac{\partial L}{\partial \dot{x}}- 
\frac{d}{dt}\frac{\partial L}{\partial \ddot{x}}\right) \right) j^{3}\sigma
^{\ast }(\pounds _{X^{3}}\pi _{3,1}^{\ast }\theta _{1})+ \\
&&+j^{3}\sigma ^{\ast }\left( \pounds _{X^{3}}\left( \frac{\partial L}{ 
\partial \ddot{x}}\right) \right) j^{3}\sigma ^{\ast }(\pi _{3,2}^{\ast
}\theta _{2})+j^{3}\sigma ^{\ast }\left( \pounds _{X^{3}}\frac{\partial L}{ 
\partial \ddot{x}}\right) j^{3}\sigma ^{\ast }(\pounds _{X^{3}}\pi
_{3,2}^{\ast }\theta _{2}).
\end{eqnarray*} 
But, 
\begin{eqnarray*}
j^{3}\sigma ^{\ast }(\pounds _{X^{3}}(\pi _{3,2}^{\ast }(L\,dt))
&=&j^{2}\sigma ^{\ast }(\pounds _{X^{2}}(L\,dt)), \\
j^{3}\sigma ^{\ast }\pi _{3,1}^{\ast }\theta _{1} &=&j^{1}\sigma ^{\ast
}\theta _{1}=0, \\
j^{3}\sigma ^{\ast }(\pi _{3,2}^{\ast }\theta _{2}) &=&j^{2}\sigma ^{\ast
}\theta _{2}=0, \\
j^{3}\sigma ^{\ast }(\pounds _{X^{3}}\pi _{3,1}^{\ast }\theta _{1})
&=&j^{3}\sigma ^{\ast }(\pi _{3,1}^{\ast }(\pounds _{X^{1}}\theta
_{1}))=j^{1}\sigma ^{\ast }(\pounds _{X^{1}}\theta _{1})=0, \\
j^{3}\sigma ^{\ast }(\pounds _{X^{3}}\pi _{3,2}^{\ast }\theta _{2})
&=&j^{3}\sigma ^{\ast }(\pi _{3,2}^{\ast }(\pounds _{X^{2}}\theta
_{2}))=j^{2}\sigma ^{\ast }(\pounds _{X^{2}}\theta _{2})=0.
\end{eqnarray*} 
Hence, 
\begin{equation*}
j^{3}\sigma ^{\ast }(\pounds _{X^{3}}\Theta )=j^{2}\sigma ^{\ast }(\pounds  
_{X^{2}}(L\,dt)).
\end{equation*}
\end{proof}

\begin{proposition}
\label{JX}If $X=\tau \frac{\partial }{\partial t}+\xi \frac{\partial }{ 
\partial x}$ is an infinitesimal symmetry of the Lagrange form $L\,dt$ then,
for every section $\sigma $ of $I\times \R^{n}\rightarrow I,$ 
\begin{equation}
j^{3}\sigma ^{\ast }\mathscr{J}_{X}=j^{3}\sigma ^{\ast }(X^{3} 
\lefthook
\Theta ),  \label{Conserved quantity}
\end{equation} 
where $X^{3}$ is the prolongation of $X$ to $J^{3}$. If $\sigma $ satisfies
the Euler-Lagrange equations for $L$, then $j^{3}\sigma ^{\ast }(X^{3} 
\lefthook
\Theta )$ is constant.
\end{proposition}

\begin{proof}
Omitting pull-backs by $j^{3}\sigma $ for the sake of transparency, we may
write 
\begin{eqnarray}
&&(X^{3} 
\lefthook
\Theta )= \\
&=&\tau \frac{\partial }{\partial t}+\xi \frac{\partial }{\partial x}+(\dot{ 
\xi}-\dot{x}\dot{\tau})\frac{\partial }{\partial \dot{x}}+(\ddot{\xi}-2\ddot{ 
x}\dot{\tau}-\dot{x}\ddot{\tau})\frac{\partial }{\partial \ddot{x}}+(\dddot{ 
\xi}-3\dddot{x}\dot{\tau}-3\ddot{x}\ddot{\tau}-\dot{x}\dddot{\tau})\frac{ 
\partial }{\partial \ddot{x}}  \notag \\
&& 
\lefthook
\left( L\,dt+\left( \frac{\partial L}{\partial \dot{x}}-\frac{d}{dt}\frac{ 
\partial L}{\partial \ddot{x}}\right) (dx-\dot{x}\,dt)+\frac{\partial L}{ 
\partial \ddot{x}}(d\dot{x}-\ddot{x}\,dt)\right)  \notag \\
&=&\tau \left( L-\dot{x}\left( \frac{\partial L}{\partial \dot{x}}-\frac{d}{ 
dt}\frac{\partial L}{\partial \ddot{x}}\right) -\ddot{x}\frac{\partial L}{ 
\partial \ddot{x}}\right) +\xi \left( \frac{\partial L}{\partial \dot{x}}- 
\frac{d}{dt}\frac{\partial L}{\partial \ddot{x}}\right) +(\dot{\xi}-\dot{x} 
\dot{\tau})\frac{\partial L}{\partial \ddot{x}}  \notag \\
&=&L\tau +\left( \frac{\partial L}{\partial \dot{x}}-\frac{d}{dt}\left( 
\frac{\partial L}{\partial \ddot{x}}\right) \right) (\xi -\tau \dot{x})+ 
\frac{\partial L}{\partial \ddot{x}}\frac{d}{dt}(\xi -\tau \dot{x})=\mathscr{J}_{X}.  \notag
\end{eqnarray} 
Hence, $j^{3}\sigma^{\ast }(X^{3} )
\lefthook
\Theta )=j^{3}\sigma^{\ast }\mathscr{J}_{X}$ and is a constant if $\sigma $
satisfies the Euler-Lagrange equations. Suppose that $\gamma :I\rightarrow
J^{3}$ is a section of the source map such that its tangent $T\gamma (I)$ is
contained in $\ker d\Theta $. Since $\pounds _{Z}\Theta =Z \lefthook d\Theta +d(Z 
\lefthook \Theta )$, 
it follows that for any infinitesimal symmetry $Z$ of $\Theta $, we have 
\begin{equation*}
d\gamma ^{\ast }(Z 
\lefthook
\Theta )=-\gamma ^{\ast }d(Z 
\lefthook
\Theta )=-\gamma ^{\ast }\pounds _{Z}\Theta =0,
\end{equation*} 
which implies that $\gamma ^{\ast }(Z 
\lefthook
\Theta )$ is constant. This holds for any section $\gamma $ with $T\gamma
(I)\subset \ker \Theta .$ Moreover, if $\gamma =j^{3}\sigma $, where $\sigma 
$ satisfies the Euler-Lagrange equations, and $Z=X^{3}$, where $X$ is an
infinitesimal symmetry of the Lagrange form $L\,dt$, then lemma 
(\ref{cartan-form-lemma})
implies that 
\begin{equation}
\gamma ^{\ast }\pounds _{Z}\Theta =j^{2}\sigma ^{\ast }\pounds _{X}L\,dt.
\end{equation} 
Hence, if $j^{2}\sigma ^{\ast }(\pounds _{X^{2}}(L\,dt))=0$, then $j^{3}\sigma
^{\ast }(\pounds _{X^{3}}\Theta )=0,$ and $j^{3}\sigma ^{\ast }(X^{3} 
\lefthook
\Theta )$ is constant. Moroever, if $X=\tau \frac{\partial }{\partial t}+\xi 
\frac{\partial }{\partial x}$, then 
\begin{eqnarray}
&&(X^{3} 
\lefthook
\Theta )= \\
&=&\tau \frac{\partial }{\partial t}+\xi \frac{\partial }{\partial x}+(\dot{ 
\xi}-\dot{x}\dot{\tau})\frac{\partial }{\partial \dot{x}}+(\ddot{\xi}-2\ddot{ 
x}\dot{\tau}-\dot{x}\ddot{\tau})\frac{\partial }{\partial \ddot{x}}+(\dddot{ 
\xi}-3\dddot{x}\dot{\tau}-3\ddot{x}\ddot{\tau}-\dot{x}\dddot{\tau})\frac{ 
\partial }{\partial \ddot{x}}  \notag \\
&& 
\lefthook
\left( L\,dt+\left( \frac{\partial L}{\partial \dot{x}}-\frac{d}{dt}\frac{ 
\partial L}{\partial \ddot{x}}\right) (dx-\dot{x}\,dt)+\frac{\partial L}{ 
\partial \ddot{x}}(d\dot{x}-\ddot{x}\,dt)\right)  \notag \\
&=&\tau \left( L-\dot{x}\left( \frac{\partial L}{\partial \dot{x}}-\frac{d}{ 
dt}\frac{\partial L}{\partial \ddot{x}}\right) -\ddot{x}\frac{\partial L}{ 
\partial \ddot{x}}\right) +\xi \left( \frac{\partial L}{\partial \dot{x}}- 
\frac{d}{dt}\frac{\partial L}{\partial \ddot{x}}\right) +(\dot{\xi}-\dot{x} 
\dot{\tau})\frac{\partial L}{\partial \ddot{x}}  \notag \\
&=&L\tau +\left( \frac{\partial L}{\partial \dot{x}}-\frac{d}{dt}\left( 
\frac{\partial L}{\partial \ddot{x}}\right) \right) (\xi -\tau \dot{x})+ 
\frac{\partial L}{\partial \ddot{x}}\frac{d}{dt}(\xi -\tau \dot{x})  \notag
\end{eqnarray}
\end{proof}

Equation (\ref{Conserved quantity}) gives a simple way of finding constants
of motion corresponding to symmetries of the Lagrange form.

\begin{example}
\label{Conservation of angular momentum}If $x=(x^{i})$ are Cartesian
coordinates in $\R^{n}$, then the action of $\SO(n)$ on $J^{3}$ is
generated by vector fields 
\begin{equation*}
X_{ij}^{3}=x^{i}\frac{\partial }{\partial x^{j}}-x^{j}\frac{\partial }{ 
\partial x^{i}}+\dot{x}^{i}\frac{\partial }{\partial \dot{x}^{j}}-\dot{x}^{j} 
\frac{\partial }{\partial \dot{x}^{i}}+\ddot{x}^{i}\frac{\partial }{\partial 
\ddot{x}^{j}}-\ddot{x}^{j}\frac{\partial }{\partial \ddot{x}^{i}}+\dddot{x} 
^{i}\frac{\partial }{\partial \dddot{x}^{j}}-\dddot{x}^{j}\frac{\partial }{\partial \dddot{x}^{i}}.
\end{equation*} 
Hence, for a section $\sigma $ of $I\times \R^n\rightarrow I$, 
\begin{eqnarray*}
&&j^{3}\sigma ^{\ast }(\mathscr{J}_{X_{ij}})=j^{3}\sigma ^{\ast }(X_{ij}^{3} 
\lefthook
\Theta )= \\
&=&j^{3}\sigma ^{\ast }(x^{i}p_{x^j}-x^{j}p_{x^i}+\dot{x}^{i}p_{\dot{x}^j}- \dot{x}^{j}p_{\dot{x}^i}).
\end{eqnarray*} 
In the following we omit the symbol $j^{3}\sigma ^{\ast },$ and  write 
\begin{equation*}
\mathscr{J}_{X_{ij}}=x^{i}p_{x^j}-x^{j}p_{x^i}+\dot{x}^{i}p_{\dot{x}^j}-\dot{x}^
{ j}p_{\dot{x}^i}.
\end{equation*} 
If $L$ is invariant under the action of $\SO(3)$ on $J^{2}$, then $\mathscr{J} 
_{X_{ij}}$ is constant on solutions of the Euler-Lagrange equations.
\end{example}

This suggests that there may be additional conserved quantities coming from
symmetries of the Cartan form that are not symmetries of the Lagrange form.

\begin{definition}
\label{4.10} An infinitesimal symmetry of the Cartan form $\Theta $ is a
vector field $Z$ on $J^{3}$ such that $\pounds _{Z}\Theta =0.$
\end{definition}
Set
\begin{equation*}
\mathscr{J}_{Z}=Z 
\lefthook
\Theta
\end{equation*} 
for each infinitesimal symmetry $Z$ of the Cartan form $\Theta .$

\begin{theorem}
Let $Z$ be an infinitesimal symmetry of the Cartan form and let $\gamma
:I\rightarrow J^{3}$ be a section of the source map. If $T\gamma (I)$ is
contained in $\ker d\Theta $, then 
\begin{equation*}
\gamma ^{\ast }\mathscr{J}_{Z}=\gamma ^{\ast }(Z 
\lefthook \Theta )
\end{equation*} 
is constant. In particular, if $\sigma $ satisfies the Euler-Lagrange
equations, then $j^{3}\sigma \mathscr{J}_{Z}$ is a constant.
\end{theorem}

\begin{proof}
Since 
\begin{equation*}
\pounds _{Z}\Theta =Z 
\lefthook
d\Theta +d(Z 
\lefthook
\Theta ),
\end{equation*} 
it follows that for any infinitesimal symmetry $Z$ of $\Theta $, and any
section $\gamma :I\rightarrow J^{3}$ of the source map, that
\begin{equation*}
d\gamma ^{\ast }(Z 
\lefthook
\Theta )=-\gamma ^{\ast }(Z 
\lefthook
d\Theta ).
\end{equation*} 
If $T\gamma (I)$ is contained in $\ker d\Theta $, then $\gamma ^{\ast }(Z 
\lefthook
d\Theta )=0$ and $\gamma ^{\ast }(Z 
\lefthook
\Theta )$ is constant.
\end{proof}

\subsubsection{Symmetries up to a differential}

The Cartan form may yield more conserved
quantities than those that follow directly from the Lagrangian approach.
However, even more conserved quantities may arise if the notion of symmetry
is relaxed somewhat.

\begin{definition}
\label{4.11}A vector field $X$ on $I\times\R$ is a {\it symmetry up to
a differential} of the Lagrange form $L\, dt$ if there exists a function $F$
on $J^{2}$ such that 
\begin{equation*}
\pounds _{X^{2}}(L\, dt)=-dF,
\end{equation*} 
where $X^{2}$ is the prolongation of $X$ to $J^{2}.$
\end{definition}

\begin{proposition}
\label{4.12}If a vector field $X=\tau \frac{\partial }{\partial t}+\xi \frac{ 
\partial }{\partial x}$ on $I\times \R^n$ satisfies the condition 
\begin{equation}
\pounds _{X^{2}}(L\,dt)=-dF,  \label{differential1}
\end{equation} 
where $X^{2}$ is the prolongation of $X$ to $J^{2},$ and $F$ is a function
on $J^{2}$, then 
\begin{equation}
\mathscr{J}_{X}=F+L\tau +\left( \frac{\partial L}{\partial \dot{x}}-\frac{d}{ 
dt}\left( \frac{\partial L}{\partial \ddot{x}}\right) \right) (\xi -\tau 
\dot{x})+\frac{\partial L}{\partial \ddot{x}}\frac{d}{dt}(\xi -\tau \dot{x})
\end{equation} 
is constant along solutions of the Euler-Lagrange equations 
\begin{equation*}
\frac{\partial L}{\partial x}-\frac{d}{dt}\left( \frac{\partial L}{\partial 
\dot{x}}\right) +\frac{d^{2}}{dt^{2}}\left( \frac{\partial L}{\partial \ddot{ 
x}}\right) =0.
\end{equation*}
\end{proposition}

\begin{proof}
The pull-back of equation (\ref{differential1}) by the jet bundle projection 
$\pi _{4,2}:J^{4}\rightarrow J^{2}$ gives 
\begin{equation*}
\pi _{4,2}^{\ast }\pounds _{X^{2}}(L\,dt)=-\pi _{4,2}^{\ast }dF.
\end{equation*} 
Therefore, for a section $\sigma $ of $I\times \R^n,$ 
\begin{equation*}
j^{4}\sigma ^{\ast }(\pi _{4,2}^{\ast }\pounds _{X^{2}}(L\,dt))=-j^{4}\sigma
^{\ast }(\pi _{4,2}^{\ast }dF)=-j^{2}\sigma ^{\ast }dF.
\end{equation*} 
Taking into account the form of the Noether identity given in equation (\ref 
{Noether2}) yields
\begin{eqnarray*}
-j^{2}\sigma ^{\ast }dF &=&j^{4}\sigma ^{\ast }\left( \left( \frac{\partial L 
}{\partial x}-\frac{d}{dt}\left( \frac{\partial L}{\partial \dot{x}}\right) + 
\frac{d^{2}}{dt^{2}}\left( \frac{\partial L}{\partial \ddot{x}}\right)
\right) (\xi -\tau \dot{x})\right) + \\
&&+j^{4}\sigma ^{\ast }\left( \frac{d}{dt}\left( L\tau +\left( \frac{ 
\partial L}{\partial \dot{x}}-\frac{d}{dt}\left( \frac{\partial L}{\partial 
\ddot{x}}\right) \right) (\xi -\tau \dot{x})+\frac{\partial L}{\partial 
\ddot{x}}\frac{d}{dt}(\xi -\tau \dot{x})\right) \right) .
\end{eqnarray*} 
If $\sigma $ satisfies the Euler-Lagrange equations, then 
\begin{equation*}
j^{4}\sigma ^{\ast }\left( \left( \frac{\partial L}{\partial x}-\frac{d}{dt} 
\left( \frac{\partial L}{\partial \dot{x}}\right) +\frac{d^{2}}{dt^{2}} 
\left( \frac{\partial L}{\partial \ddot{x}}\right) \right) (\xi -\tau \dot{x} 
)\right) =0
\end{equation*} 
and consequently
\begin{equation*}
\mathscr{J}_{X}=F+\left( L\tau +\left( \frac{\partial L}{\partial \dot{x}}- 
\frac{d}{dt}\left( \frac{\partial L}{\partial \ddot{x}}\right) \right) (\xi
-\tau \dot{x})+\frac{\partial L}{\partial \ddot{x}}\frac{d}{dt}(\xi -\tau 
\dot{x})\right)
\end{equation*} 
is constant along $\sigma $.
\end{proof}

\begin{example}
Probably the most well-known example of this sort of behaviour occurs in the first-order theory as the Runge-Lenz vector in the Kepler problem.  In this case there is a symmetry of the dynamical system that is not lifted from the configuration space, and this implies that the Lagrangian is not invariant under the action of the symmetry group, but changes by a total derivative. This is discussed in \cite{levy-leblond}.
\end{example}

In a similar way, infinitesimal symmetries up to a differential of
the Cartan form are defined as

\begin{definition}
\label{4.13}A vector field $Z$ on $J^{3}$ is a {\it symmetry up to a
differential} of the Cartan form $\Theta $ if there exists a function $F$ on 
$J^{3}$ such that 
\begin{equation}
\pounds _{Z}\Theta =-dF.  \label{differential2}
\end{equation}
\end{definition}

\begin{proposition}
\label{4.14}If $\pounds _{Z}\Theta =-dF$, then for a section $\gamma
:I\rightarrow J^{3}$ such that $T(\gamma (I))$ is contained in $\ker d\Theta 
$, the function $F+\left\langle \Theta ,Z\right\rangle $ is constant along $ 
\gamma .$ In particular, $F+\left\langle \Theta ,Z\right\rangle $ is
constant along the jet extensions of sections $\sigma $ of $I\times \R^n$ that satisfy the Euler-Lagrange equations.
\end{proposition}

\begin{proof}
Since $\pounds _{Z}\Theta =Z 
\lefthook
d\Theta +d\left\langle \Theta ,Z\right\rangle $, equation (\ref 
{differential2}) gives 
\begin{equation*}
Z 
\lefthook
d\Theta =-d(F+\left\langle \Theta ,Z\right\rangle ).
\end{equation*} 
Hence, 
\begin{equation*}
\gamma ^{\ast }\left( Z 
\lefthook
d\Theta \right) =-\gamma ^{\ast }d(F+\left\langle \Theta ,Z\right\rangle
)=-d\gamma ^{\ast }(F+\left\langle \Theta ,Z\right\rangle ).
\end{equation*} 
However, $\gamma ^{\ast }\left( Z 
\lefthook
d\Theta \right) =0$ if $T(\gamma (I))$ is contained in $\ker d\Theta $.
Therefore, $d\gamma ^{\ast }(F+\left\langle \Theta ,Z\right\rangle )=0,$
which implies that $F+\left\langle \Theta ,Z\right\rangle $ is constant
along $\gamma .$

If a section $\sigma $ of $I\times \R^n$ satisfies the Euler-Lagrange
equations, then $T(j^{3}\sigma (I))$ is in the kernel of $d\Theta $.
Therefore, $F+\left\langle \Theta ,Z\right\rangle $ is constant along $ 
j^{3}\sigma $.
\end{proof}

\subsection{Parametrization invariance}

Let $\Diff_+\R$ be the group of orientation preserving
diffeomorphisms of the real line $\R$. Then, for every $\varphi \in \Diff_+\R$, 
$\dot{\varphi}(t)>0$ for all $t$. Each $\varphi \in \Diff_+\R$ gives rise to 
another diffeomorphism 
\begin{equation*}
\varphi ^{0}:\R\times \R^{n}\rightarrow \R\times 
\R^{n}:(t,x)\mapsto \varphi ^{0}(t,x)=(\varphi (t),x).
\end{equation*} 
The prolongations of $\varphi ^{0}$ to jet bundles can be written as follows 
\begin{eqnarray}
\varphi ^{1} &:&J^{1}\rightarrow J^{1}:(t,x,\dot{x})\mapsto \varphi^{1} 
(t,x,\dot{x})=\left( \varphi (t),x,\frac{\dot{x}}{\dot{\varphi}(t)} 
\right) ,  \label{action} \\
\varphi ^{2} &:&J^{2}\rightarrow J^{2}:(t,x,\dot{x},\ddot{x})\mapsto \varphi^{2} 
(t,x,\dot{x},\ddot{x})=\left( \varphi (t),x,\frac{\dot{x}}{\dot{ 
\varphi}(t)},\frac{\ddot{x}}{\dot{\varphi}(t)^{2}}-\dot{x}\frac{\ddot{\varphi 
}}{\dot{\varphi}(t)^{3}}\right) ,  \notag \\
\varphi ^{3} &:&J^{3}\rightarrow J^{3}:(t,x,\dot{x},\ddot{x},\dddot{x} 
)\mapsto \varphi ^{3}(t,x,\dot{x},\ddot{x})=  \notag \\
&=&\left( \varphi (t),x,\frac{\dot{x}}{\dot{\varphi}(t)},\frac{\ddot{x}}{ 
\dot{\varphi}(t)^{2}}-\dot{x}\frac{\ddot{\varphi}}{\dot{\varphi}(t)^{3}}, 
\frac{\dddot{x}}{\dot{\varphi}(t)^{3}}-3\ddot{x}\frac{\ddot{\varphi}(t)}{ 
\dot{\varphi}(t)^{4}}-\dot{x}\frac{\dddot{\varphi}(t)}{\dot{\varphi}(t)^{4}} 
+3\dot{x}\frac{\ddot{\varphi}(t)^{2}}{\dot{\varphi}(t)^{5}}\right) .  \notag
\end{eqnarray}

\begin{proposition}
For $\varphi \in \Diff_+\R$, 
\begin{eqnarray*}
\varphi ^{1\ast }\theta _{1} &=&\theta _{1}, \\
\varphi ^{2\ast }\theta _{2} &=&\frac{1}{\dot{\varphi}}\,\theta _{2}.
\end{eqnarray*}
\end{proposition}

\begin{proof}
Let $I=(-\infty ,\infty )$. The contact forms are $\theta _{1}=dx-\dot{x}\,dt$ and $\theta _{2}=d\dot{x}-\ddot{x}\,dt$.
Since 
\begin{equation*}
\varphi ^{1}(t,x,\dot{x})=(\varphi (t),x,\dot{x}/\dot{\varphi}(t)),
\end{equation*} 
it follows that
\begin{eqnarray*}
\varphi ^{1\ast }\theta _{1} = dx-x^{\prime }dt=dx-\frac{\dot{x}}{ 
\dot{\varphi}(t)}d(\varphi (t))=dx-\frac{\dot{x}}{\dot{\varphi}(t)}\dot{ 
\varphi}(t)\,dt = dx-\dot{x}\,dt=\theta _{1}.
\end{eqnarray*}
Similarly, 
\begin{equation*}
\varphi ^{2}:J^{2}\rightarrow J^{2}:(t,x,\dot{x},\ddot{x})\mapsto \varphi
^{2}(t,x,\dot{x},\ddot{x})=\left(\varphi 
(t),x,\frac{\dot{x}}{\dot{\varphi}(t)}, 
\frac{\ddot{x}}{\dot{\varphi}(t)^{2}}-\dot{x}\frac{\ddot{\varphi}}{\dot{ 
\varphi}(t)^{3}}\right)
\end{equation*} 
implies 
\begin{eqnarray*}
\varphi ^{2\ast }\theta _{2} &=&\varphi ^{2\ast }(dx^{\prime}-x^{\prime\prime} \,
dt) \\
&=&d\left( \frac{\dot{x}}{\dot{\varphi}(t)}\right) -\left( \frac{\ddot{x}}{ 
\dot{\varphi}(t)^{2}}-\dot{x}\frac{\ddot{\varphi}(t)}{\dot{\varphi}(t)^{3}} 
\right) d(\varphi (t)) \\
&=&\frac{d\dot{x}}{\dot{\varphi}(t)}-\frac{\dot{x}}{\dot{\varphi}(t)^{2}} 
\ddot{\varphi}(t)\,dt-\left( \frac{\ddot{x}}{\dot{\varphi}(t)^{2}}-\dot{x} 
\frac{\ddot{\varphi}(t)}{\dot{\varphi}(t)^{3}}\right) \dot{\varphi}(t)\,dt \\
&=&\frac{d\dot{x}}{\dot{\varphi}(t)}-\frac{\dot{x}}{\dot{\varphi}(t)^{2}} 
\ddot{\varphi}(t)\,dt-\frac{\ddot{x}}{\dot{\varphi}(t)}\,dt+\dot{x}\frac{\ddot{ 
\varphi}(t)}{\dot{\varphi}(t)^{2}}\,dt \\
&=&\frac{1}{\dot{\varphi}(t)}(d\dot{x}-\ddot{x}\,dt)=\frac{1}{\dot{\varphi}(t)} 
\theta _{2}.
\end{eqnarray*}
\end{proof}

A one-parameter subgroup $\varphi _{\varepsilon }:t\mapsto \bar{t}=\varphi
_{\varepsilon }(t)$ of $\Diff_{+}\R$ is generated by a vector field $ 
X_{\tau }=\tau (t)\partial_t$, where 
\begin{equation*}
\tau (t)=\left. \frac{\partial \varphi _{\varepsilon }(t)}{\partial \varepsilon 
}\right|_{\varepsilon=0}
\end{equation*} 
is an arbitrary smooth function on $R$ with $\dot{\tau}(t)\neq 0$ for all $ 
t\in \R$. The Lie algebra of the group $\Diff_{+}\R$ is the
collection of vector fields 
\begin{equation*}
\diff_{+}\R=\left\{ X_{\tau }=\tau (t)\partial_t 
\mid \tau \in C^{\infty }(\R)\text{, and }\dot{\tau}(t)\neq 0\text{
for all }t\right\}
\end{equation*} 
with the Lie bracket 
\begin{equation*}
\left[ \tau _{1}(t)\partial_t,\tau _{2}(t)\partial_t\right] =(\tau _{1}(t)\dot{\tau}_{2}(t)-\tau _{2}(t)\dot{\tau} 
_{1}(t))\partial_t.
\end{equation*}

\begin{definition}
The variational problem with the Lagrangian $L$ is {\it parametrization invariant}
if the Lagrange form $L\,dt$ is invariant under the action of $\Diff_{+}\R$ on $J^{2}$.
\end{definition}

\begin{remark}
Suppose that the Lagrange form $L\,dt$ is $\Diff_+\R$-invariant. This
implies that for $X_{\tau }=\tau (t)\partial_t$,
the Lagrange form $L\,dt$ is invariant under the one-parameter subgroup of $\Diff_+\R$ generated by $X_{\tau }$. By Theorem \ref{JX}, $\mathscr{J}_{X_{\tau }}=\langle \Theta ,X_{\tau }^{3}\rangle $ is constant
on solutions of the Euler-Lagrange equations.
\end{remark}

\begin{theorem}
\label{JX=0}If the Lagrange form $L\,dt$ is $\Diff_+\R$-invariant,
then 
\begin{equation}
j^{3}\sigma ^{\ast }\mathscr{J}_{X_{\tau }}=0  \label{J=0}
\end{equation} 
for all $X_{\tau }\in \diff_+\R$ and all solutions $\sigma $ of the
Euler-Lagrange equations.
\end{theorem}

\begin{proof}
Recall that $\Theta =p_x\,dx+p_{\dot{x}}\,d\dot{x}-H\,dt$. Omitting pull-backs by $ 
j^{3}\sigma $ for the sake of cleanliness, 
\begin{eqnarray}
\mathscr{J}_{X_{\tau }}=\langle \Theta ,X_{\tau }^{3}\rangle
& = & \left\langle pdx+p_{\dot{x}}d\dot{x}-Hdt,\tau \frac{\partial }{\partial t}- 
\dot{x}\dot{\tau}\frac{\partial }{\partial \dot{x}}\right\rangle
\label{JXtau} \\
& = &-\left\langle p_{\dot{x}},\dot{x}\right\rangle \dot{\tau}-H\tau .  \notag
\end{eqnarray} 
If $\sigma $ satisfies the Euler-Lagrange equations, then $j^{3}\sigma
^{\ast }\mathscr{J}_{X_{\tau }}$ is constant.

Take two points, $t_{0}<t_{1}$ in $I$, and consider two other vector fields $ 
X_{\tau _{1}}$ and $X_{\tau _{2}}$ in $\diff_{+}\R$ such that 
\begin{eqnarray*}
\tau (t_{0}) &=&\tau _{1}(t_{0})=\tau _{2}(t_{0})\text{ and }\dot{\tau} 
(t_{0})=\dot{\tau}_{1}(t_{0})=\dot{\tau}_{2}(t_{0}), \\
\tau (t_{1}) &\neq &\tau _{1}(t_{1})=\tau _{2}(t_{1})\text{ and }\dot{\tau} 
(t_{1})=\dot{\tau}_{1}(t_{1})\neq \dot{\tau}_{2}(t_{1}), \\
\tau (t_{2}) &\neq &\tau _{1}(t_{2})=\tau _{2}(t_{2})\text{ and }\dot{\tau} 
(t_{2})=\dot{\tau}_{1}(t_{2})\not=\dot{\tau}_{2}(t_{2}).
\end{eqnarray*} 
Then, $\mathscr{J}_{X_{\tau }}(t),$ $\mathscr{J}_{X_{\tau _{1}}}(t)$ and $ 
\mathscr{J}_{X_{\tau _{2}}}(t)$ are constant along $j^{3}\sigma $. Moreover,
the assumption that $\tau (t_{0})=\tau _{1}(t_{0})=\tau _{2}(t_{0})$, $\dot{ 
\tau}(t_{0})=\dot{\tau}_{1}(t_{0})=\dot{\tau}_{2}(t_{0})$ and equation (\ref 
{JXtau}) imply that $\mathscr{J}_{X_{\tau }}(t)=\mathscr{J}_{X_{\tau
_{1}}}(t)=\mathscr{J}_{X_{\tau _{2}}}(t)$ for all $t$. Therefore, $\mathscr{J 
}_{X_{\tau _{1}}}(t)-\mathscr{J}_{X_{\tau }}(t)=0$ and $\mathscr{J}_{X_{\tau
_{2}}}(t)-\mathscr{J}_{X_{\tau }}(t)=0\mathscr{\ }$for all $t$. Using
equation (\ref{JXtau}) and setting $t=t_{1}$ yields  
\begin{eqnarray*}
p_x(t_{1})\dot{x}(t_{1})\dot{\tau}_{1}(t_{1})+H(t_{1})\tau 
_{1}(t_{1})-p_x(t_{1}) 
\dot{x}(t_{1})\dot{\tau}(t_{1})-H(t_{1})\tau (t_{1}) &=&0 \\
p_x(t_{1})\dot{x}(t_{1})\dot{\tau}_{2}(t_{1})+H(t_{1})\tau 
_{2}(t_{1})-p_x(t_{1}) 
\dot{x}(t_{1})\dot{\tau}_{1}(t_{1})-H(t_{1})\tau _{1}(t_{1}) &=&0.
\end{eqnarray*} 
Since $\tau (t_{1})\neq \tau _{1}(t_{1})$ and $\dot{\tau}(t_{1})=\dot{\tau} 
_{1}(t_{1})$, the first equation above yields $H(t_{1})(\tau (t_{1})-\tau
_{1}(t_{1}))=0$, which implies that $H(t_{1})=0.$ Similarly, the assumption
that $\tau _{1}(t_{1})=\tau _{2}(t_{1})$ and $\dot{\tau}_{1}(t_{1})\neq \dot{ 
\tau}_{2}(t_{1})$ together with the second equation above yield $p_x(t_{1}) 
\dot{x}(t_{1})=0$. Since $t_{1}$ is an arbitrary point in $I$ different from 
$t_{0}$, it follows that 
\begin{equation}
H(t)=0\text{ \ and \ }p_x(t)\dot{x}(t)=0\text{ for all }t\in I.
\label{Noether 2}
\end{equation} 
Substituting this result into equation \ref{JXtau} gives (\ref{J=0}).
\end{proof}

\begin{remark}
\label{Second Noether Theorem}Equations (\ref{Noether 2}), rewritten in
terms of the configuration variables read 
\begin{eqnarray}
\dot{x}\frac{\partial L}{\partial \ddot{x}} &=&0,  \label{Noether 2a} \\
\dot{x}\left( \frac{\partial L}{\partial \dot{x}}-\frac{d}{dt}p_{\dot{x} 
}\right) +\ddot{x}\frac{\partial L}{\partial \ddot{x}}-L &=&0.  \notag
\end{eqnarray} 
These are the identities for our reparametrization invariant Lagrangian that
follow from the second Noether theorem (\cite{noether}). The proof of Theorem 
\ref{JX=0} establishes the equivalence between the Noether
identities (\ref{Noether 2a}) and the vanishing of the constant of motion $ 
\mathscr{J}_{X_{\tau }}$ corresponding to every $X_{\tau }\in \diff_{+} 
\R$.
\end{remark}

\subsection{Arclength parametrization}

Denote by $\left\langle x,x^{\prime }\right\rangle $ the Euclidean scalar
product and by $|x|=\sqrt{\left\langle x,x\right\rangle },$ the
corresponding norm in $\R^{n}$. For a curve $c:I\rightarrow \R^n:t\mapsto 
x(t)$, where $I=[t_{0},t_{1}]$, the arclength of the section
of $c$ from $t_{0}$ to $t$ is 
\begin{equation}
s(t)=\int_{t_{0}}^{t}\left\vert \dot{x} 
(t^{\prime })\right\vert dt^{\prime }.  \label{arc length}
\end{equation} 
In geometric problems it is often convenient to parametrize a curve in terms of
its arclength. If $t$ is the arclength of $c$, then along $c$ 
\begin{eqnarray}
\left\langle \dot{x},\dot{x}\right\rangle &=&\left\vert \dot{x}\right\vert
^{2}=1,  \label{arc-length1} \\
\left\langle \dot{x},\ddot{x}\right\rangle &=&0,  \label{arc-length2} \\
\left\langle \dot{x},\dddot{x}\right\rangle +\left\langle \ddot{x},\ddot{x} 
\right\rangle &=&0,  \label{arc-length3} \\
\langle \dot{x},\ddddot{x}\rangle +3\left\langle \ddot{x},\dddot{x} 
\right\rangle &=&0.  \label{arc-length4}
\end{eqnarray} 
These equations determine submanifolds $M^{1}$, $M^{2}$, $M^{3}$ and $M^{4}$
of $J^{1}$, $J^{2}$, $J^{3}$ and $J^{4}$, respectively.

\begin{proposition}
Let $X=\tau \frac{\partial }{\partial t}$ be a vector field on the 
configuration space $Q$. The
necessary and sufficient condition for its prolongations $X^{1},$ $X^{2},$ $ 
X^{3}$ and $X^{4}$ to be tangent to $M^{1}$, $M^{2}$, $M^{3}$ and $M^{4},$
respectively, is that the restriction of $\tau $ to $M^{1}$, $M^{2}$, $M^{3}$
and $M^{4},$ respectively, is constant.
\end{proposition}

\begin{proof}
The vector field 
\begin{equation*}
X^{1}=-\dot{\tau}\dot{x}^{i}\frac{\partial }{\partial \dot{x}^{i}}+\tau 
\frac{\partial }{\partial t}
\end{equation*}
on $J^{1}$, generating the reparametrization transformation, differentiating 
$\left\langle \dot{x},\dot{x}\right\rangle $ gives 
\begin{equation*}
\left( -\dot{\tau}\dot{x}^{i}\frac{\partial }{\partial \dot{x}^{i}}+\tau 
\frac{\partial }{\partial t}\right) \left\langle \dot{x},\dot{x}
\right\rangle =-\dot{\tau}\left\langle \dot{x}, \dot{x}\right\rangle .
\end{equation*}
Thus, $X^{1}$ is tangent to $M^{1}$ if and only if $\dot{\tau}=0$. Moreover, 
\begin{equation*}
\left. X^{1}\left\langle \dot{x},\dot{x}\right\rangle
\right|_{M^{1}}=-\dot{\tau}.
\end{equation*}
Next, 
\begin{eqnarray*}
X^{2}\left\langle \dot{x},\ddot{x}\right\rangle &=&\left( \tau \frac{ 
\partial }{\partial t}-\dot{\tau}\dot{x}^{i}\frac{ \partial }{\partial \dot{x 
}^{i}}-(2\ddot{x}_{i}\dot{\tau}+\dot{x}_{i}\ddot{ \tau})\frac{\partial }{ 
\partial \ddot{x}^{i}}\right) \left\langle \dot{x},\ddot{x}\right\rangle \\
&=&-3\dot{\tau} 
\left\langle \dot{x},\ddot{x}\right\rangle -\ddot{\tau} \left\langle \dot{x}, 
\dot{x}\right\rangle ,
\end{eqnarray*}
and 
\begin{equation*}
\left.\left( X^{2}\left\langle \dot{x},\ddot{x}\right\rangle
\right)\right|_{M^{2}}=-\ddot{\tau}.
\end{equation*}
Further, 
\begin{eqnarray*}
&&X^{3}(\left\langle \dot{x},\dddot{x}\right\rangle +\left\langle \ddot{x}, 
\ddot{x}\right\rangle )= \\
&=&\left( \tau \frac{\partial }{\partial t}-\dot{\tau}\dot{x}^{i}\frac{
\partial }{\partial \dot{x}^{i}}-(2\ddot{x}_{i}\dot{\tau}+\dot{x}_{i}\ddot{
\tau})\frac{\partial }{\partial \ddot{x}^{i}}-(3\dddot{x}_{i}\dot{\tau}+3 
\ddot{x}_{i}\ddot{\tau}+\dot{x}_{i}\dddot{\tau})\frac{\partial }{\partial 
\dddot{x}^{i}}\right) (\left\langle \dot{x},\dddot{x} \right\rangle
+\left\langle \ddot{x},\ddot{x} \right\rangle ) \\
&=&-4\dot{\tau}\left\langle \dot{x},\dddot{x}\right\rangle -3\ddot{\tau} 
\left\langle \dot{x},\ddot{x} \right\rangle -\dddot{\tau}\left\langle \dot{x} 
, \dot{x}\right\rangle -4\dot{\tau}\left\langle ( \ddot{x},\ddot{x} 
\right\rangle -2\ddot{\tau}\left\langle \dot{x},\ddot{x}\right\rangle ,
\end{eqnarray*}
and 
\begin{equation*}
\left. X^{3}(\left\langle \dot{x},\dddot{x}\right\rangle +\left\langle 
\ddot{x},\ddot{x}\right\rangle )\right|_{M^{3}}=-\dddot{\tau}.
\end{equation*}
Finally, 
\begin{eqnarray*}
&&X^{4}\left(\langle \dot{x},x^{(4)}\rangle +3\langle \ddot{ 
x},\dddot{x}\rangle \right) \\
&=&\left( \tau \frac{\partial }{\partial t}-\dot{\tau}\dot{x}^{i}\frac{
\partial }{\partial \dot{x}^{i}}-(2\ddot{x}_{i}\dot{\tau}+\dot{x}_{i}\ddot{
\tau})\frac{\partial }{\partial \ddot{x}^{i}}-(3\dddot{x}_{i}\dot{\tau}+3 
\ddot{x}_{i}\ddot{\tau}+\dot{x}_{i}\dddot{\tau})\frac{\partial }{\partial 
\dddot{x}^{i}}\right) \left( \left\langle \dot{x} ,x^{(4)}\right\rangle
+3\left\langle \ddot{x}, \dddot{x}\right\rangle \right) + \\
&&-\left( 4x^{(4)}\dot{\tau}+6\dddot{x}\ddot{\tau}+4\ddot{x}\dddot{\tau}+ 
\dot{x}\tau ^{(4)}\right) \frac{\partial }{\partial x^{(4)}}\left(
\left\langle \dot{x},x^{(4)}\right\rangle +3\left\langle \ddot{x},\dddot{x} 
\right\rangle \right) \\
&=&-5\dot{\tau}\langle \dot{x},x^{(4)}\rangle -9\ddot{\tau} 
\left\langle \dot{x},\dddot{x} \right\rangle -9\ddot{\tau}\left\langle \ddot{ 
x}, \ddot{x}\right\rangle -15\dot{\tau}\left\langle \ddot{x},\dddot{x} 
\right\rangle -7\dddot{\tau}\left\langle \dot{x},\ddot{x}\right\rangle -\tau
^{(4)}\left\langle \dot{x},\dot{x}\right\rangle .
\end{eqnarray*}
Therefore, $X^{4}\left( \langle \dot{x},x^{(4)}\rangle
+3\langle \ddot{x},\dddot{x} \rangle \right) =0$ if $\dot{\tau}=0$
and 
\begin{equation*}
X^{4}\left.\left( \langle \dot{x},x^{(4)}\rangle +3\langle 
\ddot{x},\dddot{x}\rangle \right)\right|_{M^{4}}=-\tau^{(4)}.
\end{equation*}
\end{proof}

Suppose a local section $\sigma $ of $Q$ with domain $I\subset 
\R$ and with $j^{1}\sigma (I)$ not in $M^{1}$, satisfies $\dot{x} 
(t)\neq 0$ for all $t\in I$.

\begin{lemma}
There exists $\varphi \in \Diff_{+}\R$ such that 
\begin{equation*}
\frac{d\varphi}{dt}=\left\vert \dot{x}(t)\right\vert
\end{equation*} 
for all $t\in I.$
\end{lemma}

\begin{proof}
This follows from the fundamental theorem of the calculus.
\end{proof}

Then, 
\begin{equation*}
\frac{dx}{d\varphi}=\frac{dx}{dt}\frac{dt}{d\varphi}=\frac{dx}{dt} 
\frac{1}{\left\vert \dot{x}(t)\right\vert }=\frac{\dot{x}(t)}{\left\vert 
\dot{x}(t)\right\vert },
\end{equation*} 
and it follows that 
\[
  \left|\frac{dx}{d\varphi}\right|=1.
\]
Thus the new parametrization gives rise to a section $\varphi^{\ast
}\sigma $ with its first jet in $M^{1}$. Similarly, the $k$-jet of 
$\varphi^{\ast }\sigma $ is in $M^{k}.$

\subsection{Hamiltonian formulation}

The Liouville form on the cotangent bundle $T^{\ast }J^{1}$ with variables $ 
(t,x,\dot{x},p_{t},p_{x},p_{\dot{x}})$ is 
\begin{equation}
\theta =p_{t}\,dt+p_{x}\,dx + p_{\dot{x}}\,d\dot{x}.   \label{Liouville}
\end{equation} 
The exterior derivative 
\begin{equation}
\omega =d\theta  \label{symplectic form}
\end{equation} 
is the canonical symplectic form of $T^{\ast }J^{1}$.

\begin{lemma}
The action  
\begin{equation*}
\Diff_{+}\R\times J^{1}\rightarrow J^{1}:(\varphi ,(t,x,\dot{x} 
))\mapsto \left( \varphi (t),x,\frac{\dot{x}}{\dot{\varphi}(t)}\right)
\end{equation*}
lifts to an action 
\begin{eqnarray}
\Diff_{+}\R\times T^{\ast }J^{1} &\rightarrow &T^{\ast }J^{1}:
\label{lifted action} \\
(\varphi ,(t,x,\dot{x},p_{t},p_{x},p_{\dot{x}})) &\mapsto &\left( \varphi
(t),x,\dot{\varphi}(t)^{-1}\dot{x},p_{t}\dot{\varphi}(t)^{-1}+\left\langle
p_{\dot{x}},\dot{x}\right\rangle \dot{\varphi}(t)^{-2}\ddot{\varphi} 
(t),p_{x},\dot{\varphi}(t)p_{\dot{x}}\right) .  \notag
\end{eqnarray} 
The lifted action (\ref{lifted action}) is Hamiltonian with momentum map $ 
\mathscr{J}:T^{\ast }J^{1}\rightarrow \diff_{+}\R^{\ast }$ such that,
for  $X=\tau (t)\partial_t\in \diff_{+}\R$, 
\begin{equation}
\mathscr{J}_{X}(t,x,\dot{x},p_{t},p_{x},p_{\dot{x}})=\tau (t)p_{t}-\dot{\tau} 
(t)\left\langle p_{\dot{x}},\dot{x}\right\rangle .  \label{Jtau}
\end{equation}
\end{lemma}

\begin{proof}
The action of $\Diff_{+}\R$ on $J^{1}$ is given by 
\begin{equation*}
\Diff_{+}\R\times J^{1}\rightarrow J^{1}:(\varphi ,(t,x,\dot{x} 
))\mapsto (t^{\prime },x^{\prime },\dot{x}^{\prime })=\left( \varphi (t),x, 
\dot{\varphi}(t)^{-1}\dot{x}\right) ,
\end{equation*} 
(see equation (\ref{action}).) The lifted action takes $(p_{t},p_{x},p_{\dot{x} 
})\ $to $(p_{t}^{\prime },p_{x}^{\prime },p_{\dot{x}}^{\prime })$ such that 
\begin{equation*}
p_{t}^{\prime }dt^{\prime }+p_{x}^{\prime }dx^{\prime }+p_{\dot{x}}^{\prime
}d\dot{x}^{\prime }=p_{t}dt+p_{x}dx+p_{\dot{x}}d\dot{x}.
\end{equation*} 
But $dt^{\prime }=\dot{\varphi}(t)dt$, $dx^{\prime }=dx$ and $d\dot{x} 
^{\prime }=\dot{\varphi}(t)^{-1}d\dot{x}-\dot{\varphi}(t)^{-2}\dot{x}\ddot{ 
\varphi}(t)\,dt$. Hence, 
\begin{eqnarray*}
p_{t}^{\prime }dt^{\prime }+p_{x}^{\prime }dx^{\prime }+p_{\dot{x}}^{\prime
}d\dot{x}^{\prime } &=&p_{t}^{\prime }\dot{\varphi}(t)dt+p_{x}dx+p_{\dot{x} 
}^{\prime }\dot{\varphi}(t)^{-1}d\dot{x}-\left\langle p_{\dot{x}}^{\prime }, 
\dot{x}\right\rangle \dot{\varphi}(t)^{-2}\ddot{\varphi}(t)dt \\
&=&(p_{t}^{\prime }\dot{\varphi}(t)-\left\langle p_{\dot{x}}^{\prime },\dot{x 
}\right\rangle \dot{\varphi}(t)^{-2}\ddot{\varphi}(t))\,dt+p_{x}dx+p_{\dot{x} 
}^{\prime }\dot{\varphi}(t)^{-1}d\dot{x} \\
&=&p_{t}\,dt+p_{x}\,dx+p_{\dot{x}}\,d\dot{x}.
\end{eqnarray*} 
Hence, 
\begin{eqnarray*}
p_{t}^{\prime }\dot{\varphi}(t)-\left\langle p_{\dot{x}}^{\prime },\dot{x} 
\right\rangle \dot{\varphi}(t)^{-2}\ddot{\varphi}(t) &=&p_{t}, \\
p_{x^{\prime }} &=&p_{x}, \\
p_{\dot{x}}^{\prime }\dot{\varphi}(t)^{-1} &=&p_{\dot{x}}.
\end{eqnarray*} 
Therefore, 
\begin{eqnarray*}
p_{t}^{\prime } &=&p_{t}\dot{\varphi}(t)^{-1}+\left\langle p_{\dot{x}},\dot{x 
}\right\rangle \dot{\varphi}(t)^{-2}\ddot{\varphi}(t), \\
p_{x}^{\prime } &=&p_{x}, \\
p_{\dot{x}}^{\prime } &=&\dot{\varphi}(t)p_{\dot{x}}.
\end{eqnarray*}

The lifted action on the cotangent bundle is Hamiltonian, and the value of
the momentum map on an element of the Lie algebra is given by the evaluation
of the Liouville form on the vector field generating the action of the
one-parameter group corresponding to this element of the Lie algebra. Hence,
for $X=\tau (t)\partial_t \in \diff_{+}\R$, 
\begin{eqnarray*}
\mathscr{J}_{X}(t,x,\dot{x},p_{t},p_{x},p_{\dot{x}}) &=&\left\langle 
\mathscr{J},\,X\right\rangle (t,x,\dot{x},p_{t},p_{x},p_{\dot{x} 
})=\left\langle \theta ,X^{1}\right\rangle (t,x,\dot{x},p_{t},p_{x},p_{\dot{x 
}}) \\
&=&\left\langle p_{t}dt+p_{x}dx+p_{\dot{x}}d\dot{x},\tau \frac{{\small  
\partial }}{{\small \partial t}}-\dot{x}\dot{\tau}\frac{{\small \partial }}{ 
{\small \partial }\dot{x}}\right\rangle \\
&=&\tau (t)p_{t}-\dot{\tau}(t)\left\langle p_{\dot{x}},\dot{x}\right\rangle .
\end{eqnarray*}
\end{proof}

\begin{definition}
The Legendre-Ostrogradski transformation 
\begin{equation}
\mathscr{L}:J^{3}\rightarrow T^{\ast }J^{1}:(t,x,\dot{x},\ddot{x},\dddot{x} 
)\mapsto (t,x,\dot{x},p_{t},p_{x},p_{\dot{x}}),
\label{Legendre-Ostrogradski}
\end{equation} 
is given by 
\begin{eqnarray*}
p_{t} &=&-H=-\dot{x}p_{x}-\ddot{x}p_{\dot{x}}+L \\
p_{\dot{x}} &=&\frac{\partial L}{\partial \ddot{x}}\text{,} \\
p_x &=&\frac{\partial L}{\partial \dot{x}}-\frac{d}{dt}\frac{\partial L}{ 
\partial \ddot{x}}=\frac{\partial L}{\partial \dot{x}}-\frac{d}{dt}p_{\dot{x}}.
\end{eqnarray*}
\end{definition}
The Legendre transformation was extended by Ostragradski in 
\cite{ostrogradski}.  However, for brevity, from now on we shall just refer to  
it as the Legendre transformation.

Clearly, $\mathscr{L}$ is a smooth map of $J^{3}$ into $T^{\ast }J^{1}$. If $ 
\mathscr{L}$ is a diffeomorphism, Ostrogradski's approach leads to a regular
time-dependent Hamiltonian theory with the Hamiltonian 
\begin{equation*}
H=-p_{t}=\dot{x}p_{x}+\ddot{x}p_{\dot{x}}-L.
\end{equation*} 
In geometric problems, the Lagrangian is often reparametrization
invariant. In this case $L$ does not depend on $t$ and the range of the
Legendre transformation is restricted by the equations (see (\ref{JX=0}))
\begin{equation*}
H=0\text{ \ and \ }\left\langle p_{\dot{x}},\dot{x}\right\rangle =0.
\end{equation*}

\begin{theorem}\label{intertwine}
If the Lagrange form $L\,dt$ is invariant under the action of $\Diff_{+}\R$, then the Legendre transformation $\mathscr{L}$ intertwines the actions
of $\Diff_{+}\R$ on $J^{3}$ and on $T^{\ast }J^{1}.$
\end{theorem}

\begin{proof}
The map $\mathscr{L}$, defined by (\ref{Legendre-Ostrogradski}), intertwines 
with $\varphi^3$, the induced action of $\Diff_+\R$ on $J^3$, if and only if
\begin{align}\label{Commute}
\tilde{\varphi}^1\circ\mathscr{L}(t,x,\dot{x},\ddot{x},\dddot{x}) = \mathscr{L}\circ\varphi^3(t,x,\dot{x},\ddot{x},\dddot{x}),
\end{align} 
for every $(t,x,\dot{x},\ddot{x},\dddot{x})\in J^3$.  Here the map 
$\tilde{\varphi}^1\colon T^*J^1\rightarrow T^*J^1$ is the lifted action of 
$\varphi^1$ to the cotangent bundle $T^*J^1$, defined by (\ref{lifted action}).

Let us first compute the right hand side of (\ref{Commute}).
Let $L'\colon J^2\rightarrow\R$ be a function on $J^2$ such that 
$(\varphi^2)^*L'=L$. This function is an induced Lagrangian by the 
$\Diff_+\R$-action. Based on (\ref{Legendre-Ostrogradski}), the map 
$\mathscr{L}\circ\varphi^3$ can be calculated by:
\begin{align*}
(t,x,\dot{x},\ddot{x},\dddot{x})&\mapsto (t',x,x',p'_t,p'_x,p'_{\dot{x}})\\
t'&=\varphi(t)\\
x'&=\frac{\dot{x}}{\dot{\varphi}}\\
p'_t&=-\left<p'_x,x'\right>-\left<p'_{\dot{x}},x''\right>+L'\\
p'_{\dot{x}}&=\frac{\partial L'}{\partial x''}\\
p'_x&=\frac{\partial L'}{\partial x'}-\frac{d}{dt'}(p'_{\dot{x}})\\
x''&=\frac{\ddot{x}}{\dot{\varphi}^2}-\frac{\ddot{\varphi}\dot{x}}{\dot{\varphi}^3}.
\end{align*}

In order to calculate the partial derivatives of $L'$, we use the assumption 
that the Lagrange form $L\,dt$ is invariant under the $\Diff_+\R$-action, i.e.,
\begin{align*}
(\varphi^2)^*(L'dt')=(L'\circ\varphi^2)\dot{\varphi}dt=Ldt \Longrightarrow L=\dot{\varphi}(L'\circ\varphi^2).
\end{align*}
Therefore, the partial derivatives of $L'$ with respect to $x, x'$ and $x''$ 
are calculated as
\begin{align*}
\frac{\partial L}{\partial x}&=\dot{\varphi} \frac{\partial L'}{\partial x} \Longrightarrow  \frac{\partial L'}{\partial x} = \frac{1}{\dot{\varphi}} \frac{\partial L}{\partial x}\\
\frac{\partial L}{\partial \ddot{x}} &= \dot{\varphi}\left( \frac{\partial L'}{\partial x''}\frac{\partial x''}{\partial \ddot{x}} \right) = \frac{1}{\dot{\varphi}}\frac{\partial L'}{\partial x''} \Longrightarrow \frac{\partial L'}{\partial x''} = \dot{\varphi}\frac{\partial L}{\partial \ddot{x}} \\
\frac{\partial L}{\partial \dot{x}} &= \dot{\varphi}\left( \frac{\partial L'}{\partial x'}\frac{\partial x'}{\partial \dot{x}} + \frac{\partial L'}{\partial x''}\frac{\partial x''}{\partial \dot{x}} \right) = \frac{\partial L'}{\partial x'} - \frac{\ddot{\varphi}}{\dot{\varphi}^2}\frac{\partial L'}{\partial x''} \Longrightarrow \frac{\partial L'}{\partial x'} = \frac{\partial L}{\partial \dot{x}} + \frac{\ddot{\varphi}}{\dot{\varphi}} \frac{\partial L}{\partial \ddot{x}}
\end{align*}
Since $L\,dt$ is parametrization invariant, by theorem (\ref{JX=0}), $L$ is 
time invariant, $p_t=-H=0$ and $\left<p_{\dot{x}},\dot{x}\right>=0$. As the 
result, we show that the partial derivative of $L'$ with respect to $t'$ is 
equal to zero:
\begin{align*}
0 &= \frac{\partial L}{\partial t} = \ddot{\varphi}L' + \dot{\varphi}\left( \frac{\partial L'}{\partial t'}\frac{\partial t'}{\partial t} + \frac{\partial L'}{\partial x'}\frac{\partial x'}{\partial t} +\frac{\partial L'}{\partial x''}\frac{\partial x''}{\partial t} \right)\\
&= \ddot{\varphi}\left( L' + \frac{1}{\dot{\varphi}} \left[ -\left<\frac{\partial L'}{\partial x'},\dot{x}\right> -\frac{2}{\dot{\varphi}} \left<\frac{\partial L'}{\partial x''},\ddot{x}\right> + \frac{3\ddot{\varphi}}{\dot{\varphi}^2} \left<\frac{\partial L'}{\partial x''},\dot{x}\right> \right] \right) + \dot{\varphi}^2\frac{\partial L'}{\partial t'} - \frac{\dddot{\varphi}}{\dot{\varphi}^2} \left<\frac{\partial L'}{\partial x''},\dot{x}\right>\\
&= \frac{\ddot{\varphi}}{\dot{\varphi}} \left( L - \left<\frac{\partial L}{\partial \dot{x}}, \dot{x}\right> - \frac{\ddot{\varphi}}{\dot{\varphi}}\left<p_{\dot{x}},\dot{x}\right> - 2\left<p_{\dot{x}},\ddot{x}\right> + \frac{3\ddot{\varphi}}{\dot{\varphi}^2}\left<p_{\dot{x}},\dot{x}\right>\right) + \dot{\varphi}^2\frac{\partial L'}{\partial t'} - \frac{\dddot{\varphi}}{\dot{\varphi}} \left<p_{\dot{x}},\dot{x}\right>\\ 
&= \frac{\ddot{\varphi}}{\dot{\varphi}} \left( L - \left<\frac{\partial L}{\partial \dot{x}}, \dot{x}\right> - 2\left<p_{\dot{x}},\ddot{x}\right>\right) + \dot{\varphi}^2\frac{\partial L'}{\partial t'}\\
&= \frac{\ddot{\varphi}}{\dot{\varphi}} \left( L - \left<\frac{\partial L}{\partial \dot{x}}-\frac{d}{dt}(p_{\dot{x}}), \dot{x}\right> -\frac{d}{dt}\left<p_{\dot{x}},\dot{x}\right> + \left<p_{\dot{x}},\ddot{x}\right> - 2\left<p_{\dot{x}},\ddot{x}\right>\right) + \dot{\varphi}^2\frac{\partial L'}{\partial t'} \\
&= \frac{\ddot{\varphi}}{\dot{\varphi}} \left( L - \left<p_x, \dot{x}\right> - \left<p_{\dot{x}},\ddot{x}\right>\right) + \dot{\varphi}^2\frac{\partial L'}{\partial t'} =  \frac{\ddot{\varphi}}{\dot{\varphi}} p_t + \dot{\varphi}^2\frac{\partial L'}{\partial t'} = \dot{\varphi}^2\frac{\partial L'}{\partial t'} \Longrightarrow \frac{\partial L'}{\partial t'} = 0.
\end{align*}
Now, we can calculate the right hand side of (\ref{Commute}) in terms of elements of $J^3$:
\begin{align*}
p'_{\dot{x}}&=\frac{\partial L'}{\partial x''} = \dot{\varphi}\frac{\partial L}{\partial \ddot{x}} = \dot{\varphi}p_{\dot{x}}\\
p'_x&=\frac{\partial L'}{\partial x'}-\frac{d}{dt'}(p'_{\dot{x}}) = \frac{\partial L}{\partial \dot{x}} + \frac{\ddot{\varphi}}{\dot{\varphi}} \frac{\partial L}{\partial \ddot{x}} -  \frac{1}{\dot{\varphi}} \frac{d}{dt}\left(\dot{\varphi}\frac{\partial L}{\partial \ddot{x}}\right) = \frac{\partial L}{\partial \dot{x}}-\frac{d}{dt}(p_{\dot{x}}) = p_x\\
p'_t&=-\left<p'_x,x'\right>-\left<p_{\dot{x}}',x''\right>+L' = -\left<p_x,\frac{\dot{x}}{\dot{\varphi}}\right>
-\left<\dot{\varphi}p_{\dot{x}},\frac{\ddot{x}}{\dot{\varphi}^2}-\frac{\ddot{\varphi}\dot{x}}{\dot{\varphi}^3}\right>+\frac{L}{\dot{\varphi}}\\
&= \frac{1}{\dot{\varphi}}\left(-\left<p_x,\dot{x}\right>-\left<p_{\dot{x}},\ddot{x}\right> + L\right) = \frac{1}{\dot{\varphi}} p_t = 0,
\end{align*} 
where $p_x$ and $p_{\dot{x}}$ are considered as functions on $J^3$, defined by:
\begin{align*}
p_x &= \frac{\partial L}{\partial \dot{x}}-\frac{d}{dt}\left(\frac{\partial L}{\partial \ddot{x}}\right),\\
p_{\dot{x}} &= \frac{\partial L}{\partial\ddot{x}}.
\end{align*} 

In the following, we calculate the left hand side of (\ref{Commute}). Since 
$L\, dt$ is $\Diff_+\R$-invariant, and based on (\ref{lifted action}):
\begin{align*}
\tilde{\varphi}^1\circ\mathscr{L}(t,x,\dot{x},\ddot{x},\dddot{x}) = 
\tilde{\varphi}^1 ( t, x , \dot{x}, 0, p_x,p_{\dot{x}}) = (\varphi(t), x, 
\dot{x}/\dot{\varphi}, 0, p_x, \dot{\varphi}p_{\dot{x}} ).
\end{align*}

Therefore, for $\Diff_+\R$-invariant Lagrange forms the relation 
(\ref{Commute}) holds. This completes the proof of the theorem.
\end{proof}

\begin{corollary}
If the Lagrange form $L\,dt$ is $\Diff_{+}\R$-invariant, then the
Cartan form $\Theta $ is $\Diff_{+}\R$-invariant.
\end{corollary}

\begin{proof}
Since $L\,dt$ is $\Diff_{+}\R$-invariant, theorem (\ref{intertwine}) implies 
that for
$\varphi \in \Diff_{+}\R$, $\varphi ^{3\ast }\mathscr{L}^{\ast
}=\mathscr{L}^{\ast }\tilde{\varphi}^{1\ast }$, where $\tilde{\varphi}^{1}$
is the lift of $\varphi ^{1}$ to the cotangent bundle $T^*J^{1}$. But, $\Theta 
= 
\mathscr{L}^{\ast }\theta $, and the Liouville form $\theta $ is invariant
under the the lifted action $\tilde{\varphi}^{1.}$ Therefore,  
\begin{equation*}
\varphi ^{3\ast }\Theta =\varphi ^{3\ast }\mathscr{L}^{\ast }\theta = 
\mathscr{L}^{\ast }\tilde{\varphi}^{1\ast }\theta =\mathscr{L}^{\ast }\theta
=\Theta .
\end{equation*}
\end{proof}

\begin{corollary}
The range of the Legendre transformation $\mathscr{L}:J^{3}\rightarrow
T^{\ast }J^{1}$ is contained in the zero set of the momentum map $\mathscr{J} 
:T^{\ast }J^{1}\rightarrow \diff_{+}\R^{\ast }.$ That is, 
\begin{equation*}
\mathscr{L}(J^{3})\subseteq \mathscr{J}^{-1}(0).
\end{equation*}
\end{corollary}

\section{Classical elastica}

\subsection{The variational equations}

The elastica functional is given by 
\begin{equation*}
A[\sigma ]=\int_{t_{0}}^{t_{1}}\kappa ^{2}\left\vert \dot{x}\right\vert \,dt,
\end{equation*} 
where 
\begin{equation*}
\sigma :[t_{0},t_{1}]\rightarrow \lbrack t_{0},t_{1}]\times \R 
^{3}:t\mapsto (t,x(t))
\end{equation*} 
corresponds to a curve $\sigma :t\mapsto x(t)$ in $\R^{3}$, and $ 
\kappa $ is the curvature of $\sigma $. Since the curvature of the curve
depends on its second derivatives, this is naturally a second order
variational problem.  As the curvature in Cartesian coordinates 
$x=(x^{1},x^{2},x^{3}) \in\R^3$ is  
\begin{equation*}
\kappa ^{2}=\frac{\left\vert \ddot{x}\right\vert ^{2}}{\left\vert \dot{x} 
\right\vert ^{4}}-\frac{\left\langle \dot{x}, \ddot{x}\right\rangle }{ 
\left\vert \dot{x}\right\vert ^{6}},
\end{equation*} 
the elastica Lagrangian is 
\begin{equation}
L(x,\dot{x},\ddot{x})=\frac{\left\vert \ddot{x}\right\vert ^{2}}{\left\vert 
\dot{x}\right\vert ^{3}}-\frac{\left\langle \dot{x}, \ddot{x} 
\right\rangle ^{2}}{\left\vert \dot{x}\right\vert ^{5}}.
\label{elastica Lagrangian}
\end{equation} 
It is defined on $\{(t,x,\dot{x},\ddot{x})\in J^{2}\mid \dot{x}\neq 0\}.$

\begin{proposition}
The elastica Lagrangian ({\rm \ref{elastica Lagrangian}}) is invariant under
translations and rotations in $\R^{3}$ and is independent of
parametrization.
\end{proposition}

\begin{proof}
The expression (\ref{elastica Lagrangian}) for $L$ is independent of $x$ and
 depends only on Euclidean scalar products of $\dot{x}$ and $\ddot{x}.$
Hence, $L$ is invariant under translations and rotations. Moreover, the
curvature $\kappa $ of a curve is independent of its parametrization, and $ 
\left\vert \dot{x}\right\vert dt=ds$ is the element of arclength.
Therefore, $L\,dt=\kappa ^{2}\,ds$ is independent of parametrization.
\end{proof}

For elastica, Ostrogradski's momenta are 
\begin{equation}
p_{\dot{x}}=2\frac{\ddot{x}}{\left\vert \dot{x}\right\vert ^{3}}-2\frac{ 
\left\langle \dot{x},\ddot{x}\right\rangle \dot{x}}{\left\vert \dot{x} 
\right\vert ^{5}},  \notag
\end{equation} 
and 
\begin{equation*}
p_{x}=-\frac{2}{\left\langle \dot{x},\dot{x}\right\rangle ^{5/2}} 
(\left\langle \dot{x},\dot{x}\right\rangle \dddot{x}-\left\langle \dot{x}, 
\dddot{x}\right\rangle \dot{x})-\frac{\left\langle \ddot{x},\ddot{x} 
\right\rangle \dot{x}}{\left\langle \dot{x},\dot{x}\right\rangle ^{5/2}}+6 
\frac{\left\langle \dot{x},\ddot{x}\right\rangle \ddot{x}}{\left\langle \dot{ 
x},\dot{x}\right\rangle ^{5/2}}-5\frac{\left\langle \dot{x},\ddot{x} 
\right\rangle ^{2}\dot{x}}{\left\langle \dot{x},\dot{x}\right\rangle ^{7/2}}.
\end{equation*} 
The Euler-Lagrange equations 
\begin{equation*}
\frac{\partial L}{\partial x}-\frac{d}{dt}\frac{\partial L}{\partial \dot{x}} 
+\frac{d^{2}}{dt^{2}}\frac{\partial L}{\partial \ddot{x}}=0
\end{equation*} 
can be written in the form 
\begin{equation*}
\frac{\partial L}{\partial x}-\frac{d}{dt}p_{x}=0.
\end{equation*} 
Since the Lagrangian is parameter-independent, the Euler-Lagrange equations
are necessarily degenerate in that they do not determine the fourth
derivative $\ddddot{x}$ uniquely. Let 
\begin{equation*}
\ddddot{x}^{\parallel }=\frac{\left\langle \ddddot{x},\dot{x}\right\rangle }{ 
\left\vert \dot{x}\right\vert ^{2}}\dot{x}\quad\text{ \ and \ 
}\quad \ddddot{x}^{\perp
}=\ddddot{x}-\ddddot{x}^{\parallel }
\end{equation*} 
denote the components of $\ddddot{x}$ that are parallel and perpendicular to 
$\dot{x}$, respectively. The Euler-Lagrange equations written in terms of
this decomposition are
\begin{equation}
\ddddot{x}^{\perp }=6\left\vert \dot{x}\right\vert ^{2}\left\langle \dot{x}, 
\ddot{x}\right\rangle \dddot{x}+4\left\langle \dot{x},\dddot{x}\right\rangle 
\ddot{x}+\frac{{\small 5}}{{\small 2}}\left\vert \ddot{x}\right\vert ^{2} 
\ddot{x}-10\left\langle \dot{x},\ddot{x}\right\rangle \left\langle \dot{x}, 
\dddot{x}\right\rangle -\frac{{\small 5}}{{\small 2}}\frac{\left\vert \ddot{x 
}\right\vert ^{2}\left\langle \dot{x},\ddot{x}\right\rangle }{\left\vert 
\dot{x}\right\vert ^{2}}\dot{x}+\frac{{\small 35}}{4}\frac{\left\langle \dot{ 
x},\ddot{x}\right\rangle ^{3}}{\left\vert \dot{x}\right\vert ^{6}}\dot{x}
\label{Euler-Lagrange 1}
\end{equation} 
They determine $\ddddot{x}^{\perp }$, but leave the
component $\ddddot{x}^{\parallel }$ undetermined. On the other hand, the
parametrization-invariance of the problem allows us to use parametrization
by the arclength. In the following, assume that $t$ is the arclength 
parameter of the curve. Therefore 
\begin{equation}
\left\vert \dot{x}\right\vert ^{2}=\left\langle \dot{x},\dot{x}\right\rangle
=1,  \label{arc-length 1}
\end{equation} 
and, by differentiation  
\begin{eqnarray}
\left\langle \dot{x},\ddot{x}\right\rangle &=&0,  \label{arc-length 2} \\
\left\langle \dot{x},\dddot{x}\right\rangle +\left\langle \ddot{x},\ddot{x} 
\right\rangle &=&0,  \label{arc-length 3} \\
\left\langle \dot{x},\ddddot{x}\right\rangle +3\left\langle \ddot{x},\dddot{x 
}\right\rangle &=&0,  \label{arc-length 4}
\end{eqnarray} 
as well.  Substitution into (\ref{Euler-Lagrange 1}) and (\ref{arc-length 
4}) yields 
\begin{equation}
\ddddot{x}^{\perp }=-\frac{{\small 3}}{{\small 2}}\left\vert \ddot{x} 
\right\vert ^{2}\ddot{x} \quad\text{and}\quad \ddddot{x}^{\parallel 
}=-3\left\langle \ddot{x},\dddot{x}\right\rangle \dot{x}. \label{elastica 1}
\end{equation}
These equations determine the elastica completely.  In other words, the choice 
of a parametrization determines an equation for the component 
$\ddddot{x}^{\parallel}$.  
\begin{remark}
  This yields an equation of the form 
\[
  \ddddot{x} = f(x,\dot{x},\ddot{x},\dddot{x})
\]
to which theorems in differential equations apply that guarantee the local 
existence and uniqueness of solutions.
\end{remark}

\subsection{The Frenet frame}
The elastica equations (\ref{elastica 1}) are
conveniently studied in the moving frame $(T,N,B)$, where $T=\dot{x}$ 
is the unit tangent vector, $N$ the normal vector and $B$ the binormal
vector of the curve $t\mapsto x(t)$. The Frenet equations are
\begin{eqnarray}
\dot{T} &=&\kappa N,  \label{Frenet 1} \\
\dot{N} &=&-\kappa T+\tau B,  \label{Frenet 2} \\
\dot{B} &=&-\tau N,  \label{frenet 3}
\end{eqnarray} 
with $\kappa =\left\vert \ddot{x}\right\vert $ the curvature and $\tau $
the torsion of the curve. In order to relate the torsion $\tau $ to the 
derivative variables, observe that 
\begin{equation}
\dddot{x}=\dot{\kappa}N+\kappa \dot{N}=\dot{\kappa}N-\kappa ^{2}\dot{x} 
+\kappa \tau B,  \label{torsion 0}
\end{equation}
which implies that, if $\kappa \neq 0$,  
\begin{equation}
\tau =\kappa ^{-1}\left\langle B,\dddot{x}\right\rangle =\kappa
^{-1}\left\langle T\times N,\dddot{x}\right\rangle =\kappa
^{-2}\left\langle \dot{x}\times \ddot{x},\dddot{x}\right\rangle .
\label{torsion}
\end{equation}

Differentiating (\ref{torsion 0}) and the Frenet equations imply 
\begin{eqnarray}
\ddddot{x} = -3\kappa \dot{\kappa}T+(\ddot{\kappa}-\kappa ^{3}-\kappa \tau
^{2})N+(2\dot{\kappa}\tau +\kappa \dot{\tau})B.
\end{eqnarray} 
This, together with equation (\ref{elastica 1}) implies that 
\begin{eqnarray}
2\dot{\kappa}\tau +\kappa \dot{\tau} &=&0,  \label{scalar 0} \\
2\ddot{\kappa}+\kappa ^{3}-2\kappa \tau ^{2} &=&0.  \label{scalar 1}
\end{eqnarray} 
Equation (\ref{elastica 1}) does not lead to any new condition because $ 
\kappa =\left\vert \ddot{x}\right\vert $ implies that $\left\langle \ddot{x}, 
\dddot{x}\right\rangle =\kappa \dot{\kappa}.$  Equation (\ref{scalar 0}) can be 
immediately integrated to yield 
\begin{equation}
\kappa ^{2}\tau =c, \qquad c\,\text{ a constant.}  \label{scalar 2}
\end{equation} 
If $\kappa \neq 0$, substituting $\tau =\frac{c}{\kappa ^{2}}$ into
equation (\ref{scalar 1}) and integrating gives 
\begin{equation}
\dot{\kappa}^{2}+\frac14\kappa^{4}+\frac{c^{2}}{\kappa^{2}}=~\text{constant.}  
\label{scalar 3}
\end{equation} 
Integration of equation (\ref{scalar 3}) determines completely the functions 
$\kappa (t)$ and $\tau (t)$ in terms of the initial data $\kappa
(t_{0})$, $\dot{\kappa}(t_{0})$ and $\tau (t_{0})$. Thus, in order to find
the solution $t\mapsto x(t)$, it suffices to integrate Frenet's equations
assuming that the curvature $\kappa $ and the torsion $\tau $ are known
functions of $t$. This can be achieved using the conservation laws for elastica.

\subsection{Conserved momenta}

Since the Lagrange form for elastica is invariant under translations, it
follows that the linear momentum $p=\frac{\partial L}{\partial \dot{x}}- 
\frac{d}{dt}\left( \frac{\partial L}{\partial \ddot{x}}\right) $ is
conserved. In the arclength parametrization
\begin{equation}
-2(\dddot{x}-\left\langle \dot{x},\dddot{x}\right\rangle \dot{x} 
)-\left\langle \ddot{x},\ddot{x}\right\rangle \dot{x}=p=\text{ constant.}
\label{p}
\end{equation} 
The arclength parametrization implies  
\begin{equation}
p=-2\dddot{x}-3\left\langle \ddot{x},\ddot{x}\right\rangle \dot{x}.
\label{p 1}
\end{equation} 
Similarly, rotational invariance of the Lagrange form implies that the
angular momentum 
\begin{equation*}
\mathscr{J}_{X_{ij}}=x^{i}p_{x^j}-x^{j}p_{x^i}+\dot{x}^{i}p_{\dot{x}^j}- 
\dot{x}^j p_{\dot{x}^i}
\end{equation*} 
is conserved (see example \ref{Conservation of angular momentum}.) Setting $ 
l=(l_{1},l_{2},l_{3})$, where 
\begin{equation*}
l_{i}=\epsilon _{ijk}\mathscr{J}_{X_{jk}},
\end{equation*} 
gives
\begin{equation}
l=x\times p_{x}+\dot{x}\times p_{\dot{x}}.  \label{l2}
\end{equation} 
The expression (\ref{elastica Lagrangian}) for the elastica Lagrangian in
arbitrary parametrization yields, in the arclength parametrization 
\begin{equation*}
p_{\dot{x}}=2\ddot{x}-\dot{x},
\end{equation*} 
which implies that 
\begin{equation}
l=x\times p+2\dot{x}\times \ddot{x}.  \label{l0}
\end{equation}

\begin{proposition}
The conserved momenta in the moving frame $(T,N,B)$ are 
\begin{eqnarray}
p &=&-\kappa ^{2}T-2\dot{\kappa}N-2\kappa \tau B,  \label{p1} \\
l &=&x\times p+2\kappa B.  \label{l}
\end{eqnarray}
\end{proposition}

\begin{proof}
Equation (\ref{torsion 0}) implies that 
\begin{equation*}
p=-2(\dot{\kappa}N-\kappa ^{2}T+\kappa \tau B)-3\kappa ^{2}T 
=-\kappa ^{2}T-2\dot{\kappa}N-2\kappa \tau B.
\end{equation*} 
while (\ref{Frenet 1}) and (\ref{frenet 3}) yield 
\begin{equation}
\dot{x}\times \ddot{x}=T\times (\kappa N)=\kappa (T\times
N)=\kappa B.  \label{b}
\end{equation}
\end{proof}

\begin{proposition}
Scalar equations for the curvature and torsion (\ref{scalar 2}) and (\ref 
{scalar 3}) can be rewritten in the form 
\begin{eqnarray}
\kappa ^{2}\tau &=&-\frac{{\small 1}}{{\small 4}}\left\langle
l,p\right\rangle \text{,}  \label{scalar 4} \\
\dot{\kappa}^{2}+\frac{{\small 1}}{{\small 4}}\kappa ^{4}+\frac{\left\langle
l,p\right\rangle ^{2}}{{\small 16}\kappa ^{2}} &=&\frac{{\small 1}}{{\small 4} 
}\left\vert p\right\vert ^{2}.  \label{scalar 5}
\end{eqnarray}
\end{proposition}

\begin{proof}
Equations (\ref{p1}) and (\ref{l}) imply 
\begin{equation*}
\left\langle l,p\right\rangle =2\kappa \left\langle B,p\right\rangle
=2\kappa \left\langle B,-\kappa ^{2}T-2\dot{\kappa}N-2\kappa \tau
B\right\rangle =-4\kappa ^{2}\tau .
\end{equation*} 
Equation (\ref{p1}) implies 
\begin{eqnarray*}
\left\vert p\right\vert ^{2} &=&\kappa ^{4}+4\dot{\kappa}^{2}+4\kappa
^{2}\tau ^{2}=4(\dot{\kappa}^{2}+\frac{1}{4}\kappa ^{4}+\frac{\kappa
^{4}\tau ^{2}}{\kappa ^{2}}) \\
&=&4\left(\dot{\kappa}^{2}+\frac{{\small 1}}{{\small 4}}\kappa ^{4}+\frac{ 
\left\langle l,p\right\rangle ^{2}}{{\small 16}\kappa ^{2}}\right).
\end{eqnarray*}
\end{proof}

Suppose that $\kappa $ and $\tau $ as known as functions of $t$ (the 
differential equations imply that they may be expressed as elliptic functions.) 
 It remains to show how the integration of the Frenet equations is aided by the 
conservation laws
\begin{eqnarray*}
p & = & -2\dot{\kappa}N-2\kappa \tau B-\kappa ^{2}T , \\
l & = & x\times p+2\kappa B .
\end{eqnarray*}

\begin{theorem}
The velocity of the elastica in the direction of the conserved momentum $p$ is 
\begin{equation}
\left<\dot{x},p\right> = -\kappa^2.
\label{xdot direction p}
\end{equation} 
Hence 
\begin{equation}
\left<x,p\right> = \left<x(t_0),p\right> - \int_{t_0}^{t} \kappa^2 ds.
\label{x direction p}
\end{equation}
\end{theorem}

\begin{proof}
Taking the scalar product of $l$ and $p$ yields
$$\left<B,p\right>=\frac{1}{2\kappa}\left<l,p\right>.$$
Moreover, by taking the derivative of both sides of this equation and using the
third Frenet equation,
$$\frac{d}{dt}\left<B,p\right> = \left<\dot{B},p\right> = -\tau\left<N,p\right> = \frac{d}{dt}\left(\frac{1}{2\kappa}\right)\left<l,p\right> = \frac{-\dot{\kappa}}{2\kappa^2}\left<l,p\right>.$$
Therefore if $\tau\neq 0$, by (\ref{scalar 4}) we have
$$\left<N,p\right> = \frac{\dot{\kappa}}{2\tau\kappa^2}\left<l,p\right> =  \frac{-2\dot{\kappa}}{\left<l,p\right>}\left<l,p\right> = -2\dot{\kappa}.$$
Finally, the second Frenet equation yields
$$\frac{d}{dt}\left<N,p\right> = -\kappa\left<\dot{x},p\right> + \tau\left<B,p\right> = -2\ddot{\kappa}.$$
Therefore if $\kappa\neq 0$, (\ref{scalar 1}) implies 
$$ \left<\dot{x},p\right> = \frac{2\ddot{\kappa}}{\kappa} + \frac{\tau}{\kappa}\left<B,p\right> = 
\frac{-\kappa^3 + 2\kappa\tau^2}{\kappa} + \frac{\tau}{2\kappa^2}\left<l,p\right> = 
-\kappa^2 + 2\tau^2 + \frac{\tau}{2\kappa^2}(-4\kappa^2\tau) = -\kappa^2. $$
In particular,
$$ \left<x,p\right> = \left<x(t_0),p\right> + \int_{t_0}^{t} \left<\dot{x}(s),p\right> ds = \left<x(t_0),p\right> - \int_{t_0}^{t} \kappa^2 ds. $$
\end{proof}

It remains to determine the component of the motion perpendicular to $p$.
If $\dot{x}$ is not parallel to $p$, then the vectors $\dot{x}\times p$ and $ 
(\dot{x}\times p)\times p$ span the plane of directions perpendicular to $p.$
Their lengths are 
\begin{equation*}
\left\vert \dot{x}\times p\right\vert =\left\vert -2\dot{\kappa}B+2\kappa
\tau N\right\vert =\sqrt{4\dot{\kappa}^{2}+4\kappa ^{2}\tau ^{2}}=\sqrt{ 
\left\vert p\right\vert ^{2}-\kappa ^{4}},
\end{equation*} 
and 
\begin{equation*}
\left\vert (\dot{x}\times p)\times p\right\vert =\sqrt{\left\vert
p\right\vert ^{4}-\left\vert p\right\vert ^{2}\kappa ^{4}}=\left\vert
p\right\vert \sqrt{\left\vert p\right\vert ^{2}-\kappa ^{4}}.
\end{equation*} 
Let $D$ and $E$ denote the unit vectors in the direction of $\dot{x}\times p$
and $(\dot{x}\times p)\times p$, respectively. Equations (\ref{p1}) and (\ref 
{l}) give 
\begin{equation}
D=\frac{\dot{x}\times p}{\left\vert \dot{x}\times p\right\vert }=\left(
\left\vert p\right\vert ^{2}-\kappa ^{4}\right) ^{-1/2}\left( -2\dot{\kappa} 
B+2\kappa \tau N\right) .  \label{d}
\end{equation} 
Similarly, 
\begin{equation}
E=\frac{(\dot{x}\times p)\times p}{\left\vert (\dot{x}\times p)\times
p\right\vert }=\frac{-(4\dot{\kappa}^{2}+4\kappa ^{2}\tau ^{2})T 
+2\kappa ^{2}\dot{\kappa}N+2\kappa ^{3}\tau B}{\left\vert p\right\vert
\left( \left\vert p\right\vert ^{2}-\kappa ^{4}\right) ^{1/2}}  \label{e}
\end{equation}

\begin{proposition}
\label{F}The frame $(D(t),E(t))$ satisfies the equations 
\begin{eqnarray*}
\dot{D} &=&-\frac{\left\langle l,p\right\rangle \left\vert p\right\vert }{ 
2\left( \left\vert p\right\vert ^{2}-\kappa ^{4}\right) }E, \\
\dot{E} &=&\frac{\left\langle l,p\right\rangle \left\vert p\right\vert }{ 
2(\left\vert p\right\vert ^{2}-\kappa ^{4})}D,
\end{eqnarray*} 
where $\kappa ^{2}\tau =-\frac{{\small 1}}{{\small 4}}\left\langle
l,p\right\rangle $.
\end{proposition}

\begin{proof}
 Equations (\ref{p1}) and (\ref{l}) give 
\begin{equation}
\dot{x}\times p=-2\dot{\kappa}B-2\kappa \dot{B}=-2\dot{\kappa}B+2\kappa \tau
N  \label{xdotp}
\end{equation} 
and 
\[
(\dot{x}\times p)\times p =\left\langle p,p\right\rangle \dot{x} 
-\left\langle \dot{x},p\right\rangle p=(4\dot{\kappa}^{2}+4\kappa ^{2}\tau 
^{2})T-2\kappa ^{2}\dot{\kappa} N-2\kappa ^{3}\tau B.  
\] 
Differentiation yields
\begin{eqnarray*}
\frac{{\small d}}{{\small dt}}(\dot{x}\times p) &=&\ddot{x}\times p=\kappa
N\times (-2\dot{\kappa}N-2\kappa \tau B-\kappa ^{2}T) \\
&=&-2\kappa ^{2}\tau T+\kappa ^{3}B \\
&=&-\frac{2\kappa ^{2}\tau }{(4\dot{\kappa}^{2}+4\kappa ^{2}\tau ^{2})}\{( 
\dot{x}\times p)\times p+2\kappa ^{2}\dot{\kappa}N+2\kappa ^{3}\tau
B\}+\kappa ^{3}B \\
&=&-\frac{2\kappa ^{2}\tau }{(4\dot{\kappa}^{2}+4\kappa ^{2}\tau ^{2})}(\dot{ 
x}\times p)\times p-\frac{4\kappa ^{4}\dot{\kappa}\tau }{(4\dot{\kappa} 
^{2}+4\kappa ^{2}\tau ^{2})}N+\frac{4\dot{\kappa}^{2}}{(4\dot{\kappa} 
^{2}+4\kappa ^{2}\tau ^{2})}\kappa ^{3}B \\
&=&-\frac{2\kappa ^{2}\tau }{(4\dot{\kappa}^{2}+4\kappa ^{2}\tau ^{2})}(\dot{ 
x}\times p)\times p-\frac{2\kappa ^{3}\dot{\kappa}}{(4\dot{\kappa} 
^{2}+4\kappa ^{2}\tau ^{2})}(2\kappa \tau N-2\dot{\kappa}B) \\
&=&-\frac{2\kappa ^{2}\tau }{(4\dot{\kappa}^{2}+4\kappa ^{2}\tau ^{2})}(\dot{ 
x}\times p)\times p-\frac{2\kappa ^{3}\dot{\kappa}}{(4\dot{\kappa} 
^{2}+4\kappa ^{2}\tau ^{2})}(\dot{x}\times p),
\end{eqnarray*} 
and 
\begin{eqnarray*}
\frac{{\small d}}{{\small dt}}(\dot{x}\times p)\times p &=&(\ddot{x}\times
p)\times p \\
&=&\left( -\frac{2\kappa ^{2}\tau }{(4\dot{\kappa}^{2}+4\kappa ^{2}\tau ^{2}) 
}(\dot{x}\times p)\times p-\frac{2\kappa ^{3}\dot{\kappa}}{(4\dot{\kappa} 
^{2}+4\kappa ^{2}\tau ^{2})}(\dot{x}\times p)\right) \times p \\
&=&-\frac{2\kappa ^{2}\tau }{(4\dot{\kappa}^{2}+4\kappa ^{2}\tau ^{2})}(( 
\dot{x}\times p)\times p)\times p-\frac{2\kappa ^{3}\dot{\kappa}}{(4\dot{ 
\kappa}^{2}+4\kappa ^{2}\tau ^{2})}(\dot{x}\times p)\times p \\
&=&-\frac{2\kappa ^{2}\tau }{(4\dot{\kappa}^{2}+4\kappa ^{2}\tau ^{2})}(- 
\dot{x} | p|^{2} +p\left\langle \dot{x},p\right\rangle
)\times p-\frac{2\kappa ^{3}\dot{\kappa}}{(4\dot{\kappa}^{2}+4\kappa
^{2}\tau ^{2})}(\dot{x}\times p)\times p \\
&=&\frac{2\kappa ^{2}\tau | p|^{2}\vert }{(4\dot{\kappa} 
^{2}+4\kappa ^{2}\tau ^{2})}\dot{x}\times p-\frac{2\kappa ^{3}\dot{\kappa}}{ 
(4\dot{\kappa}^{2}+4\kappa ^{2}\tau ^{2})}(\dot{x}\times p)\times p.
\end{eqnarray*}

Now compute for the orthonormal frame  
\begin{equation*}
\left\{\left\vert \dot{x}\times p\right\vert ^{-1}\dot{x}\times p,\left\vert ( 
\dot{x}\times p)\times p\right\vert ^{-1}(\dot{x}\times p)\times p\right\}.
\end{equation*} 
Since 
\begin{equation*}
\left\vert \dot{x}\times p\right\vert =\left\vert -2\dot{\kappa}b+2\kappa
\tau n\right\vert =\sqrt{4\dot{\kappa}^{2}+4\kappa ^{2}\tau ^{2}}=\sqrt{ 
\left\vert p\right\vert ^{2}-\kappa ^{4}},
\end{equation*} 
it follows that
\begin{equation*}
\frac{{\small d}}{{\small dt}}\left\vert \dot{x}\times p\right\vert ^{-1}= 
\frac{{\small d}}{{\small dt}}(\left\vert p\right\vert ^{2}-\kappa
^{4})^{-1/2}=-\frac{{\small 1}}{{\small 2}}(\left\vert p\right\vert
^{2}-\kappa ^{4})^{-3/2}(-4\kappa ^{3}\dot{\kappa})=2\kappa ^{3}\dot{\kappa} 
(\left\vert p\right\vert ^{2}-\kappa ^{4})^{-3/2}.
\end{equation*} 
This further implies
\begin{eqnarray}
\frac{{\small d}}{{\small dt}}(\dot{x}\times p) &=&-\frac{2\kappa ^{2}\tau }{ 
(4\dot{\kappa}^{2}+4\kappa ^{2}\tau ^{2})}(\dot{x}\times p)\times p-\frac{ 
2\kappa ^{3}\dot{\kappa}}{(4\dot{\kappa}^{2}+4\kappa ^{2}\tau ^{2})}(\dot{x} 
\times p)  \label{revised} \\
&=&-\frac{2\kappa ^{2}\tau }{\left\vert p\right\vert ^{2}-\kappa ^{4}}(\dot{x 
}\times p)\times p-\frac{2\kappa ^{3}\dot{\kappa}}{\left\vert p\right\vert
^{2}-\kappa ^{4}}(\dot{x}\times p),  \notag \\
\frac{{\small d}}{{\small dt}}(\dot{x}\times p)\times p &=&\frac{2\kappa
^{2}\tau \left\vert p^{2}\right\vert }{(4\dot{\kappa}^{2}+4\kappa ^{2}\tau
^{2})}\dot{x}\times p-\frac{2\kappa ^{3}\dot{\kappa}}{(4\dot{\kappa} 
^{2}+4\kappa ^{2}\tau ^{2})}(\dot{x}\times p)\times p  \notag \\
&=&\frac{2\kappa ^{2}\tau \left\vert p^{2}\right\vert }{\left\vert
p\right\vert ^{2}-\kappa ^{4}}\dot{x}\times p-\frac{2\kappa ^{3}\dot{\kappa} 
}{\left\vert p\right\vert ^{2}-\kappa ^{4}}(\dot{x}\times p)\times p.  \notag
\end{eqnarray} 
Similarly, 
\begin{eqnarray*}
\left\vert (\dot{x}\times p)\times p\right\vert ^{2} &=&\left\vert -\dot{x} 
\left\vert p\right\vert ^{2}+p\left\langle \dot{x},p\right\rangle
\right\vert ^{2}=\left\vert p\right\vert ^{4}-\left\vert
p\right\vert ^{2}\kappa ^{4},
\end{eqnarray*} 
and thus 
\begin{eqnarray*}
\frac{{\small d}}{{\small dt}}\left\vert (\dot{x}\times p)\times
p\right\vert ^{-1} &=&\frac{{\small d}}{{\small dt}}(\left\vert p\right\vert
^{4}-\left\vert p\right\vert ^{2}\kappa ^{4})^{-1/2}=\left\vert p\right\vert
^{-1}\frac{{\small d}}{{\small dt}}(\left\vert p\right\vert ^{2}-\kappa
^{4})^{-1/2} \\
&=&\frac{{\small 2}\kappa ^{3}\dot{\kappa}}{\left\vert p\right\vert } 
(\left\vert p\right\vert ^{2}-\kappa ^{4})^{-3/2}.
\end{eqnarray*} 
Therefore, 
\begin{eqnarray}
\frac{{\small d}}{{\small dt}}(\left\vert \dot{x}\times p\right\vert ^{-1} 
\dot{x}\times p) &=&\frac{{\small d}}{{\small dt}}(\left\vert \dot{x}\times
p\right\vert ^{-1})\dot{x}\times p+\left\vert \dot{x}\times p\right\vert
^{-1}\frac{{\small d}}{{\small dt}}(\dot{x}\times p) \\
&=&-\frac{2\kappa ^{2}\tau }{\left( \left\vert p\right\vert ^{2}-\kappa
^{4}\right) ^{3/2}}(\dot{x}\times p)\times p \\
&=&-\frac{2\kappa ^{2}\tau \left\vert p\right\vert }{\left( \left\vert
p\right\vert ^{2}-\kappa ^{4}\right) }\frac{1}{\left\vert (\dot{x}\times
p)\times p\right\vert }(\dot{x}\times p)\times p.
\end{eqnarray} 
Similarly, 
\begin{eqnarray*}
\frac{{\small d}}{{\small dt}}(\left\vert (\dot{x}\times p)\times
p\right\vert ^{-1}(\dot{x}\times p)\times p) &=&\left( \frac{{\small d}}{ 
{\small dt}}\left\vert (\dot{x}\times p)\times p\right\vert ^{-1}\right) ( 
\dot{x}\times p)\times p+  \\
 & \phantom{=} & +\left\vert (\dot{x}\times p)\times p\right\vert 
\frac{{\small d}}{{\small dt}}((\dot{x}\times p)\times p) \\
&=&(\left\vert p\right\vert ^{2}-\kappa ^{4})^{-3/2}2\kappa ^{2}\tau
\left\vert p\right\vert \dot{x}\times p \\
&=&\frac{2\kappa ^{2}\tau \left\vert p\right\vert }{(\left\vert p\right\vert
^{2}-\kappa ^{4})^{3/2}}\frac{\left\vert \dot{x}\times p\right\vert }{ 
\left\vert \dot{x}\times p\right\vert }\dot{x}\times p \\
&=&\frac{2\kappa ^{2}\tau \left\vert p\right\vert }{(\left\vert p\right\vert
^{2}-\kappa ^{4})}\frac{1}{\left\vert \dot{x}\times p\right\vert }\dot{x} 
\times p.
\end{eqnarray*} 
Thus, 
\begin{eqnarray*}
\dot{D} &=&-\frac{2\kappa ^{2}\tau \left\vert p\right\vert }{\left(
\left\vert p\right\vert ^{2}-\kappa ^{4}\right) }E, \\
\dot{E} &=&\frac{2\kappa ^{2}\tau \left\vert p\right\vert }{(\left\vert
p\right\vert ^{2}-\kappa ^{4})}D,
\end{eqnarray*} 
and the proof is finished since $\kappa ^{2}\tau =-\frac{{\small 
1}}{{\small 4}}\left\langle
l,p\right\rangle .$
\end{proof}

Define the curve of complex-valued vectors $Z(t)$ by 
\begin{equation}
Z(t)=D(t)+iE(t).  \label{z}
\end{equation} 
Proposition \ref{F} implies
\begin{equation*}
\dot{Z}=\dot{D}+i\dot{E}=-i\frac{\left\langle l,p\right\rangle
\left\vert p\right\vert }{2(\left\vert p\right\vert ^{2}-\kappa ^{4})}Z.
\end{equation*}

\begin{proposition}
Define 
\begin{equation*}
\phi (t)=\left\langle l,p\right\rangle \left\vert p\right\vert
\int_{t_{0}}^{t}\frac{1}{(\left\vert p\right\vert ^{2}-\kappa ^{4})} 
ds,
\end{equation*} 
and set $Z_{0}=Z(t_{0})=Z_{0}=D_{0}+iE_{0},$ then 
\begin{equation}
Z(t)=e^{-i\phi (t)}Z_{0}.  \label{z(t)}
\end{equation} 
In particular, 
\begin{eqnarray}
D(t) &=&\cos \phi (t)D_{0}+\sin \phi (t)E_{0},  \label{d(t)} \\
E(t) &=&-\sin \phi (t)D_{0}+\cos \phi (t)E_{0}.  \label{e(t)}
\end{eqnarray}
\end{proposition}

\begin{proof} This follows immediately upon differentiating
\begin{equation*}
\dot{Z}=\frac{d}{dt}Z=\frac{d}{dt}e^{-i\phi (t)}Z_{0}=-i\dot{\phi}e^{-i\phi
(t)}Z_{0}=-i\frac{\left\langle l,p\right\rangle \left\vert p\right\vert }{ 
2(\left\vert p\right\vert ^{2}-\kappa ^{4})}Z,
\end{equation*} 
and $Z(t_{0})=e^{-i\phi (t_{0})}Z_{0}=Z_{0}.$
\end{proof}

Note that 
\begin{equation*}
(\dot{x}\times p)\times p=-\left\vert p\right\vert ^{2}\dot{x}+\left\langle 
\dot{x},p\right\rangle p
\end{equation*} 
implies that 
\begin{equation*}
\dot{x}^{\perp}_p:=-\left\vert p\right\vert ^{-2}(\dot{x}\times p)\times p
\end{equation*} 
is the component of $\dot{x}$ perpendicular to $p$.

\begin{theorem}
The time evolution of $\dot{x}^{\perp}_p$ is 
\begin{equation*}
\dot{x}^{\perp}_p(t)=-\frac{(\left\vert p\right\vert ^{2}-\kappa
(t)^{4})^{1/2}}{\left\vert p\right\vert }(-\sin \phi (t)D_{0}+\cos \phi
(t)E_{0}).
\end{equation*} 
Hence, the component $t\mapsto x_{0}+x^{\perp}_p(t)$ of the motion of the
elastica in the plane perpendicular to $p$ through $x_{0}=x(t_{0})$ is 
\begin{equation*}
x^{\perp}_p(t)=-\left\vert p\right\vert ^{-2}(x_{0}\times p)\times p-\frac{ 
{\small 1}}{\left\vert p\right\vert }\int_{t_{0}}^{t}(\left\vert
p\right\vert ^{2}-\kappa (s)^{4})^{1/2}(-\sin \phi (s)D_{0}+\cos \phi 
(s)E_{0})\,ds.
\end{equation*}
\end{theorem}

\begin{proof}
\begin{eqnarray*}
\dot{x}^{\perp}_p &=&\frac{\left\vert (\dot{x}\times p)\times
p\right\vert }{\left\vert p\right\vert ^{2}}E \\
&=&\frac{\left\vert p\right\vert \left( \left\vert p\right\vert ^{2}-\kappa
^{4}\right) ^{1/2}}{\left\vert p\right\vert ^{2}}E \\
&=&\frac{\left( \left\vert p\right\vert ^{2}-\kappa ^{4}\right) ^{1/2}}{ 
\left\vert p\right\vert }(-\sin \phi (t)D_{0}+\cos \phi (t)E_{0}).
\end{eqnarray*}
\end{proof}

\begin{corollary}
The elastica equations in the arclength parametrization 
\begin{equation*}
\sigma :I\rightarrow R^{n}:t\mapsto x(t),
\end{equation*} 
have a unique solution 
\begin{equation*}
x(t)=x_{0}+\int_{t_{0}}^{t}\left( - \frac{\kappa(s)^2}{\left\vert p\right\vert
^{2}}p-\frac{\left( \left\vert p\right\vert ^{2}-\kappa (s)^{4}\right) 
^{1/2}}{\left\vert p\right\vert }(-\sin \phi (s)D_{0}+\cos \phi 
(s)E_{0})\right) ds
\end{equation*} 
for initial data in 
\begin{equation*}
M_{0}^{3}=\{(t,x,\dot{x},\ddot{x},\dddot{x})\in M^{3}\mid \kappa \neq
0,~\tau \neq 0\}.
\end{equation*}
\end{corollary}

It remains to consider the special cases when $\kappa \neq 0$ and $\tau =0$, and 
when $\kappa =0$. Equation (\ref{scalar 4}), $\kappa ^{2}\tau =-\frac{ 
{\small 1}}{{\small 4}}\left\langle l,p\right\rangle $, shows that $ 
\left\langle l,p\right\rangle =0$. Hence, if either $\kappa $ or $\tau $
vanishes at some point $t_{0},$ then it vanishes for all $t$ for which the
solution exists.
\begin{enumerate}
  \item If $\tau =0$ and $\kappa \neq 0$, the Frenet equations 
are $\dot{T}=\kappa N$, $\dot{N}=-\kappa T$, $\dot{B} =0$, 
and the conservation of the linear momentum $p$ and the angular momentum $l$ 
are 
\begin{eqnarray}
p & = & -2\dot{\kappa}N-\kappa ^{2}T,  \label{p2} \\
l & = & x\times p+2\kappa B .  \notag
\end{eqnarray} 
Thus, there is an additional conserved quantity, 
\begin{equation*}
B=\frac{1}{\kappa }\dot{x}\times \ddot{x}.
\end{equation*} 
As the scalar product $\left\langle \dot{x},p\right\rangle =-\kappa 
^{2}\left\langle \dot{x},\dot{x}\right\rangle =-\kappa ^{2}$,
\begin{equation*}
\left\langle x(t),p\right\rangle =\left\langle x(t_{0}),p\right\rangle
-\int_{t_{0}}^{t}\kappa ^{2}(s)\,ds.
\end{equation*}
Taking the cross product of equation (\ref{p2}) with $B$ yields 
\begin{equation*}
-2\dot{\kappa}N\times B-\kappa ^{2}T\times B=p\times B.
\end{equation*} 
Since $N\times B=T$ and $T\times B=-N$, 
\begin{equation*}
-2\dot{\kappa}T+\kappa ^{2}N=p\times B.
\end{equation*} 
Therefore, 
\begin{equation*}
\left\langle \dot{x},p\times B\right\rangle =-2\dot{\kappa},
\end{equation*} 
and 
\begin{eqnarray*}
\left\langle x(t),p\times B\right\rangle & = & \left\langle x(t_{0}),p\times
B\right\rangle +\int_{t_{0}}^{t}\left\langle \dot{x}(s),p\times
B\right\rangle ds \\
 & = & \left\langle x(t_{0}),p\times B\right\rangle -2\kappa
(t)+2\kappa (t_{0}).
\end{eqnarray*} 
Thus, if $p\times B\neq 0$, then 
\begin{equation*}
x(t)=x(t_{0})-\frac{2\kappa (t_{0})}{\left\vert p\times B\right\vert ^{2}} 
p\times B-\left( \frac{1}{\left\vert p\right\vert ^{2}}\int_{t_{0}}^{t} 
\kappa ^{2}(t)\,dt\right) p-\frac{2\kappa (t)}{\left\vert
p\times B\right\vert ^{2}}p\times B.
\end{equation*}

\item The special case $p\times B=0,$ $p\neq 0$.
If $p\times B=0$, and $p\neq 0$, then $p$ is parallel to $B$, and equation 
(\ref{p2}) implies that $\kappa =0$, so the solution is a straight line.

\item  If $p=0,$ then equation (\ref{p2}) implies that $\kappa =0.$  If 
$\kappa =0$, then $\dot{x}$ is constant, and the motion is again a straight 
line.
\end{enumerate}

\subsection{Closed elastica}

As mentioned in the introduction, a significant motivatation for this work was 
understanding how symmetry and conservation laws could be systematically 
exploited to integrate the elastica equations.  In the case of closed elastica, 
as studied by Langer and Singer \cite{langer-singer}, it is necessary to add an 
arclength constraint to the variational problem.  This results in studying a 
modified problem with an undetermined Lagrange multiplier.  The modifications 
to our analysis are straightforward insofar as the use of the conserved 
quantities is concerned.  However, since there is also an immediate integration 
and reduction of order in the problem, which results in a loss of manifest 
Euclidean invariance, it seemed preferable to avoid the arclength constrained 
problem and keep the full symmetry in order to see more clearly how the 
conservation laws enabled the integration, which is the route taken in the 
previous section.

In more detail, a slicker, but less transparent approach to the Euler-Lagrange 
equations runs as follows.  For  a second order Lagrangian 
$L(x,\dot{x},\ddot{x},t)$ the 
Euler-Lagrange equations are
\[
 \frac{\partial L}{\partial x}-\frac{d}{dt}\left( \frac{\partial L}{\partial 
\dot{x}}\right) + \frac{d^2}{dt^2}\left( \frac{\partial L}{\partial 
\ddot{x}}\right) = 0,
\]
it follows that if 
\[
 \frac{\partial L}{\partial x} \equiv 0, 
\]
which is the case in the elastica problem, 
\[
 \frac{d}{dt} \left(\frac{\partial L}{\partial \dot{x}}-\frac{d}{dt}\left( 
\frac{\partial L}{\partial 
\ddot{x}}\right)\right) =0,
\]
immediately integrates to 
\[
 \frac{\partial L}{\partial \dot{x}}-\frac{d}{dt}\left( 
\frac{\partial L}{\partial 
\ddot{x}}\right) = c,
\]
where $c$ is a constant.

Now put $q=\dot{x}$, $\dot{q}=\ddot{x}$, and set
\[
l(q,\dot{q},t) := L(x,\dot{x},\ddot{x},t) -c \cdot \dot{x}
\]
Then the integrated equations are now the Euler-Lagrange equations for the 
{\it first order} Lagrangian $l$
\[
 \frac{\partial l}{\partial q}-\frac{d}{dt}\left( 
\frac{\partial l}{\partial 
\dot{q}}\right) = 0.
\]

\subsubsection{Reduction for elastica}

\paragraph{Linear momentum}

The elastica functional for fixed arclength is 
\[
 \int_{\gamma} \kappa^2+ \lambda\,ds.
\]
Here $\lambda$ is a constant whose value is {\it a priori} unknown. 
The curvature $\kappa$ is
\[
 \kappa = \frac{|\dot{x} \times \ddot{x}|}{|\dot{x}|^3} 
\]
so since $ds = |\dot{x}|\,dt$, the reduced Lagrangian is
\[
 l(q,\dot{q}) = \frac{|q \times \dot{q}|^2}{|q|^5} +\lambda |q|- c\cdot q
\]
It remains to compute the Euler-Lagrange equations and look at them in the 
Frenet frame.  The derivatives are
\[
 \frac{\partial l}{\partial q} = \left( 
-5\frac{|q\times\dot{q}|}{|q|^7} + \frac{\lambda}{|q|}\right) q 
+\frac{2}{|q|^5}\dot{q}\times(q\times\dot{q}) - c,
\]
\[
 \frac{\partial l}{\partial \dot{q}} = \frac{2}{|q|^5} q\times(\dot{q}\times q).
\]
Define new variables $T,N,B$ and $v$ by setting $v=|q|$, $T=v^{-1}q$, 
\[
 N= \frac{(q\times\dot{q})\times q}{|q\times\dot{q}|\,|q|} = 
\frac{|q|^2\,\dot{q} 
-\langle q,\dot{q}\rangle \,q}{|q\times\dot{q}|\,|q|}, \qquad B=T\times N.
\]
This implies 
\[
 \dot{q} = \dot{v}T+v^2\kappa N, \qquad \frac{\partial l}{\partial \dot{q}} = 
\frac{2}{v}\kappa N.
\]
If we recall the Frenet equations 
then it follows that
\[
 \frac{\partial l}{\partial q} = (-3\kappa^2+ \lambda ) T 
-2\frac{\kappa\dot{v}}{v^2}N -c,
\]
and 
\[
 \frac{d}{dt}\left(\frac{\partial l}{\partial\dot{q}}\right) = 2\left( 
-\kappa^2 T + \left( \frac{\kappa}{v} \right)^{\cdot} N +\kappa\tau B \right).
\]
This implies that the Euler-Lagrange equations are
\[
 (\lambda-\kappa^2) T -2\frac{\dot{\kappa}}{v} N -2\kappa\tau B = c.
\]
Taking the inner product of this equation with itself and choosing the 
arclength parametrization (so $v=1$) yields
\[
 4(\kappa^{\prime})^2 +(\lambda - \kappa^2)^2 +4\kappa^2\tau^2=c^2.
\]

\paragraph{Angular momentum}

The reduced Lagrangian $l$ is not invariant under the rotation group $\SO(3)$, 
but it is invariant under the $\SO(2)$ subgroup generated by the vector field 
\[
 X = (c\times q) \frac{\partial}{\partial q}
\]
which is rotation about the axis defined by $c\neq 0$.   The associated 
conserved momentum is 
\[
 j=\langle p\,dq, X\rangle = \langle 2v^{-1}\kappa N, c\times q\rangle.
\]
The only nonzero component of this contraction is in the $N$ component of 
$c\times q$, and since the Frenet frame is orthonormal,
\[
 j= 2\kappa \langle c,B\rangle.
 \]
 Taking the inner product of $c$ with the Euler-Lagrange equations gives 
 \[
  \langle c,B\rangle = -2\kappa\tau,
 \]
and this implies that the conserved angular momentum $j$ is
\[
 j= -4\kappa^2\tau.
\]
Thus $4\kappa^2\tau^2 = j^2/4\kappa^2$, and substituting back into the equation 
for $\kappa^{\prime}$ yields
\[
 4(\kappa^{\prime})^2 +(\lambda - \kappa^2)^2 + \frac{j^2}{4\kappa^2} = c^2.
\]
This recovers equations $(3)$ and $(4)$ of {\sc foltinek} \cite{foltinek}, 
together with the interpretation of $c$ as linear momentum and $j$ as angular 
momentum.

\section{Elastica as a constrained Hamiltonian system}

\subsection{Range of the Legendre transformation}

In this section we discuss the range of the Legendre
transformation 
\begin{equation*}
\mathscr{L}:J_{0}^{3}\rightarrow T^{\ast }J_{0}^{1}:(t,x,\dot{x},\ddot{x}, 
\dddot{x})\mapsto (t,x,\dot{x},p_{t},p_{x},p_{\dot{x}})
\end{equation*} 
for elastica with Lagrangian 
\begin{equation}
L(x,\dot{x},\ddot{x})=\frac{\left\langle \ddot{x}^{\perp },\ddot{x}^{\perp
}\right\rangle }{\left\vert \dot{x}\right\vert ^{3}}.  \label{L}
\end{equation}
Here, the subscript ${0}$ denotes that $\dot{x}\neq 0$, and 
\begin{eqnarray}
p_{\dot{x}} &=&\frac{\partial L}{\partial \ddot{x}}=2\frac{\ddot{x}^{\perp } 
}{\left\vert \dot{x}\right\vert ^{3}},  \label{pxdot} \\
p_{x} &=&\frac{\partial L}{\partial \dot{x}}-\frac{d}{dt}\frac{\partial L}{ 
\partial \ddot{x}}=-\frac{2}{\left\vert \dot{x}\right\vert ^{3}}\dddot{x} 
^{\perp }+\frac{6}{\left\vert \dot{x}\right\vert ^{5}}\left\langle \dot{x}, 
\ddot{x}\right\rangle \ddot{x}^{\perp }-\frac{\left\langle \ddot{x}^{\perp }, 
\ddot{x}^{\perp }\right\rangle }{\left\vert \dot{x}\right\vert ^{5}}\dot{x},
\label{px} \\
p_{t} &=&L-\left\langle p_{x},\dot{x}\right\rangle -\left\langle p_{ 
\dot{x}}, \ddot{x}\right\rangle =0.  \label{pt}
\end{eqnarray} 
The $\Diff_{+}\R$-invariance of $L\,dt$ is responsible for the
vanishing of $p_{t}$ above, and it implies that the variable $p_{\dot{x}}$
satisfies the equation 
\begin{equation}
\left\langle p_{\dot{x}},\dot{x}\right\rangle =0  \label{xdotpxdot}
\end{equation} 
(see remark \ref{Second Noether Theorem}.) Note that equations (\ref{pxdot}) and 
(\ref{px}) imply 
\begin{eqnarray}
\left\langle p_{\dot{x}},p_{\dot{x}}\right\rangle &=&4\frac{\left\langle 
\ddot{x}^{\perp },\ddot{x}^{\perp }\right\rangle }{\left\vert \dot{x} 
\right\vert ^{6}}=\frac{4}{\left\vert \dot{x}\right\vert ^{3}}L,  \label{px2}
\\
\left\langle p_{\dot{x}},\ddot{x}\right\rangle &=&\left\langle 2\frac{\ddot{x 
}^{\perp }}{\left\vert \dot{x}\right\vert ^{3}},\ddot{x}\right\rangle =2 
\frac{\left\langle \ddot{x}^{\perp },\ddot{x}\right\rangle }{\left\vert \dot{ 
x}\right\vert ^{3}}=2 
\frac{\left\langle \ddot{x}^{\perp },\ddot{x}^{\perp }\right\rangle }{ 
\left\vert \dot{x}\right\vert ^{3}}=2L=\frac{{\small 1}}{{\small 2}} 
\left\vert \dot{x}\right\vert ^{3}\left\langle p_{\dot{x}},p_{\dot{x} 
}\right\rangle , \label{pxdotxddot} \\
\left\langle p_{x},\dot{x}\right\rangle &=&\left\langle -\frac{2}{\left\vert 
\dot{x}\right\vert ^{3}}\dddot{x}^{\perp }+\frac{6}{\left\vert \dot{x} 
\right\vert ^{5}}\left\langle \dot{x},\ddot{x}\right\rangle \ddot{x}^{\perp
}-\frac{\left\langle \ddot{x}^{\perp },\ddot{x}^{\perp }\right\rangle }{ 
\left\vert \dot{x}\right\vert ^{5}}\dot{x},\dot{x}\right\rangle =-\frac{{\small 
1}}{{\small 4}}\left\vert \dot{x} 
\right\vert ^{3}\left\langle p_{\dot{x}},p_{\dot{x}}\right\rangle .
\label{pxxdot}
\end{eqnarray} 
Equation (\ref{pt}), written in terms of the variables on $T^{\ast
}J_{0}^{1},$ reads 
\begin{equation*}
H(x,\dot{x},p_{x},p_{\dot{x}})=-\frac{{\small 1}}{ 
{\small 4}}\left\vert \dot{x}\right\vert ^{3}\left\langle p_{\dot{x}},p_{ 
\dot{x}}\right\rangle +\frac{{\small 1}}{{\small 2}}\left\vert \dot{x} 
\right\vert ^{3}\left\langle p_{\dot{x}},p_{\dot{x}}\right\rangle -\frac{ 
{\small 1}}{{\small 4}}\left\vert \dot{x}\right\vert ^{3}\left\langle p_{ 
\dot{x}},p_{\dot{x}}\right\rangle =0.
\end{equation*} 
Hence, it does not introduce further restrictions of the variables $(t,x, 
\dot{x},p_{t},p_{x},p_{\dot{x}})$.
Equations (\ref{pxdotxddot}) and (\ref{pxxdot}) lead to the new constraint
equation 
\begin{equation}
\left\langle p_{\dot{x}},\ddot{x}\right\rangle +2\left\langle p_{x},\dot{x} 
\right\rangle =0,  \label{new constraint}
\end{equation} 
while equation (\ref{new constraint}) in terms of the variables on $T^{\ast 
}J_{0}^{1}$ is
\begin{equation}
\left\vert \dot{x}\right\vert ^{3}\left\langle p_{\dot{x}},p_{\dot{x} 
}\right\rangle +4\left\langle p_{x},\dot{x}\right\rangle =0.
\label{new constraint 1}
\end{equation}

\begin{theorem}
The range of the Legendre transformation $\mathscr{L}:J_{0}^{3}\rightarrow
T^{\ast }J_{0}^{1}$ is the common zero set of the three 
functions $p_{t}$, $\left\langle p_{\dot{x}},\dot{x}\right\rangle$ and 
$\left\langle p_{\dot{x}},p_{\dot{x} 
}\right\rangle +\frac{4}{\left\vert \dot{x}\right\vert ^{3}}\left\langle
p_{x},\dot{x}\right\rangle $. That is, 
\begin{equation}
\mathrm{range}~\mathscr{L}=\{(t,x,\dot{x},p_{t},p_{x,}p_{\dot{x}})\in
T^{\ast }J_{0}^{1}\mid p_{t}=\left\langle p_{\dot{x}},\dot{x} 
\right\rangle =\left\vert \dot{x}\right\vert ^{3}\left\langle
p_{\dot{x}},p_{\dot{x}}\right\rangle +4\left\langle p_{x},\dot{x} 
\right\rangle =0\}.  \label{rangeL 1}
\end{equation}
\end{theorem}

\begin{proof}
Equations (\ref{pt}), (\ref{xdotpxdot}) and (\ref{new constraint 1}) imply
that 
\begin{equation*}
\mathrm{range}~\mathscr{L}\subseteq \{(t,x,\dot{x},p_{t},p_{x,}p_{\dot{x} 
})\in T^{\ast }J_{0}^{1}\mid p_{t}=\left\langle p_{\dot{x}},\dot{x} 
\right\rangle =\left\vert \dot{x}\right\vert^{3}\left\langle
p_{\dot{x}},p_{\dot{x}}\right\rangle +4\left\langle p_{x},\dot{x} 
\right\rangle =0\}.
\end{equation*} 
Suppose $(t,x,\dot{x},p_{t},p_{x,}p_{\dot{x}})\in T^{\ast
}J_{0}^{1} $ is such that $p_{t}=0$, $\left\langle p_{\dot{x}},\dot{x} 
\right\rangle =0$ and $\left\vert \dot{x}\right\vert ^{3}\left\langle p_{ 
\dot{x}},p_{\dot{x}}\right\rangle +4\left\langle p_{x},\dot{x}\right\rangle
=0$. Since $\left\langle p_{\dot{x}},\dot{x}\right\rangle =0,$ it follows
that $p_{\dot{x}}=p_{\dot{x}}^{\perp },$ and equation (\ref{pxdot}) implies $ 
\ddot{x}^{\perp }=\frac{1}{2}\left\vert \dot{x}\right\vert ^{3}p_{\dot{x}}.$
By definition, the vanishing of $\left\vert \dot{x}\right\vert
^{3}\left\langle p_{\dot{x}},p_{\dot{x}}\right\rangle +4\left\langle p_{x}, 
\dot{x}\right\rangle $ is equivalent to equation (\ref{new constraint});
that is $\left\langle p_{\dot{x}},\ddot{x}\right\rangle +2\left\langle p_{x}, 
\dot{x}\right\rangle =0$. Splitting equation (\ref{px}) into its components
perpendicular and parallel to $\dot{x}$ gives 
\begin{eqnarray}
p_{x}^{\perp } &=&-\frac{2}{\left\vert \dot{x}\right\vert ^{3}}\dddot{x} 
^{\perp }+\frac{6}{\left\vert \dot{x}\right\vert ^{5}}\left\langle \dot{x}, 
\ddot{x}\right\rangle \ddot{x}^{\perp },  \label{px perp} \\
p_{x}^{\parallel } &=&-\frac{\left\langle \ddot{x}^{\perp },\ddot{x}^{\perp
}\right\rangle }{\left\vert \dot{x}\right\vert ^{5}}\dot{x}.
\label{px parallel}
\end{eqnarray} 
Equation (\ref{px perp}) gives 
\begin{eqnarray*}
\frac{2}{\left\vert \dot{x}\right\vert ^{3}}\dddot{x}^{\perp }
&=&-p_{x}^{\perp }+\frac{6}{\left\vert \dot{x}\right\vert ^{5}}\left\langle 
\dot{x},\ddot{x}\right\rangle \ddot{x}^{\perp }=-p_{x}^{\perp }+\frac{6}{ 
\left\vert \dot{x}\right\vert ^{5}}\left\langle \dot{x},\ddot{x} 
\right\rangle \frac{1}{2}\left\vert \dot{x}\right\vert ^{3}p_{\dot{x}} \\
&=&-p_{x}^{\perp }+\frac{3}{\left\vert \dot{x}\right\vert ^{2}}\left\langle 
\dot{x},\ddot{x}\right\rangle p_{\dot{x}},
\end{eqnarray*} 
where $\left\langle \dot{x},\ddot{x}\right\rangle $ is arbitrary. Equation ( 
\ref{px parallel}) is equivalent to equation (\ref{new constraint}) because 
\begin{equation*}
\left\langle p_{x},\dot{x}\right\rangle =\left\langle p_{x}^{\parallel }, 
\dot{x}\right\rangle =-\frac{\left\langle \ddot{x}^{\perp },\ddot{x}^{\perp
}\right\rangle }{\left\vert \dot{x}\right\vert ^{5}}\left\langle \dot{x}, 
\dot{x}\right\rangle =-\frac{\left\langle \ddot{x}^{\perp },\ddot{x}^{\perp
}\right\rangle }{\left\vert \dot{x}\right\vert ^{3}}=-\frac{1}{2} 
\left\langle p_{\dot{x}},\ddot{x}\right\rangle
\end{equation*} 
by equation (\ref{pxdotxddot}). 

The above argument shows that the fibre of $\mathscr{L}$ over the point $ 
(t,x,\dot{x},p_{t},p_{x,}p_{\dot{x}})\in T^{\ast }J_{0}^{1}$ such that $ 
p_{t}=0,~\left\langle p_{\dot{x}},\dot{x}\right\rangle =0$ and $\left\vert 
\dot{x}\right\vert ^{3}\left\langle p_{\dot{x}},p_{\dot{x}}\right\rangle
+4\left\langle p_{x},\dot{x}\right\rangle =0$ is not empty. In fact, 
\begin{equation*}
\mathscr{L}^{-1}(t,x,\dot{x},p_{t},p_{x,}p_{\dot{x}})=\{(t,x,\dot{x},\ddot{x} 
^{\perp }+\ddot{x}^{\parallel },\dddot{x}^{\perp }+\dddot{x}^{\parallel }\},
\end{equation*} 
where 
\begin{eqnarray*}
\ddot{x}^{\perp } &=&\frac{{\small 1}}{{\small 2}}\left\vert \dot{x} 
\right\vert ^{3}p_{\dot{x}}, \\
\dddot{x}^{\perp } &=&\frac{{\small 1}}{{\small 2}}\left( -\left\vert \dot{x} 
\right\vert ^{3}p_{x}^{\perp }+3\left\vert \dot{x}\right\vert \left\langle 
\dot{x},\ddot{x}^{\parallel }\right\rangle p_{\dot{x}}\right),
\end{eqnarray*} 
and $\ddot{x}^{\parallel }$ and $\dddot{x}^{\parallel }$ are arbitrary.
\end{proof}

\begin{theorem}
\label{range}The range of the Legendre transformation is a submanifold of $ 
T^{\ast }J_{0}^{1}$.
\end{theorem}

\begin{proof}
Fix $(t,x,\dot{x})\in J_{0}^{1}$. The constraint equations 
\begin{eqnarray*}
\left\langle p_{\dot{x}},\dot{x}\right\rangle &=&0, \\
\left\langle \dot{x},\dot{x}\right\rangle ^{3/2}\left\langle p_{\dot{x}},p_{ 
\dot{x}}\right\rangle +4\left\langle p_{x},\dot{x}\right\rangle &=&0,
\end{eqnarray*}
give 
\begin{eqnarray*}
p_{\dot{x}}^{\parallel } &=&0, \\
p_{x}^{\parallel } &=&-\frac{1}{4}\left\langle \dot{x},\dot{x}\right\rangle
^{3/2}\left\langle p_{\dot{x}},p_{\dot{x}}\right\rangle =-\frac{1}{4} 
\left\langle \dot{x},\dot{x}\right\rangle ^{3/2}\left\langle p_{\dot{x} 
}^{\perp },p_{\dot{x}}^{\perp }\right\rangle .
\end{eqnarray*} 
By assumption, $\dot{x}\neq 0$, which implies that the splitting of vectors into 
components parallel and perpendicular to $\dot{x}$ is smooth. 
Therefore, 
\[
\{(t,x,\dot{x},p_{t},p_{x,}p_{\dot{x}})\in T^{\ast }J_{0}^{1}\mid
p_{t}=0,\left\langle p_{\dot{x}},\dot{x}\right\rangle =0\text{ and } 
\left\vert \dot{x}\right\vert^{3}\left\langle p_{\dot{x}},p_{\dot{x} 
}\right\rangle +4\left\langle p_{x},\dot{x}\right\rangle =0\}
\]
is equal to 
\[
\{(t,x,\dot{x},p_{t},p_{x,}p_{\dot{x}})\in T^{\ast }J_{0}^{1}\mid
p_{t}=0,~p_{\dot{x}}^{\parallel }=0\text{ and }p_{x}^{\parallel 
}=-\frac{1}{4}\left\vert \dot{x}\right\vert ^{3}\langle p_{\dot{x}}^{\perp 
},p_{\dot{ 
x}}^{\perp }\rangle \}
\]
and is a submanifold of $T^{\ast }J_{0}^{1}.$ Hence, 
$\mathrm{range}~\mathscr{L}$ is a submanifold of $T^{\ast }J_{0}^{1}$.
\end{proof}
Recall that the Liouville form of $T^{\ast }J_{0}^{1}$ is 
\begin{equation}
\theta =p_{t}\,dt+p_{x}\,dx +p_{\dot{x}} \,d\dot{x}  \label{theta}
\end{equation} 
with exterior derivative 
\begin{equation}
\omega =d\theta =dp_{t}\wedge dt+ dp_{x} \wedge dx
+ dp_{\dot{x}} \wedge d\dot{x}  \label{omega}
\end{equation} 
the canonical symplectic form of $T^{\ast }J_{0}^{1}$.

For $f\in C^{\infty }(T^{\ast }J_{0}^{1})$, the Hamiltonian vector field of $ 
f$ is the unique vector field $X_{f}$ on $T^{\ast }J_{0}^{1}$ such that 
\begin{equation*}
X_{f} 
\lefthook
\omega =-df,
\end{equation*} 
where $ 
\lefthook$ denote the left interior product (contraction). The Poisson bracket of two
functions $f_{1},f_{2}\in C^{\infty }(T^{\ast }J_{0}^{1})$ is given by 
\begin{equation}
\{f_{1},f_{2}\}=X_{f_{2}}(f_{1}).  \label{Poissson bracket}
\end{equation} 
It is bilinear, antisymmetric, and it satisfies the Jacobi identity 

For the sake of future convenience, define the reparametrization-invariant 
function $h$ by
\begin{equation}
h=\frac{\left\vert \dot{x}\right\vert ^{2}}{4}\left\langle p_{\dot{x}},p_{ 
\dot{x}}\right\rangle +\frac{\left\langle p_{x},\dot{x}\right\rangle }{ 
\left\vert \dot{x}\right\vert }.  \label{h}
\end{equation} 
Note that $h$ is smooth, because $\dot{x}\neq 0$, and we can use the
constraint $h=0$ instead of $\left\vert \dot{x}\right\vert ^{3}\left\langle
p_{\dot{x}},p_{\dot{x}}\right\rangle +4\left\langle p_{x},\dot{x} 
\right\rangle =0$ in describing the range of $\mathscr{L}.$ In other \
words, 
\begin{equation}
\mathrm{range}~\mathscr{L}=\{(t,x,\dot{x},p_{t},p_{x,}p_{\dot{x}})\in
T^{\ast }J_{0}^{1}\mid p_{t}=0,~\left\langle p_{\dot{x}},\dot{x} 
\right\rangle =0\text{ and }h=0\}.  \label{rangeL 2}
\end{equation} 
The Hamiltonian vector fields of the constraint functions $p_{t}$, $ 
\left\langle p_{\dot{x}},\dot{x}\right\rangle $ and $h$ are 
\begin{eqnarray*}
X_{p_{t}} &=&\frac{\partial }{\partial t}, \\
X_{\left\langle p_{\dot{x}},\dot{x}\right\rangle } &=&\dot{x}\frac{\partial 
}{\partial \dot{x}}-p_{\dot{x}}\frac{\partial }{\partial p_{\dot{x}}}, \\
X_{h} &=&\frac{1}{2}\left\langle \dot{x},\dot{x}\right\rangle p_{\dot{x}} 
\frac{\partial }{\partial \dot{x}}+\frac{1}{| 
\dot{x}|}\dot{x}\frac{\partial }{\partial 
x}-\frac{1}{|\dot{x}|}p_{x}\frac{\partial }{ 
\partial p_{\dot{x}}}+\left( -\frac{1}{2}\left\langle p_{\dot{x}},p_{\dot{x} 
}\right\rangle +\frac{\left\langle p_{x},\dot{x}\right\rangle 
}{|\dot{x}|^3}\right) \dot{x}\frac{\partial }{\partial
p_{\dot{x}}}
\end{eqnarray*}
Note that all the Poisson brackets of the constraint functions vanish 
identically
\begin{eqnarray}
\{\left\langle p_{\dot{x}},\dot{x}\right\rangle ,p_{t}\} 
=\{h,p_{t}\} =\{h,\left\langle 
p_{\dot{x}},\dot{x}\right\rangle \} =  0. \label{Poisson brackets}
\end{eqnarray} 
This implies that $\mathrm{range}~\mathscr{L}$ is a coisotropic submanifold
of $(T^{\ast }J_{0}^{1},\omega )$.

\subsection{Action of $\Diff_+\R$ on $T^{\ast }J_{0}^{1}$}

Recall that for $X=\tau \partial_t\in \diff_+ \R$, the action of the 
one-parameter subgroup $\exp sX$ on $J_{0}^{1}$
is generated by the vector field $X^{1}=\tau \frac{\partial }{\partial t}- 
\dot{\tau}\dot{x}\frac{\partial }{\partial \dot{x}}$. The lifted action of 
$\exp sX$ on $T^{\ast }J_{0}^{1}$ is generated by the
Hamiltonian vector field $X_{\mathscr{J}_{\tau }}$, where 
\begin{equation*}
\mathscr{J}_{\tau }(t,x,\dot{x},p_{t},p_{x},p_{\dot{x}})=\left\langle
p_{t}dt+p_{x}dx+p_{\dot{x}}d\dot{x},X_{\tau }(t,x,\dot{x})\right\rangle
=\tau (t)p_{t}-\dot{\tau}(t)\left\langle p_{\dot{x}},\dot{x}\right\rangle .
\end{equation*} 
The map 
\begin{equation*}
\mathscr{J}_{\diff}:\text{\textrm{diff}}_{+}\R\rightarrow
T^{\ast }J_{0}^{1}:\tau \frac{{\small \partial }}{\partial t}\mapsto 
\mathscr{J}_{\tau }=\tau (t)p_{t}-\dot{\tau}(t)\left\langle p_{\dot{x}},\dot{ 
x}\right\rangle
\end{equation*} 
may be interpreted as the momentum map for the action of the group 
$\Diff_+\R$ on $T^{\ast }J_{0}^{1}$. Writing it this way avoids
unnecessary discussion about the topology of the dual of the Lie algebra 
\textrm{diff}$_{+}\R$. The constraint equations $p_{t}=0$ and $ 
\left\langle p_{\dot{x}},\dot{x}\right\rangle =0$ imply that $\mathscr{J}_{ 
\diff}$ vanishes on $\mathrm{range~}\mathscr{L}$. In other words, 
\begin{equation*}
\mathrm{range}~\mathscr{L}\subseteq \mathscr{J}_{\diff}^{-1}(0).
\end{equation*}

\begin{proposition}
$\mathscr{J}_{\diff}^{-1}(0)$ is a coisotropic submanifold of $ 
T^{\ast }J_{0}^{1}$. The null distribution of the pullback of $\omega $ to $ 
\mathscr{J}_{\diff}^{-1}(0)$ is spanned by the Hamiltonian vector
fields $X_{p_{t}}$ and $X_{\left\langle p_{\dot{x}},\dot{x}\right\rangle }$.
\end{proposition}

\begin{proof}
This follows from the proof of theorem \ref{range} and
equation (\ref{Poisson brackets}).
\end{proof}

Integral curves of the Hamiltonian vector field $X_{p_{t}}=\frac{\partial }{ 
\partial t}$ are lines parallel to the $t$-axis. Integral curves of $ 
X_{\left\langle p_{\dot{x}},\dot{x}\right\rangle }$ satisfy equations 
\begin{eqnarray*}
\frac{d}{ds}\dot{x}(s) &=&\dot{x}(s)\text{,} \\
\frac{d}{ds}p_{\dot{x}}(s) &=&-p_{\dot{x}}(s).
\end{eqnarray*} 
Hence, for each $\boldsymbol{p}=(t,x,\dot{x},p_{t},p_{x},p_{\dot{x}\,})\in
T^{\ast }J_{0}^{1}$, the integral manifold of the distribution on $T^{\ast
}J_{0}^{1}$ spanned by $X_{p_{t}}$ and $X_{\left\langle p_{\dot{x}},\dot{x} 
\right\rangle }$ that passes through $\boldsymbol{p}$ is 
\begin{equation}
O_{\boldsymbol{p}}=\{(u,x,e^{s}\dot{x},p_{t},p_{x},e^{-s}p_{\dot{x}})\mid
(u,s)\in \R^{2}\}.  \label{O0}
\end{equation}

\begin{theorem}
\label{orbits}For each $\boldsymbol{p}\in \mathscr{J}_{\diff 
}^{-1}(0)\subset T^{\ast }J_{0}^{1}$, the orbit of
the Lie algebra $\diff_{+}\R$ through $\boldsymbol{p}$ and of the reparametrization
group $\Diff_{+}\R$ coincides with the integral manifold $O_{ 
\boldsymbol{p}}$ given by equation (\ref{O0}), where $p_{t}=0.$
\end{theorem}

\begin{proof}
Orbits of the action of the Lie algebra $\diff_{+}\R$ on $ 
T^{\ast }J_{0}^{1}$ are orbits (accessible sets) of the family $\{X_{ 
\mathscr{J}_{\tau }}\mid \tau \frac{\partial }{\partial t}\in \diff 
_{+}\R\}$ of Hamiltonian vector fields on $T^{\ast }J_{0}^{1}$.
Since $X_{\mathscr{J}_{\tau }}=X_{\tau p_{t}}-X_{\dot{\tau}\left\langle p_{ 
\dot{x}},\dot{x}\right\rangle }$ and $\mathscr{J}_{\tau }$ vanishes on $ 
\mathscr{J}_{\diff}^{-1}(0)$, it follows that the restriction of $X_{ 
\mathscr{J}_{\tau }}$ to $\mathscr{J}_{\diff}^{-1}(0)$ is 
\begin{equation*}
X_{\mathscr{J}_{\tau }\mid \mathscr{J}_{\diff}^{-1}(0)}=\tau
X_{p_{t}\mid \mathscr{J}_{\diff}^{-1}(0)}-\dot{\tau}X_{\left\langle
p_{\dot{x}},\dot{x}\right\rangle \mid \mathscr{J}_{\diff}^{-1}(0)} 
\text{.}
\end{equation*} 
Therefore, $X_{\mathscr{J}_{\tau }\mid \mathscr{J}_{\diff}^{-1}(0)}$
is contained in the distribution spanned by $X_{p_{t}\mid \mathscr{J}_{ 
\diff}^{-1}(0)}$ and $X_{\left\langle p_{\dot{x}},\dot{x} 
\right\rangle \mid \mathscr{J}_{\diff}^{-1}(0)}$. Hence, for each $ 
\boldsymbol{p}\in \mathscr{J}_{\diff}^{-1}(0)$, the orbit of $\diff_{+} 
\R$ through $\boldsymbol{p}$ coincides with the integral manifold $O_{\boldsymbol{p}}$ given
by equation (\ref{O0}).

The reparametrization group $\Diff_{+}\R$ acts on $J_{0}^{1}$
by
\begin{equation*}
\Diff_{+}\R\times J_{0}^{1}\rightarrow J_{0}^{1}:(\varphi
,(t,x,\dot{x}))\mapsto \left( \varphi (t),x,\frac{\dot{x}}{\dot{\varphi}(t)} 
\right) ,
\end{equation*} 
where $\varphi $ is a smooth function on $\R$ such that $\dot{\varphi 
}(t)>0$ for $t\in \R$.\ The lift of this action to $T^{\ast }J_{0}^{1}$
is 
\begin{equation*}
\Diff_{+}\R\times T^{\ast }J_{0}^{1}\rightarrow T^{\ast
}J_{0}^{1}:(\varphi ,\boldsymbol{p})\mapsto 
\left(
\varphi (t),x,\frac{\dot{x}}{\dot{\varphi}(t)},\frac{p_{t}}{\dot{\varphi}(t)} 
+\frac{\left\langle p_{\dot{x}},\dot{x}\right\rangle \ddot{\varphi}(t)}{\dot{ 
\varphi}(t)^{2}},p_{x},\dot{\varphi}(t)p_{\dot{x}}\right) .
\end{equation*} 
It preserves $\mathscr{J}_{\diff}^{-1}(0)$. Hence, the orbit of $ 
\Diff_{+}\R$ is 
\begin{equation*}
\{(\varphi (t),x,\dot{\varphi}^{-1}(t)\dot{x},\dot{\varphi} 
(t)^{-1}p_{t},p_{x},\dot{\varphi}(t)p_{\dot{x}})\in T^{\ast }J_{0}^{1}\mid
\varphi \in C^{\infty }(R)\text{, }\dot{\varphi}(t)>0\}.
\end{equation*} 
The action of $\Diff_{+}\R$ preserves $\mathscr{J}_{\mathrm{ 
diff}}^{-1}(0)$, given by $p_{t}=0$ and $\left\langle p_{\dot{x}},\dot{x} 
\right\rangle =0$. Hence,\ orbits of $\Diff_{+}\R$ contained
in $\mathscr{J}_{\diff}^{-1}(0)$ are 
\begin{equation*}
\{(\varphi (t),x,\dot{\varphi}^{-1}(t)\dot{x},0,p_{x},\dot{\varphi}(t)p_{ 
\dot{x}})\in T^{\ast }J_{0}^{1}\mid \left\langle p_{\dot{x}},\dot{x} 
\right\rangle =0\text{, }\varphi \in C^{\infty }(R)\text{, }\dot{\varphi} 
(t)>0\}.
\end{equation*} 
For each $t$, $\varphi (t)=u$ and $\dot{\varphi}(t)=-s$, are independent.
Therefore, orbits of $\Diff_{+}\R$\ contained in $\mathscr{J} 
_{\diff}^{-1}(0)$ coincide with the corresponding integral manifolds
given by equation (\ref{O0}).
\end{proof}

\subsection{Reduction of Diff$_{+}\R$ symmetries}

In this section, we discuss the space 
\begin{equation*}
R=\mathscr{J}_{\diff}^{-1}(0)/\Diff_{+}\R
\end{equation*} 
of $\Diff_{+}\R$-orbits in $\mathscr{J}_{\diff 
}^{-1}(0)$. According to Theorem \ref{orbits}, the reduced phase space $R$
is the space of integral manifolds in $\mathscr{J}_{\diff}^{-1}(0)$
of the distribution spanned by $X_{p_{t}}$ and $X_{\left\langle p_{\dot{x}}, 
\dot{x}\right\rangle }$. We have shown that the orbit of the vector fields $ 
\{X_{p_{t}},X_{\left\langle p_{\dot{x}},\dot{x}\right\rangle }\}$ through $ 
\boldsymbol{p}=(t,x,\dot{x},p_{t},p_{x},p_{\dot{x}})\in \mathscr{J}_{\mathrm{ 
diff}}^{-1}(0)$ is 
\begin{equation}
O_{\boldsymbol{p}}=\{(u,x,e^{s}\dot{x},p_{t},p_{x},e^{-s}p_{\dot{x}})\mid
(u,s)\in \R^{2}\}.  \label{orbit}
\end{equation}

We are going to show $R$ is a quotient manifold of $\mathscr{J}_{\mathrm{diff 
}}^{-1}(0)$, which will imply that $R$ has a unique symplectic form $ 
\omega _{R}$ such that 
\begin{equation}
\rho ^{\ast }\omega _{R}=\iota ^{\ast }\omega ,  \label{omegaR}
\end{equation} 
where $\iota :\mathscr{J}_{\diff}^{-1}(0)\rightarrow T^{\ast
}J_{0}^{1}$ is the inclusion map.

In order to parametrize the reduced phase space $R$, define spherical
 coordinates $(\dot{r},$ $\dot{\alpha},\dot{\beta})$ by 
\begin{eqnarray}
\dot{x}^{1} &=&\dot{r}\sin \dot{\beta}\cos \dot{\alpha},
\label{polar coordinates} \\
\dot{x}^{2} &=&\dot{r}\sin \dot{\beta}\sin \dot{\alpha},  \notag \\
\dot{x}^{3} &=&\dot{r}\cos \dot{\beta},  \notag
\end{eqnarray} 
together with the dual momentum variables 
$(p_{\dot{r}},p_{\dot{\beta}},p_{\dot{\alpha} 
})$ defined by 
\begin{equation}
p_{\dot{x}} d\dot{x} =p_{\dot{r}}d\dot{r}+p_{\dot{ 
\beta}}d\dot{\beta}+p_{\dot{\alpha}}d\dot{\alpha}.  \label{polar momenta}
\end{equation}

\begin{proposition}
$p_{\dot{r}}=\left\langle p_{\dot{x}},\dot{x}\right\rangle /\dot{r}$.
\end{proposition}

\begin{proof}
This is a simple verification.
\end{proof}

Denote by $\rho :\mathscr{J}_{\diff}^{-1}(0)\rightarrow R$ the
reduction map associating to each point in $\mathscr{J}_{\diff 
}^{-1}(0)$ the orbit of $\{X_{p_{t}},X_{\left\langle p_{\dot{x}},\dot{x} 
\right\rangle }\}$ through that point.$\,$\ Let 
\begin{equation*}
S=\{(x,\dot{x})\in T\R^{3}\mid \left\vert \dot{x}\right\vert =1\}
\end{equation*} 
be the unit sphere bundle over $\R^{3}$ parametrized by coordinates $ 
(x,\dot{\alpha},\dot{\beta}).$ The Liouville form of $T^{\ast }S$ is 
\begin{equation}
\theta _{S}= p_{x} dx +p_{\dot{\beta}}d\dot{\beta} 
+p_{\dot{\alpha}}d\dot{\alpha},  \label{thetaS}
\end{equation} 
and 
\begin{equation*}
\omega _{S}=d\theta _{S}
\end{equation*} 
is the canonical symplectic form of $T^{\ast }S.$

\begin{proposition}
\label{symplectomorphism}There is a unique symplectomorphism $\varkappa
:(R,\omega _{R})\rightarrow (T^{\ast }S,\omega _{S})$ such that 
\begin{equation*}
\varkappa \circ \rho :\mathscr{J}_{\diff}^{-1}(0)\rightarrow T^{\ast }S:(t,x,\dot{x} 
,0,p_{x},p_{\dot{x}})\mapsto (x,\dot{\beta},\dot{\alpha},p_{x},p_{\dot{\beta}},p_{ 
\dot{\alpha}}),
\end{equation*} 
where the $(\dot{\beta},\dot{\alpha},p_{\dot{\beta}},p_{\dot{\alpha}})$ are
related to $(x,\dot{x},p_{x},p_{\dot{x}})$ by equations (\ref{polar
coordinates}) and (\ref{polar momenta}).
\end{proposition}

\begin{proof}
Consider first the space $R_{1}=p_{t}^{-1}(0)/X_{p_{t}}$ of integral curves
of $X_{p_{t}}$ in $p_{t}^{-1}(0)$. It is a quotient manifold of $p_{t}^{-1}(0)$
with projection map 
\begin{equation*}
\rho _{1}:p_{t}^{-1}(0)\rightarrow R_{1}:(t,x,\dot{x},0,p_{x},p_{\dot{x} 
})\mapsto (x,\dot{x},p_{x},p_{\dot{x}}).
\end{equation*} 
Moreover, it is a symplectic manifold with the symplectic form 
\begin{equation*}
\omega _{1}=dp_{x}\wedge dx+dp_{\dot{x}}\wedge d\dot{x}.
\end{equation*} 
The constraint function $\left\langle p_{\dot{x}},\dot{x}\right\rangle $ is
left invariant by the action $X_{p_{t}}$, and  pushes forward to a
function on $R_{1}$, denoted by $\left\langle p_{\dot{x}},\dot{x} 
\right\rangle _{1}$. That is, $\left\langle p_{\dot{x}},\dot{x}\right\rangle
=\rho _{1}^{\ast }\left\langle p_{\dot{x}},\dot{x}\right\rangle _{1}.$
Moreover, the Hamiltonian vector field $X_{\left\langle p_{\dot{x}},\dot{x} 
\right\rangle }$ restricted to $p_{t}^{-1}(0)$ pushes forward to the
Hamiltonian vector field on $R_{1}$ corresponding to the function $ 
\left\langle p_{\dot{x}},\dot{x}\right\rangle_1 $ on $R_{1}$. Denote
this vector field by $X_{\left\langle p_{\dot{x}},\dot{x}\right\rangle _{1}}$.

By definition, $\dot{r}=\left\vert \dot{x}\right\vert \neq 0$ on $T^{\ast
}J_{0}^{1}$. Since $p_{\dot{r}}\dot{r}=\left\langle p_{\dot{x}},\dot{x} 
\right\rangle $, it follows that on $\rho _{1}\left(\mathscr{J}_{\diff 
}^{-1}(0)\right)$ the Hamiltonian vector field of $\left\langle 
p_{\dot{x}},\dot{x} 
\right\rangle _{1}$ is proportional to the Hamiltonian vector field $ 
X_{p_{r}}=\frac{\partial }{\partial \dot{r}}$. Therefore, the space $ 
R_{2}=\left\langle p_{\dot{x}},\dot{x}\right\rangle
_{1}^{-1}(0)/X_{\left\langle p_{\dot{x}},\dot{x}\right\rangle _{1}}$ of
orbits of $X_{\left\langle p_{\dot{x}},\dot{x}\right\rangle _{1}}$ in $ 
\left\langle p_{\dot{x}},\dot{x}\right\rangle _{1}^{-1}(0)$ can be
parametrized by $(x,\dot{\beta},\dot{\alpha},p_{x},p_{\dot{\beta}},p_{\dot{ 
\alpha}})$. It is a symplectic manifold with the symplectic form 
\begin{equation*}
\omega _{2}=dp_{x}\wedge dx+dp_{\dot{\beta}}\wedge d\dot{\beta}+dp_{\dot{ 
\alpha}}\wedge d\dot{\alpha}.
\end{equation*}

The coordinates $(x,\dot{\beta},\dot{\alpha},p_{x},p_{\dot{\beta}},p_{\dot{ 
\alpha}})$ define a symplectomorphism between $(R_{2},\omega _{2})$ and $ 
(T^{\ast }S,\omega _{S})$, where $\omega _{s}$ is the pullback to $T^{\ast
}S$ of the canonical symplectic form on $T^{\ast}(T\R^{3})$. However, 
$R=\mathscr{J}_{\diff}^{-1}(0)/\{X_{p_{t}},X_{\left 
\langle p_{\dot{x}},\dot{x}\right\rangle }\}$ with the symplectic form $ 
\omega _{R}$ is naturally symplectomorphic to $(R_{2},\omega _{2})$. Hence, $ 
(R,\omega _{R})$ is symplectomorphic to $(T^{\ast }S,\omega _{S}).$
\end{proof}

It follows from Proposition \ref{symplectomorphism} that we may identify $ 
(R,\omega _{R})$ with $(T^{\ast }S,\omega _{S})$.

The action of the Euclidean group $\SE(3)$ on $\R^{3}$ induces a Hamiltonian 
action of $E$ on $T^{\ast }J_{0}^{1}$ generated by the Hamiltonian vector 
fields $X_{p_{x}}$, 
$X_{p_{\dot{\beta}}}$ and $X_{p_{\dot{\alpha} 
}}$. This action preserves the constraint functions $p_{t},$ $\left\langle
p_{\dot{x}},\dot{x}\right\rangle $ and $h$. In particular, it induces an
action of $\SE(3)$ on the zero level set $\mathscr{J}_{\diff}^{-1}(0)$ of
the momentum map for the action of $\diff_{+}\R$. On the
other hand, the action of $E$ on $\R^{3}$ induces a Hamiltonian action of 
$\SE(3)$ on $T^{\ast }S$, presented as 
\begin{equation*}
T^{\ast }S=\{\,(x,\dot{x},p_{x},p_{\dot{x}})\in T^{\ast }(T\R^{3})\mid 
\left\vert \dot{x}\right\vert =1,~p_{\dot{r}}=0\,\},
\end{equation*} 
which is generated by the Hamiltonian vector fields of $p_{x},~p_{\dot{\beta} 
}$, and $p_{\dot{\alpha}}$ considered as functions on $T^{\ast }S$.
Moreover, these actions of $\SE(3)$ are intertwined by the reduction map $\rho 
: 
\mathscr{J}_{\diff}^{-1}(0)\rightarrow R$ followed by the
identification $R\cong T^{\ast }S$.

\subsection{Hamiltonian dynamics}

The range of the Legendre transformation is characterized as 
\begin{equation*}
\mathrm{range}~\mathscr{L}=\mathscr{J}_{\diff}^{-1}(0)\cap h^{-1}(0),
\end{equation*} 
where
\begin{equation*}
h=\frac{\left\vert \dot{x}\right\vert ^{2}}{4}\left\langle p_{\dot{x}},p_{ 
\dot{x}}\right\rangle +\frac{\left\langle p_{x},\dot{x}\right\rangle }{ 
\left\vert \dot{x}\right\vert }.
\end{equation*} 
The Hamiltonian vector field of $h$ is 
\begin{equation*}
X_{h}=\frac{1}{2}\left\langle \dot{x},\dot{x}\right\rangle p_{\dot{x}}\frac{ 
\partial }{\partial \dot{x}}+\frac{1}{\left\langle \dot{x},\dot{x} 
\right\rangle ^{1/2}}\dot{x}\frac{\partial }{\partial x}-\frac{1}{ 
\left\langle \dot{x},\dot{x}\right\rangle ^{1/2}}p_{x}\frac{\partial }{ 
\partial p_{\dot{x}}}+\left( -\frac{1}{2}\left\langle p_{\dot{x}},p_{\dot{x} 
}\right\rangle +\frac{\left\langle p_{x},\dot{x}\right\rangle }{\left\langle 
\dot{x},\dot{x}\right\rangle ^{3/2}}\right) \dot{x}\frac{\partial }{\partial
p_{\dot{x}}}.
\end{equation*} 
In order to find the integral curves of $X_{h}$ on the range of 
$\mathscr{L}$,
observe that they satisfy the equations 
\begin{eqnarray*}
\frac{d}{ds}x &=&\frac{1}{\left\langle \dot{x},\dot{x}\right\rangle ^{1/2}} 
\dot{x}, \\
\frac{d}{ds}\dot{x} &=&\frac{1}{2}\left\langle \dot{x},\dot{x}\right\rangle
p_{\dot{x}}, \\
\frac{d}{ds}p_{\dot{x}} &=&-\frac{1}{\left\langle \dot{x},\dot{x} 
\right\rangle ^{1/2}}p_{x}+\left( -\frac{1}{2}\left\langle p_{\dot{x}},p_{ 
\dot{x}}\right\rangle +\frac{\left\langle p_{x},\dot{x}\right\rangle }{ 
\left\langle \dot{x},\dot{x}\right\rangle ^{3/2}}\right) \dot{x}, \\
\frac{d}{ds}p_{x} &=&0, \\
\frac{d}{ds}t &=&0.
\end{eqnarray*} 
Multiplying equation $\frac{d}{ds}\dot{x}=\frac{1}{2}\left\langle \dot{x}, 
\dot{x}\right\rangle p_{\dot{x}}$ by $\dot{x}$, and using the constraint
equation $\left\langle p_{\dot{x}},\dot{x}\right\rangle =0$, yields $\frac{d 
}{ds}\left\langle \dot{x},\dot{x}\right\rangle =0$, hence $\left\langle \dot{ 
x}(s),\dot{x}(s)\right\rangle =\left\vert \dot{x}_{0}\right\vert ^{2}$. 
Since $h$ is reparametrization invariant, without loss of generality, we may
assume that $\left\vert \dot{x}_{0}\right\vert =1$, that is, the
arclength parametrization. Moreover, on the range of $\mathscr{L}$, $h=0$,
\ which implies that $\frac{1}{2}\left\langle p_{\dot{x}},p_{\dot{x} 
}\right\rangle =-2\left\langle p_{x},\dot{x}\right\rangle $. This leads to 
\begin{eqnarray}
\frac{d}{ds}x &=&\dot{x},  \label{x} \\
\frac{d}{ds}\dot{x} &=&\frac{1}{2}p_{\dot{x}},  \label{xdot} \\
\frac{d}{ds}p_{\dot{x}} &=&-p_{x}+3\left\langle p_{x},\dot{x}\right\rangle 
\dot{x},  \label{dpxdot} \\
\frac{d}{ds}p_{x} &=&0,  \label{dpx} \\
\frac{d}{ds}t &=&0.  \label{dt}
\end{eqnarray}
Equation (\ref{dpx}) implies that $p_{x}$ is constant. The angular momentum
is given by $l=x\times p_{x}+\dot{x}\times p_{\dot{x}}$ (see equation (\ref 
{l2}.)  Hence, 
\begin{eqnarray*}
\frac{d}{ds}l &=&\left( \frac{d}{ds}x\right) \times p_{x}+x\times \left( 
\frac{d}{ds}p_{x}\right) +\left( \frac{d}{ds}\dot{x}\right) \times p_{\dot{x} 
}+\dot{x}\times \left( \frac{d}{ds}p_{\dot{x}}\right) \\
&=&\dot{x}\times p_{x}+\frac{1}{2}p_{\dot{x}}\times p_{\dot{x}}+\dot{x} 
\times \left( -p_{x}+3\left\langle p_{x},\dot{x}\right\rangle \dot{x}\right)
\\
&=&0,
\end{eqnarray*} 
and so is conserved. Therefore, $\left\langle p_x,l\right\rangle
=\left\langle p_{x},\dot{x}\times p_{\dot{x}}\right\rangle $ is also
conserved.
Multiplying equations (\ref{xdot}) and (\ref{dpxdot}) by $ 
p_{x}$ yields  
\begin{eqnarray*}
\frac{d}{ds}\left\langle p_{x},\dot{x}\right\rangle &=&\left\langle p_{x},p_{ 
\dot{x}}\right\rangle , \\
\frac{d}{ds}\left\langle p_{x},p_{\dot{x}}\right\rangle &=&-\left\langle
p_{x},p_{x}\right\rangle +3\left\langle p_{x},\dot{x}\right\rangle ^{2},
\end{eqnarray*} 
or 
\begin{equation*}
\frac{d^{2}}{ds^{2}}\left\langle p_{x},\dot{x}\right\rangle =-\left\langle
p_{x},p_{x}\right\rangle +3\left\langle p_{x},\dot{x}\right\rangle ^{2}.
\end{equation*} 
Multiplying by $\frac{d}{ds}\left\langle p_{x},\dot{x}\right\rangle $ and
integrating gives
\begin{equation*}
\frac{1}{2}\left( \frac{d\left\langle p_{x},\dot{x}\right\rangle }{ds} 
\right) ^{2}=-\left\langle p_{x},p_{x}\right\rangle \left\langle p_{x},\dot{x 
}\right\rangle +\left\langle p_{x},\dot{x}\right\rangle ^{3}+\text{\textrm{ 
constant}},
\end{equation*} 
which can be integrated since it is separable. If $p_{x}\neq 0$, then this 
equation gives the
component 
\begin{equation*}
\dot{x}^{\parallel }=\frac{\left\langle p_{x},\dot{x}\right\rangle }{ 
\left\vert p_{x}\right\vert ^{2}}p_{x}
\end{equation*} 
of $\dot{x}$ parallel to $p_{x}$. Integrating $\dot{x}^{\parallel
}(s)$ yields the component of the motion in the direction of $p_{x}$.
Returning to equations (\ref{xdot}) and (\ref{dpxdot}) gives
\begin{eqnarray}
\frac{d}{ds}\dot{x} &=&\frac{1}{2}p_{\dot{x}},  \label{xdot1} \\
\frac{d}{ds}p_{\dot{x}} &=&-p_{x}+3\left\langle p_{x},\dot{x}\right\rangle 
\dot{x},  \label{dpxdot1}
\end{eqnarray} 
where $\left\langle p_{x},\dot{x}\right\rangle $ is assumed  known from
the discussion above. Hence, 
\begin{equation}
\frac{d^{2}}{ds^{2}}\dot{x}=-\frac{1}{2}p_{x}+\frac{3}{2}\left\langle p_{x}, 
\dot{x}\right\rangle \dot{x}.  \label{d2xdot}
\end{equation} 
Writing $\dot{x}$ and $p_{x}$ in terms of their components $\dot{x}^{i}$ and 
$p_{x^i},$ 
\begin{equation}
\frac{d^{2}}{ds^{2}}\dot{x}^{i}=-\frac{1}{2}p_{x^i}+\frac{3}{2}\left\langle
p_{x},\dot{x}\right\rangle \dot{x}^{i}.  \label{d2xdoti}
\end{equation} 
Division by $\dot{x}_{i}$ implies
\begin{equation*}
\frac{d}{ds}\left( \ln \left\vert \frac{d}{ds}\dot{x}^{i}\right\vert \right)
=\frac{1}{\frac{d\dot{x}^{i}}{ds}}\frac{d^{2}\dot{x}^{i}}{ds^{2}}=-\frac{ 
p_{x^i}}{2\frac{d\dot{x}^{i}}{ds}}+\frac{3}{2}\left\langle p_{x},\dot{x} 
\right\rangle ,
\end{equation*} 
which implies 
\begin{equation*}
\ln \left\vert \frac{d}{ds}\dot{x}^{i}\right\vert =-\frac{p_{x^i}}{2}\ln
\left\vert \dot{x}^{i}\right\vert +\frac{3}{2}\int \left\langle p_{x},\dot{x} 
\right\rangle (s)\,ds
\end{equation*} 
or 
\begin{equation*}
\ln \left\vert \dot{x}^{p_{x^i}/2}\frac{d}{ds}\dot{x}^{i}\right\vert =\frac{3 
}{2}\int \left\langle p_{x},\dot{x}\right\rangle (s)\,ds,
\end{equation*} 
so that 
\begin{equation*}
(\dot{x}^{i})^{p_{x^i}/2}\frac{d}{ds}\dot{x}^{i}=c\exp \left( \int
f(s)ds\right) ,
\end{equation*} 
where $c$ is a constant dependent on the initial data. Integrating once 
more yields  
\begin{equation*}
\frac{1}{c+1}(\dot{x}^{i})^{1+p_{xi}/2}=c\int \exp \left( \int 
f(s)\, ds\right)\, ds
\end{equation*} 
if $p_{x^i}/2\neq -1$, and 
\begin{equation*}
\ln \left\vert \dot{x}^{i}\right\vert =\int \exp \left( \int f(s)\,ds\right) 
\, ds
\end{equation*} 
if $p_{xi}/2=-1.$

\begin{remark}
The Hamiltonian vector field of $p_{t}+h$ is 
\begin{equation*}
X_{p_{t}+h}=X_{p_{t}}+X_{h}=\frac{\partial }{\partial t}+\frac{1}{2}p_{\dot{x 
}}\frac{\partial }{\partial 
\dot{x}}+\frac{1}{|\dot{x}|^3}\dot{x}\frac{\partial }{\partial x}-\frac{1}{ 
|\dot{x}|^3}p_{x}\frac{\partial }{ 
\partial p_{\dot{x}}}+3\frac{\left\langle p_{x},\dot{x}\right\rangle }{ 
| \dot{x}|^5}\dot{x}\frac{\partial }{ 
\partial p_{\dot{x}}}.
\end{equation*} 
In the arclength parametrization, it is 
\begin{eqnarray*}
\frac{d}{ds}t &=&1, \\
\frac{d}{ds}x &=&\dot{x}, \\
\frac{d}{ds}\dot{x} &=&\frac{1}{2}p_{\dot{x}}, \\
\frac{d}{ds}p_{\dot{x}} &=&-p_{x}+3\left\langle p_{x},\dot{x}\right\rangle 
\dot{x}, \\
\frac{d}{ds}p_{x} &=&0.
\end{eqnarray*} 
Hence, $t=t_{0}+s$, and the solutions of this system can be obtained from
the solutions for integral curves of $X_{h}$ by replacing $s$ by $t-t_{0}.$
\end{remark}

\begin{theorem}
If $(t_{0},x_{0},\dot{x}_{0},0,p_{x_0},p_{\dot{x}_0})=\mathscr{L}(t_{0},x_{0}, 
\dot{x}_{0},\ddot{x}_{0},\dddot{x}_{0})$, the solution of the Euler-Lagrange
equation with initial data $(t_{0},x_{0},\dot{x}_{0},\ddot{x}_{0},\dddot{x} 
_{0})$ is equivalent to the integral curve of $ 
X_{p_{t}+h}$ through $(t_{0},x_{0},\dot{x}_{0},0,p_{x_0},p_{\dot{x}_0})$ given
above.
\end{theorem}

\begin{proof}
The Euler-Lagrange equations for a second order Lagrangian are
\begin{equation*}
\frac{\partial L}{\partial x}-\frac{d}{dt}\frac{\partial L}{\partial \dot{x}} 
+\frac{d^{2}}{dt^{2}}\frac{\partial L}{\partial \ddot{x}}=0.
\end{equation*} 
Since $p_{\dot{x}}=\frac{\partial L}{\partial \ddot{x}}$ and $p_{x}=\frac{ 
\partial L}{\partial \dot{x}}-\frac{d}{dt}p_{\dot{x}},$ they are  
$\frac{\partial L}{\partial x}-\frac{d}{dt}p_{x}=0.$ For
elastica, 
\begin{equation*}
L(x,\dot{x},\ddot{x})=\frac{\left\vert \ddot{x}\right\vert ^{2}}{\left\vert 
\dot{x}\right\vert ^{3}}-\frac{\left\langle \dot{x}, \ddot{x} 
\right\rangle ^{2}}{\left\vert \dot{x}\right\vert ^{5}},
\end{equation*} 
$\frac{\partial L}{\partial x}=0$, and the Euler-Lagrange equations reduce to 
\begin{equation*}
\frac{d}{dt}p_{x}=0.
\end{equation*} 
Since $p_{x}=\frac{\partial L}{\partial \dot{x}}-\frac{d}{dt}p_{\dot{x}},$
the evolution of $p_{\dot{x}}$ is given by 
\begin{equation*}
\frac{d}{dt}p_{\dot{x}}=-p_{x}+\frac{\partial L}{\partial \dot{x}},
\end{equation*} 
where 
\begin{equation}
p_{\dot{x}}=\frac{\partial L}{\partial \ddot{x}}=2\frac{\ddot{x}}{\left\vert 
\dot{x}\right\vert ^{3}}-2\frac{\left\langle \dot{x},\ddot{x}\right\rangle 
\dot{x}}{\left\vert \dot{x}\right\vert ^{5}},  \notag
\end{equation} 
and 
\begin{equation*}
\frac{\partial L}{\partial \dot{x}}=\frac{\partial }{\partial \dot{x}}\left( 
\frac{\left\vert \ddot{x}\right\vert ^{2}}{| \dot{x}|^3}-\frac{\left\langle 
\dot{x}, \ddot{x}\right\rangle
^{2}}{| \dot{x}|^5}\right) =-3\frac{ 
\left\vert \ddot{x}\right\vert 
^{2}}{|\dot{x}|^5}\dot{x}-2\frac{\left\langle \dot{x}, \ddot{x} 
\right\rangle }{| \dot{x}|^5}\ddot{x} 
+5\frac{\left\langle \dot{x}, \ddot{x}\right\rangle 
^{2}}{|\dot{x}|^7}\dot{x}.
\end{equation*} 
In the arclength parametrization $| \dot{x}| =1$, $\left\langle 
\dot{x}, \ddot{x}\right\rangle =0$, $p_{\dot{x}}=2\ddot{x}$ and 
\begin{equation*}
\frac{\partial L}{\partial \dot{x}}=-3\left\vert \ddot{x}\right\vert ^{2} 
\dot{x}=-\frac{3}{4}\left\vert p_{\dot{x}}\right\vert ^{2}\dot{x} 
=3\left\langle p_{x},\dot{x}\right\rangle \dot{x}
\end{equation*} 
because 
\begin{equation*}
h=\frac{1}{4}\left\langle p_{\dot{x}},p_{\dot{x}}\right\rangle +\left\langle
p_{x},\dot{x}\right\rangle =0
\end{equation*} 
on the range of $\mathscr{L}$. Thus, the Euler-Lagrange equations of
elastica in the arclength parametrization are equivalent to 
\begin{eqnarray*}
\frac{d}{dt}p_{x} &=&0, \\
\frac{d}{dt}p_{\dot{x}} &=&-p_{x}+3\left\langle p_{x},\dot{x}\right\rangle 
\dot{x}, \\
\frac{d}{dt}\dot{x} &=&\frac{1}{2}p_{\dot{x}}, \\
\frac{d}{dt}x &=&\dot{x}.
\end{eqnarray*} 
This system of equations, together with the substitution $t=t_{0}+s$, 
leads to the
equation for integral curves of the Hamiltonian vector field of $X_{p_{t}+h}$
given in the remark above.
\end{proof}

\section{Quantization}

There are two ways of quantization of a system with constraints: the
Bleuler-Gupta quantization of the extended phase space followed by the
quantum reduction \cite{bleuler}, \cite{gupta}, and Dirac's classical
reduction followed by quantization of the reduced phase space studied by
Dirac in \cite{dirac63}. Here, we follow the Bleuler-Gupta approach.

The extended phase space $T^{\ast }J_{0}^{1}$ is the cotangent bundle of an
open subset of $\R^{14}.$ Therefore, it is convenient to use the
geometric quantization in terms of the vertical polarization tangent to
fibres of the cotangent bundle projection. This approach leads to
Schr\"odinger quantization. Since the functions we want to quantize are at
most quadratic in momenta, we can use results on Schr\"odinger quantization
as they are presented in texts on quantum mechanics. For technical details
see \cite{sniatycki80}.

Consider the space $C_{0}^{\infty }(J_{0}^{1})\otimes \C$ of
complex-valued compactly supported smooth functions $\Psi $ on $J_{0}^{1}$
endowed with the scalar product 
\begin{equation}
\left( \Psi _{1}\mid \Psi _{2}\right) =\int_{J_{0}^{1}}\overline{\Psi } 
_{1}(t,x,\dot{x})\Psi _{2}(t,x,\dot{x})\,dt\,\,d^{3}x\,\,d^{3}\dot{x}.
\label{scalar product 1}
\end{equation} 
The completion of $C_{0}^{\infty }(J_{0}^{1})\otimes \C$ with
respect to the norm given by this scalar product gives rise to the Hilbert
space $\mathfrak{H}_{0}$ of quantum states of the system.

If $f\in C^{\infty }(J_{0}^{1})$, then the operator $\boldsymbol{Q}_{\pi
^{\ast }f}$ corresponding to the pullback $\pi ^{\ast }f$ of $f$ by the
cotangent bundle projection $\pi :T^{\ast }J_{0}^{1}\rightarrow J_{0}^{1}$
acts on $C^{\infty }(J_{0}^{1})\otimes \C$ via multiplication by $f$ 
\begin{equation}
\boldsymbol{Q}_{\pi ^{\ast }f}\Psi =f\Psi .  \label{Qpi*f}
\end{equation} 
Operators corresponding to canonical momenta act by derivations with respect
to the corresponding variables in $J_{0}^{1}$. In particular, the linear
momentum $p_{x}$ gives rise to 
\begin{equation}
\boldsymbol{Q}_{p_{x}}=-i\hbar \frac{\partial }{\partial x},  \label{Qpx}
\end{equation} 
and the angular momentum $l=x\times p_{x}+\dot{x}\times p_{\dot{x}}$ leads
to 
\begin{equation}
\boldsymbol{Q}_{l}=-i\hbar x\times \frac{\partial }{\partial x}-i\hbar \dot{x 
}\times \frac{\partial }{\partial \dot{x}},  \label{Ql}
\end{equation} 
where $\hbar $ is Planck's constant. Furthermore, $\mathscr{J}_{\tau }=\tau
(t)p_{t}-\dot{\tau}(t)\left\langle p_{\dot{x}},\dot{x}\right\rangle \ $ 
quantizes to 
\begin{equation}
\boldsymbol{Q}_{\mathscr{J}_{\tau }}=-i\hbar \tau \frac{\partial }{\partial t 
}+i\hbar \dot{\tau}\left\langle \dot{x},\frac{\partial }{\partial \dot{x}} 
\right\rangle -i\hbar \dot{\tau}.  \label{QJtau}
\end{equation} 
Moreover, for $h=\frac{\left\vert \dot{x}\right\vert ^{2}}{4}\left\langle p_{ 
\dot{x}},p_{\dot{x}}\right\rangle +\frac{\left\langle p_{x},\dot{x} 
\right\rangle }{\left\vert \dot{x}\right\vert }$, the metric on $\R 
^{3}$ giving the quadratic term has scalar curvature 2, and the
corresponding operator is 
\begin{equation}
\boldsymbol{Q}_{h}=-\frac{\hbar ^{2}}{4}\left( \left\vert \dot{x}\right\vert
^{2}\dot{\Delta}-\frac{1}{3}\right) -\frac{i\hbar }{\left\vert \dot{x} 
\right\vert }\left\langle \dot{x},\frac{\partial }{\partial \dot{x}} 
\right\rangle ,  \label{Qh}
\end{equation}
where 
\begin{equation*}
\dot{\Delta}=\left\langle \frac{\partial }{\partial \dot{x}},\frac{\partial 
}{\partial \dot{x}}\right\rangle
\end{equation*} 
is the Laplace operator in the variables $\dot{x}$.\footnote{ 
Correction terms in equations (\ref{QJtau}) and (\ref{Qh}) are consequence
of using half-forms in order to ensure that the space $\mathfrak{H}_{0}$ of
quantum states can be described as the space of square-integrable complex
functions on $J_{0}^{1}$ (see \cite{sniatycki80}.)} These differential 
operators 
on 
$C^{\infty }(J_{0}^{1})\otimes \mathbb{C}$ extend to self-adjoint operators
on $\mathfrak{H}_{0}$.

\subsection{Quantization representation of $\SE(3)$}

The skew-adjoint operators $\frac{-i}{\hbar }\boldsymbol{Q}_{p_{x}}$ and $ 
\frac{-i}{\hbar }\boldsymbol{Q}_{l}$ generate a unitary representation of
the Euclidean group $\SE(3)$ on $\mathfrak{H}_{0}$ such that for $g\in
\SE(3)$ and $\Psi \in C_{0}^{\infty }(J_{0}^{1})\otimes \C$ 
\begin{equation*}
U_{g}\Psi (t,x,\dot{x})=\Phi _{g^{-1}}^{\ast }\Psi (t,x,\dot{x}),
\end{equation*} 
where 
\begin{equation*}
\Phi :\SE(3)\times J_{0}^{1}\rightarrow J_{0}^{1}
\end{equation*} 
is the action of $\SE(3)$ on $J_{0}^{1}.$ 
 The invariant 
vectors of this representation are eigenvectors of $\frac{-i}{ 
\hbar }\boldsymbol{Q}_{p_{x}}$ and $\frac{-i}{\hbar }\boldsymbol{Q}_{l}$
corresponding to the eigenvalue $0$. In other words, invariant vectors $\Psi 
$ of $U$ are characterized by the equations 
\begin{eqnarray*}
\boldsymbol{Q}_{p_{x}}\Psi &=&0, \\
\boldsymbol{Q}_{l}\Psi &=&0.
\end{eqnarray*} 
Since $\SE(3)$ is not compact, invariant vectors of $U$ are distributions on $ 
J_{0}^{1}$.

\subsection{Quantization representation of $\Diff_{+}\R$}

The reparametrization group $\Diff_{+}\R$ acts on 
$C^{\infty}(J_0^1)\otimes \C$ by the pullback of its action on $J_{0}^{1}$ 
\begin{equation*}
\Diff_+\R\times (C_{0}^{\infty }(J_0^1)\otimes \C)\rightarrow C^{\infty 
}(J_{0}^{1})\otimes \C:(\varphi ,\Psi
)\mapsto (\varphi ^{-1})^{1\ast }\Psi ,
\end{equation*} 
where $\varphi ^{-1}$ is the inverse of $\varphi $, and 
\begin{equation*}
(\varphi ^{-1})^{1\ast }\Psi (t,x,\dot{x})=\Psi (\varphi ^{-1}(t,x,\dot{x} 
))=\Psi \left( \varphi ^{-1}(t),x,\frac{\dot{x}}{(\varphi ^{-1}\dot{)}(t)} 
\right) .
\end{equation*} 
For an infinitesimal diffeomorphism $\varphi_{\epsilon}(t)=t+\epsilon \tau
(t)+\dots$ generated by $\tau (t)\partial_t$,

\begin{equation*}
(\varphi _{\epsilon }^{-1})^{1\ast }\Psi (t,x,\dot{x})=\Psi (t,x,\dot{x} 
)-\epsilon \left( \tau \frac{\partial }{\partial t}-\dot{x}\dot{\tau}\frac{ 
\partial f}{\partial \dot{x}}\right) \Psi +...
\end{equation*}
and 
\begin{eqnarray}
\frac{d}{d\epsilon }(\varphi _{\epsilon }^{-1})^{1\ast }\Psi (t,x,\dot{x} 
)|_{\epsilon =0} &=&\left( \tau \frac{\partial }{\partial t}-\dot{x}\dot{ 
\tau}\frac{\partial f}{\partial \dot{x}}\right) \Psi (t,x,\dot{x})
\label{dpull-back} \\
&=&\left( \frac{-i}{\hbar }\boldsymbol{Q}_{\mathscr{J}_{\tau }}\Psi \right)
(t,x,\dot{x})-\dot{\tau}\Psi (t,x,\dot{x}).  \notag
\end{eqnarray} 
Therefore, 
\begin{eqnarray*}
\left( \frac{-i}{\hbar }\boldsymbol{Q}_{\mathscr{J}_{\tau }}\Psi \right)
(t,x,\dot{x}) &=&\frac{d}{d\epsilon }(\varphi _{\epsilon }^{-1})^{1\ast
}\Psi (t,x,\dot{x})|_{\epsilon =0}+\dot{\tau}\Psi (t,x,\dot{x}) \\
&=&\frac{d}{d\epsilon }\{(\varphi _{\epsilon }^{-1})^{1\ast }\Psi (t,x,\dot{x 
})+\dot{\varphi}_{\epsilon }(t)\Psi (t,x,\dot{x}))\}|_{\epsilon =0} \\
&=&\frac{d}{d\epsilon }\{(\varphi _{\epsilon }^{-1})^{1\ast }[\dot{\varphi} 
_{\epsilon }(t)\Psi (t,x,\dot{x})]\}|_{\epsilon =0}.
\end{eqnarray*} 
This establishes
\begin{proposition}
\label{phi epsilon}The operator 
\begin{equation*}
\frac{-i}{\hbar }\boldsymbol{Q}_{\mathscr{J}_{\tau }}=\tau \frac{\partial }{ 
\partial t}-i\hbar \dot{\tau}\left\langle \dot{x},\frac{\partial }{\partial 
\dot{x}}\right\rangle +\dot{\tau}
\end{equation*}
generates the action on $C^{\infty }(J_{0}^{1})\otimes \C$ of the
one-parameter group $\varphi _{\epsilon }=\exp \epsilon \tau $ given by 
\begin{equation*}
\Xi :(\varphi _{\epsilon },\Psi )\mapsto \Xi _{\varphi _{\epsilon }}\Psi ,
\end{equation*} 
where 
\begin{equation*}
\Xi _{\varphi _{\epsilon }}\Psi (t,x,\dot{x})=(\varphi _{\epsilon
}^{-1})^{1\ast }[\dot{\varphi}_{\epsilon }\Psi (t,x,\dot{x})].
\end{equation*}
\end{proposition}

\begin{theorem}
The map 
\begin{equation*}
\Xi :\Diff_{+}\R\times \left( C^{\infty }(J_{0}^{1})\otimes 
\C\right) \rightarrow C^{\infty }(J_{0}^{1})\otimes \C
:(\varphi ,\Psi )\mapsto \Xi _{\varphi }\Psi ,
\end{equation*} 
where 
\begin{equation}
\Xi _{\varphi }\Psi (t,x,\dot{x})=(\varphi ^{-1})^{1\ast }[\dot{\varphi} 
(t)\Psi (t,x,\dot{x})],  \label{Ksi phi}
\end{equation} 
is a linear representation of $\Diff_{+}\R$ on $C^{\infty
}(J_{0}^{1})\otimes \C$ preserving the scalar product given by
equation {\rm(\ref{scalar product 1})}. 
\end{theorem}

\begin{proof}
Clearly, $\Xi _{\varphi }$ acts linearly on $C^{\infty }(J_{0}^{1})\otimes 
\C$. For $\varphi _{1},\varphi _{2}\in \Diff_{+}\R$
and $\Psi \in C^{\infty }(J_{0}^{1})\otimes \C$, 
\begin{eqnarray*}
(\Xi _{\varphi _{2}}\Xi _{\varphi _{1}}\Psi ) &=&\Xi _{\varphi
_{2}}\{(\varphi _{1}^{-1})^{1\ast }[\dot{\varphi}_{1}\Psi ]\} \\
&=&(\varphi _{2}^{-1})^{1\ast }\{\dot{\varphi}_{2}(t)(\varphi
_{1}^{-1})^{1\ast }[\dot{\varphi}_{1}\Psi ]\} \\
&=&(\varphi _{2}^{-1})^{\ast }\dot{\varphi}_{2}(\varphi _{2}^{-1})^{1\ast
}\{(\varphi _{1}^{-1})^{1\ast }\dot{\varphi}_{1}[(\varphi _{1}^{-1})^{1\ast
}\Psi ]\} \\
&=&(\varphi _{2}^{-1})^{\ast }\dot{\varphi}_{2}(\varphi _{2}^{-1})^{\ast
}(\varphi _{1}^{-1})^{\ast }\dot{\varphi}_{1}[(\varphi _{2}^{-1})^{1\ast
}(\varphi _{1}^{-1})^{1\ast }\Psi ] \\
&=&(\varphi _{2}^{-1})^{\ast }[\dot{\varphi}_{2}(\varphi _{1}^{-1})^{\ast } 
\dot{\varphi}_{1}][((\varphi _{2}\circ \varphi _{1})^{-1})^{\ast }\Psi ]
\end{eqnarray*} 
But 
\begin{eqnarray*}
\lbrack (\varphi _{2}\circ \varphi _{1})^{-1}]^{\ast }(\varphi _{2}\circ
\varphi _{1}\dot{)}(t) &=&(\varphi _{2}\circ \varphi _{1}\dot{)}((\varphi
_{2}\circ \varphi _{1})^{-1}(t)) \\
&=&(\varphi _{2}\circ \varphi _{1}\dot{)}((\varphi _{1}^{-1}\circ \varphi
_{2}^{-1})(t)) \\
&=&(\varphi _{2}\circ \varphi _{1}\dot{)}((\varphi _{1}^{-1}(\varphi
_{2}^{-1}(t))) \\
&=&\dot{\varphi}_{2}(\varphi _{1}(\varphi _{1}^{-1}(\varphi _{2}^{-1}(t)))) 
\dot{\varphi}_{1}(\varphi _{1}^{-1}(\varphi _{2}^{-1}(t))) \\
&=&\dot{\varphi}_{2}(\varphi _{2}^{-1}(t))\dot{\varphi}_{1}(\varphi
_{1}^{-1}(\varphi _{2}^{-1}(t))) \\
&=&(\varphi _{2}^{-1})^{\ast }[\dot{\varphi}_{2}(\varphi _{1}^{-1})^{\ast } 
\dot{\varphi}_{1}],
\end{eqnarray*} 
as required.

For $\Psi _{1},\Psi _{2}\in C^{\infty }(J_{0}^{1})\otimes \C$ and $ 
\varphi \in \Diff_+\R$, 
\begin{eqnarray*}
(\Xi _{\varphi }\Psi _{1}\mid \Xi _{\varphi }\Psi _{2}) & = &\int_{J_{0}^{1}} 
\overline{\Xi _{\varphi }\Psi _{1}(t,x,\dot{x})}\,\Xi _{\varphi }\Psi _{2}(t,x, 
\dot{x})\,dt\,d^3x\,d^3\dot{x} \\
&=&\int_{J_{0}^{1}}\overline{(\varphi ^{-1})^{1\ast }\{\dot{\varphi}(t)\Psi
_{1}(t,x,\dot{x})\}}\,(\varphi ^{-1})^{1\ast }\{\dot{\varphi}(t)\Psi _{2}(t,x, 
\dot{x})\}dt\,d^3x\,d^3\dot{x} \\
&=&\int_{J_{0}^{1}}[\dot{\varphi}(\varphi ^{-1}(t))]^{2}\, \overline{\Psi
_{1}(\varphi ^{-1}(t),x,\dot{x}/(\varphi ^{-1}\dot{)}(t))}\,\{\Psi _{1}(\varphi
^{-1}(t),x,\dot{x}/(\varphi ^{-1}\dot{)}(t)\}\,dt\,d^3x\,d^3\dot{x}.
\end{eqnarray*} 
Note that the inverse function theorem guarantees 
\begin{equation*}
\dot{\varphi}(\varphi ^{-1}(t))=\frac{1}{(\varphi ^{-1}\dot{)}(t)}.
\end{equation*} 
Introducing new variables 
\begin{equation*}
\bar{t}=\varphi ^{-1}(t),~\bar{x}=x\text{ and }\bar{x}^{\prime }=\dot{x} 
/(\varphi ^{-1}\dot{)}(t),
\end{equation*} 
yields
\begin{eqnarray*}
d\bar{t} &=&(\varphi ^{-1}\dot{)}(t)\,dt, \\
d\bar{x} &=&dx, \\
d\bar{x}^{\prime } &=&\frac{1}{(\varphi ^{-1}\dot{)}(t)}\,d\dot{x}-\frac{ 
(\varphi ^{-1}\ddot{)}(t)\dot{x}}{[(\varphi ^{-1}\dot{)}(t)]^{2}}\,dt, \\
d\bar{t}\,d^3\bar{x}\,d^3\bar{x}^{\prime } &=&\frac{1}{\left( (\varphi 
^{-1}\dot{)}(t)\right) ^{2}}\,dt\,d^3x\,d^3\dot{x},
\end{eqnarray*} 
so that 
\begin{equation*}
dt\,d^3x\,d^3\dot{x}=\left[ (\varphi ^{-1}\dot{)}(t)\right] ^{2}d\bar{t}\,d^3 
\bar{x}\,d^3\bar{x}^{\prime }.
\end{equation*} 
Therefore, 
\begin{eqnarray*}
(\Xi _{\varphi }\Psi _{1} \mid \Xi _{\varphi }\Psi _{2}) & = & 
\int_{J_{0}^{1}}[ 
\dot{\varphi}(\varphi ^{-1}(t))]^{2}\overline{\Psi _{1}(\bar{t},\bar{x},\bar{ 
x}^{\prime })}\Psi _{2}(\bar{t},\bar{x},\bar{x}^{\prime })\,dt\,d^3x\,d^3\dot{x 
} \\
&=&\int_{J_{0}^{1}}\frac{1}{[(\varphi ^{-1}\dot{)}(t)]^{2}}\overline{\Psi
_{1}(\bar{t},\bar{x},\bar{x}^{\prime })}\Psi _{2}(\bar{t},\bar{x},\bar{x} 
^{\prime })\left[ (\varphi ^{-1}\dot{)}(t)\right] ^{2}d\bar{t}\,d^3\bar{x} 
\,d^3\bar{x} \\
&=&\int_{J_{0}^{1}}\overline{\Psi _{1}(\bar{t},\bar{x},\bar{x}^{\prime })} 
\Psi _{2}(\bar{t},\bar{x},\bar{x}^{\prime })\,d\bar{t}d^3\bar{x}\,d^3\bar{x}.
\end{eqnarray*}
\end{proof}

Note that, for $\Psi,\Psi^{\prime }\in C^{\infty }_0(J_{0}^{1})\otimes \C$, the
integral defining the scalar product (\ref{scalar product 1}), and 
\begin{equation}
(\Psi ^{\prime }\mid \Psi )=\int_{J_{0}^{1}}\overline{\Psi ^{\prime }(t,x, 
\dot{x})}\Psi (t,x,\dot{x})\,dt\,d^3x\,d^3\dot{x}  \label{scalar product 2}
\end{equation} 
can be interpreted as the evaluation on $\Psi $ of the generalized function
(distribution) $\Psi ^{\prime }\in (C_{0}^{\infty }(J_{0}^{1})\otimes 
\C)^{\prime }$. The representation $\Xi $ of $\Diff_+ 
\R$ on $C_{0}^{\infty }(J_{0}^{1})\otimes \C$ extends to a
representation of $\Diff_+\R$ on $(C^{\infty
}(J_{0}^{1})\otimes \C)^{\prime }$, which we also denote by $\Xi $,
such 
\begin{equation}
(\Xi _{\varphi }\Psi ^{\prime }\mid \Psi )=(\Psi ^{\prime }\mid \Xi
_{\varphi ^{-1}}\Psi )  \label{Ksi distribution}
\end{equation}
for $\varphi \in \Diff_+\R\,$, $\Psi ^{\prime }\in
(C_{0}^{\infty }(J_{0}^{1})\otimes \C)^{\prime }$. Note that, if $ 
\Psi ^{\prime }\in C^{\infty }(J_{0}^{1})\otimes \C$, then the
definition of the action $\Xi _{\varphi }$ on $\Psi ^{\prime }$, given here,
coincides with the definition given in equation (\ref{Ksi phi}).

It remains to examine the space of $\Diff_{+}\R$-invariant
functions. Since the group $\Diff_{+}\R$ is not compact, the
only compactly supported $\Diff_{+}\R$-invariant function in 
$C_{0}^{\infty }(J_{0}^{1})\otimes \C$ is identically zero. Hence,
 $\Diff_{+}\R$-invariant functions are in $C^{\infty 
}(J_{0}^{1})\otimes\C$. 

\begin{lemma}
For $\boldsymbol{q}=(t_{0},x_{0},\dot{x}_{0})\in J_{0}^{1}$, the orbit $ 
\exp (\diff_+\R)(\boldsymbol{q})$ of $\diff_+\R$ through $\boldsymbol{q}$ 
coincides with the orbit $\Diff_+\R(\boldsymbol{q})$ of $\Diff_+\R$ 
through $ \boldsymbol{q}$ 
\begin{equation*}
\exp (\diff_+\R)(\boldsymbol{q})=\Diff_+\R(\boldsymbol{q}).
\end{equation*}
\end{lemma}

\begin{proof}
For $X=\tau \partial_t\in \diff_+\R$, the
integral curves of $X^{1}=\tau \frac{\partial }{\partial t}-\dot{\tau}\dot{x} 
\frac{\partial }{\partial \dot{x}}$ satisfy the differential equations 
\begin{equation*}
\frac{dt}{ds}=\tau ,\quad \frac{dx}{ds}=0, \quad\text{and}\quad 
\frac{d\dot{x}}{ds}=- 
\dot{\tau}\dot{x}.
\end{equation*} 
Hence, 
\begin{equation*}
\frac{dt}{\tau }=ds
\end{equation*} 
and 
\begin{equation*}
\int_{t_{0}}^{t}\frac{dt^{\prime }}{\tau (t^{\prime })}=s.
\end{equation*} 
Choosing $\tau (t)=e^{-t}$, 
\begin{equation*}
\int_{t_{0}}^{t}\frac{dt^{\prime }}{\tau (t^{\prime })} 
=\int_{t_{0}}^{t}e^{t^{\prime }}dt^{\prime }=e^{t}-e^{t_{0}}.
\end{equation*} 
Hence, $e^{t}-e^{t_{0}}=s$, which implies that 
\begin{equation*}
t=\log \left\vert s+e^{-t_{0}}\right\vert .
\end{equation*} 
But the range of the logarithm is $(-\infty ,\infty )$. Therefore,
the range of values of $t$ on the orbit of $\diff_+\R$
through $\boldsymbol{q}$ is $(-\infty ,\infty ).$

Moreover, for $i=1,2,3,$ 
\begin{equation*}
\frac{d\dot{x}_{i}}{ds}=-\dot{\tau}\dot{x}_{i}
\end{equation*} 
implies 
\begin{equation*}
\frac{d\dot{x}_{i}}{\dot{x}_{i}}=-\dot{\tau}(t(s))\,ds
\end{equation*} 
so that 
\begin{equation*}
\dot{x}(s)=\dot{x}_{0}\exp \left( -\int_{0}^{s}\dot{\tau}(t(s))\,ds\right) .
\end{equation*} 
Since $\tau (t)$ is an arbitrary function of $t$, it follows that the orbit
of $\diff_+\R$ through $q=(t_{0},x_{0},\dot{x}_{0})$ is 
\begin{equation*}
\exp (\diff_+\R)(q)=\{(u,x_{0},e^{v}\dot{x})\mid (u,v)\in 
\R^{2}\}.
\end{equation*}

The action of $\Diff_+\R$ on $J_{0}^{1}$
is 
\begin{equation*}
\Diff_+\R\times J_{0}^{1}\rightarrow J_{0}^{1}:(\varphi
,(t,x,\dot{x}))\mapsto \left( \varphi (t),x,\frac{\dot{x}}{\dot{\varphi}(t)} 
\right) ,
\end{equation*} 
where $\dot{\varphi}(t)>0$. Since, $\varphi (t)$ and $\dot{\varphi}(t)$ are
independent, it follows that the orbit of $\Diff_+\R$
through $q$ is 
\begin{equation*}
\Diff_+\R(q)=\{(u,x_{0},w\dot{x})\mid (u,w)\in \R 
^{2},~w>0\}.
\end{equation*} 
Hence, $\exp (\diff_+\R)(q)=\Diff_+\R(q)$.
\end{proof}

\begin{theorem}
A function $\Psi ^{\prime }\in C^{\infty }(J_{0}^{1})\otimes \C$ is 
$\Diff_+\R$-invariant if and only if 
\begin{equation*}
\boldsymbol{Q}_{J_{\tau }}\Psi ^{\prime }=0
\end{equation*} 
for all $\tau \in \diff_{+}\R$.
\end{theorem}

\begin{proof}
If $\Psi ^{\prime }\in C^{\infty }(J_{0}^{1})\otimes \C$ is 
$\Diff_+\R$-invariant, then it is invariant under the action of
every one-parameter subgroup of $\Diff_+\R$. By
proposition \ref{phi epsilon}, actions of one-parameter subgroups of $ 
\Diff_+\R$ on $C^{\infty }(J_{0}^{1})\otimes \C$
are generated by $\frac{-i}{\hbar }\boldsymbol{Q}_{\mathscr{J}_{\tau }}$ for 
$\tau \in \diff_+\R$. Hence, $\boldsymbol{Q}_{J_{\tau
}}\Psi ^{\prime }=0$ for all $\tau \in \diff_+\R$.

Conversely, suppose that $\Psi ^{\prime }$ is a function in $C^{\infty
}(J_{0}^{1})\otimes \C$ such that $\boldsymbol{Q}_{J_{\tau }}\Psi
^{\prime }=0$ for all $\tau \in \diff_{+}\R$. Hence, $\Xi
_{\varphi _{\epsilon }}\Psi ^{\prime }=\Psi ^{\prime }$ for every
one-parameter subgroup $\varphi _{\epsilon }$ of $\Diff_{+}\R
$. Recall that, for $\varphi \in \diff_+\R$, 
\begin{eqnarray*}
\Xi _{\varphi }\Psi ^{\prime }(t,x,\dot{x}) &=&(\varphi ^{-1})^{1\ast }[\dot{ 
\varphi}(t)\Psi ^{\prime }(t,x,\dot{x})]=[(\varphi ^{-1})^{\ast }\dot{\varphi 
}(t)][(\varphi ^{-1})^{1\ast }\Psi ^{\prime }(t,x,\dot{x})] \\
&=&\dot{\varphi}(\varphi ^{-1}(t))\Psi ^{\prime }(\varphi ^{-1}(t,x,\dot{x} 
))=\frac{1}{\frac{d\varphi ^{-1}(t)}{dt}}\Psi ^{\prime }(\varphi ^{-1}(t,x, 
\dot{x})),
\end{eqnarray*} 
because 
\begin{equation*}
\varphi \circ \varphi ^{-1}=\mathrm{identity}
\end{equation*} 
implies 
\begin{equation*}
\dot{\varphi}(\varphi ^{-1}(t))\frac{d\varphi ^{-1}(t)}{dt}=1.
\end{equation*} 
Therefore, 
\begin{equation*}
\Xi _{\varphi ^{-1}}\Psi ^{\prime }(t,x,\dot{x})=\frac{1}{\dot{\varphi}(t)} 
\Psi ^{\prime }(\varphi (t,x,\dot{x})).
\end{equation*}

It follows from the lemma above that there exists a finite sequence of
one-parameter subgroups $\varphi_{\epsilon _{1}},\dots,\varphi _{\epsilon
_{k}}$ such that 
\begin{equation*}
\varphi (t,x,\dot{x})=\varphi _{\epsilon _{k}}(\dots(\varphi _{\epsilon
_{1}}(t,x,\dot{x}))\dots).
\end{equation*} 
Since 
\begin{equation*}
\varphi (t,x,\dot{x})=\left( \varphi (t),x,\frac{\dot{x}}{\dot{\varphi}(t)} 
\right),
\end{equation*} 
\begin{equation*}
\varphi (t)=\varphi _{\epsilon _{k}}(\dots(\varphi _{\epsilon
_{1}}(t))\dots)=\varphi _{\epsilon _{k}}\circ \cdots\circ \varphi _{\epsilon 
_{1}}(t)
\end{equation*} 
and 
\begin{equation*}
\dot{\varphi}(t)=\dot{\varphi}_{\epsilon _{k}}(\dots(\varphi _{\epsilon
_{1}}(t))\dots)\dots\dot{\varphi}_{\epsilon _{1}}(t)=\frac{d}{dt}\varphi 
_{\epsilon
_{k}}\circ \cdots\circ \varphi _{\epsilon _{1}}(t).
\end{equation*} 
Hence, 
\begin{eqnarray*}
\Xi _{\varphi ^{-1}}\Psi ^{\prime }(t,x,\dot{x}) &=&\frac{1}{\dot{\varphi}(t) 
}\Psi ^{\prime }(\varphi (t,x,\dot{x}))=\frac{1}{\dot{\varphi}(t)}\Psi
^{\prime }\left( \varphi (t),x,\frac{\dot{x}}{\dot{\varphi}(t)}\right) \\
&=&\frac{1}{\frac{d}{dt}\varphi _{\epsilon _{k}}\circ \cdots\circ \varphi
_{\epsilon _{1}}(t)}\Psi ^{\prime }\left( \varphi _{\epsilon _{k}}\circ
\cdots\circ \varphi _{\epsilon _{1}}(t),x,\frac{\dot{x}}{\frac{d}{dt}\varphi
_{\epsilon _{k}}\circ \cdots\circ \varphi _{\epsilon _{1}}(t)}\right) \\
&=&\Xi _{(\varphi _{\epsilon _{k}}\circ \cdots\circ \varphi _{\epsilon
_{1}})^{-1}}\Psi ^{\prime }(t,x,\dot{x})=\Xi _{\varphi _{\epsilon
_{1}}^{-1}\circ \cdots\circ \varphi _{\epsilon _{k}}^{-1}}\Psi ^{\prime }(t,x, 
\dot{x})=\Xi _{\varphi _{-\epsilon _{1}}\circ \cdots\circ \varphi _{-\epsilon
_{k}}}\Psi ^{\prime }(t,x,\dot{x}) \\
&=&\Xi _{\varphi _{-\epsilon _{1}}}\cdots\Xi _{\varphi _{-\epsilon _{k}}}\Psi
^{\prime }(t,x,\dot{x})=\Psi ^{\prime }(t,x,\dot{x}),
\end{eqnarray*} 
because $\Psi ^{\prime }$ is invariant under the action of one-parameter
subgroups of $\Diff_+\R$.
\end{proof}

\subsection{Quantum implementation of constraints}

In the Bleuler-Gupta approach, the extended phase space is quantized first.
This associates to each constraint function the corresponding quantum
operator. The next step is to implement the constraint conditions on the
quantum level. This is done by placing a restriction on the states of the 
system.

\begin{definition}
{\it Admissible quantum states} are eigenstates of the quantum operators
associated to the constraint functions corresponding to the joint eigenvalue
zero.
\end{definition}

For elastica, the classical constraints are  
\begin{eqnarray*}
p_{t} &=&0, \\
\left\langle p_{\dot{x}},\dot{x}\right\rangle &=&0, \\
\left\vert \dot{x}\right\vert ^{3}\left\langle p_{\dot{x}},p_{\dot{x} 
}\right\rangle +4\left\langle p_{x},\dot{x}\right\rangle &=&0.
\end{eqnarray*} 
However, linear combinations of these functions with smooth coefficients
leads to further functions that vanish on the range of Legendre
transformation. We might not be able to quantize these functions in our
chosen quantization scheme. This is why it is important to have criteria that
help select a convenient basis of the ideal of functions that
vanish on the range of Legendre transformation.

The first two constraint conditions $p_{t}=0$ and $ 
\left\langle p_{\dot{x}},\dot{x}\right\rangle =0$ are equivalent to
vanishing of the momenta $\mathscr{J}_{\tau }$ for the action of one-parameter
subgroups of $\Diff_+\R$. Therefore, the quantum
implementation of these constraints is the requirement that the admissible wave
functions $\Psi $ should satisfy the conditions 
\begin{equation*}
\boldsymbol{Q}_{\mathscr{J}_{\tau }}\Psi =0\text{ for }\tau \in \diff_+\R.
\end{equation*} 
By the results of the preceding section, this is equivalent to requiring
that admissible states should be invariant under the action of the
quantization representation of $C^{\infty }(J_{0}^{1})\otimes \C$.

The third constraint function has no immediately clear geometric 
interpretation. We have
used the freedom of the choice of generators of the ideal of constraint
functions and replaced it by the reparametrization invariant function 
\begin{equation*}
h=\frac{\left\vert \dot{x}\right\vert ^{2}}{4}\left\langle p_{\dot{x}},p_{ 
\dot{x}}\right\rangle +\frac{\left\langle p_{x},\dot{x}\right\rangle }{ 
\left\vert \dot{x}\right\vert }.
\end{equation*} 
The quantum operator corresponding to $h$ is 
\begin{equation*}
\boldsymbol{Q}_{h}=-\frac{\hbar ^{2}}{4}\left( \left\vert \dot{x}\right\vert
^{2}\dot{\Delta}-\frac{1}{3}\right) -\frac{i\hbar }{\left\vert \dot{x} 
\right\vert }\left\langle \dot{x},\frac{\partial }{\partial \dot{x}} 
\right\rangle .
\end{equation*} 
Thus,  admissible states $\Psi $ of quantum elastica are also required to 
satisfy the equation $\boldsymbol{Q}_{h}\Psi =0$.

\begin{summary}
The admissible states of quantum elastica are given by functions $\Psi \in
C^{\infty }(J_{0}^{1})\otimes \C$ that are invariant under the
quantization representation of the reparametrization group $\Diff_+\R$ and 
satisfy the equation 
\begin{equation*}
-\frac{\hbar ^{2}}{4}\left( \left\vert \dot{x}\right\vert ^{2}\dot{\Delta}- 
\frac{1}{3}\right) \Psi -\frac{i\hbar }{\left\vert \dot{x}\right\vert } 
\left\langle \dot{x},\frac{\partial }{\partial \dot{x}}\right\rangle \Psi =0.
\end{equation*}
\end{summary}

As we have mentioned before, admissible functions of quantum elastica
are not square integrable on $J_{0}^{1}$. Therefore, in this
formulation of the theory, we need to introduce a new scalar product
on the space of admissible states using physical or geometric
criteria. For example, we could use $ \Diff_+\R$ invariance
of admissible quantum states and relate them to smooth functions on
the space $S$ of $\Diff_+\R $ -orbits in $J_{0}^{1}$, which
satisfy the differential equation obtained from the quantization of $h$
considered as a function on $T^{\ast }S$.

\bibliographystyle{plain}


\vspace{20pt}

\noindent Larry M. Bates \newline
Department of Mathematics \newline
University of Calgary \newline
Calgary, Alberta \newline
Canada T2N 1N4 \newline
bates@ucalgary.ca

\vspace{20pt}

\noindent Robin Chhabra \newline
Guidance, Navigation and Control Department \newline
MacDonald, Dettwiler and Associates Ltd. \newline
Brampton, Ontario \newline
Canada L6S 4J3 \newline
robin.chhabra@mdacorporation.com

\vspace{20pt}

\noindent J\k{e}drzej \'Sniatycki \newline
Department of Mathematics \newline
University of Calgary \newline
Calgary, Alberta \newline
Canada T2N 1N4 \newline
sniatyck@ucalgary.ca

\end{document}